\documentclass[10pt,a4paper,fleqn]{article}

\usepackage[paper=a4paper,left=25mm,right=25mm,top=25mm,bottom=25mm]{geometry}
\usepackage[utf8]{inputenc}
\usepackage[T1]{fontenc}
\usepackage[english]{babel}
\usepackage{lmodern}

\usepackage{amsmath,amssymb,amsfonts,mathtools}
\usepackage{bm}
\usepackage{stmaryrd}
\usepackage{mathrsfs}
\usepackage{dsfont}       
\usepackage{pifont}
\usepackage{cancel}
\usepackage{nicefrac}
\usepackage{xfrac}
\usepackage{units}

\usepackage{graphicx}
\usepackage{xcolor}
\usepackage{tikz}
\usetikzlibrary{arrows.meta,decorations.markings,matrix,fit}

\usepackage{booktabs}
\usepackage{multirow}
\usepackage{makecell}
\usepackage[labelfont={bf,sf},font=small,labelsep=space]{caption}

\usepackage[shortlabels]{enumitem}
\setlist{itemsep=0.2em, topsep=0.3em, parsep=0pt}

\usepackage{authblk}
\usepackage{verbatim}
\usepackage{soul}
\usepackage{fancyhdr}
\usepackage{acronym}
\usepackage{tcolorbox}
\usepackage[linesnumbered,ruled,vlined]{algorithm2e}
\usepackage{hologo}
\usepackage{textcomp}
\usepackage{lipsum}

\usepackage[numbers]{natbib}

\usepackage[colorlinks=false,linkcolor=darkblue]{hyperref}
\usepackage{cleveref}  

\tikzset{
  mleftdelimiter/.style={
    inner ysep=0pt, inner xsep=1ex,
    left delimiter=\{,
    label={[label distance=3mm]left:#1}
  }
}

\usepackage[Bjarne]{fncychap}

\usepackage{setspace}
\usepackage[squaren]{SIunits}
\usepackage{overpic}
\usepackage{subcaption}

\definecolor{light-gray}{gray}{0.95}
\definecolor{darkblue}{rgb}{0,0,.5}

\newcommand{\cmark}{\color{green}{\ding{51}}\color{black}}%
\newcommand{\xmark}{\color{red}{\ding{55}}\color{black}}%

\newcommand{\C}{\mathbb{C}}

\newcommand{\E}{\mathbb{E}}

\newcommand{\N}{\mathbb{N}}

\newcommand{\R}{\mathbb{R}}

\newcommand{\UU}{\bold{U}}
\newcommand{\VV}{\bold{V}}
\newcommand{\WW}{\bold{W}}
\newcommand{\XX}{\bold{X}}

\newcommand{\ZZ}{\bold{Z}}

\newcommand{\uu}{\bold{u}}

\newcommand{\xx}{\bold{x}}
\newcommand{\yy}{\bold{y}}
\newcommand{\zz}{\bold{z}}

\newcommand{\cC}{\mathcal{C}}

\newcommand{\cF}{\mathcal{F}}

\newcommand{\cL}{\mathcal{L}}
\newcommand{\cA}{\mathcal{A}}
\newcommand{\cB}{\mathcal{B}}

\newcommand{\cR}{\mathcal{R}}
\newcommand{\sR}{\mathsf{R}}

\newcommand{\cW}{\mathcal{W}}
\newcommand{\dW}{\overrightarrow{\mathcal{W}}}

\providecommand{\keywords}[1]{\textbf{Keywords } #1}
\newcommand{\eqd}{\stackrel{\mathrm{d}}=}

\newcommand{\1}{\mathds{1}}

\newcommand{\de}{\mathsf{\,d}}

\newcommand{\Ran}{\mathsf{Ran}}

\newcommand{\rank}{\mathsf{rank}}

\newcommand{\supp}{\mathop{\mathrm{supp}}}

\newcommand{\Cor}{\mathrm{Cor}}

\newcommand{\var}{\mathrm{Var}}
\newcommand{\Var}{\mathrm{Var}}

\DeclareMathOperator{\Cov}{Cov}

\DeclareMathOperator{\v@r}{V@R}

\newtheorem{theorem}{Theorem}[section]
\newtheorem{proposition}[theorem]{Proposition}
\newtheorem{corollary}[theorem]{Corollary}
\newtheorem{definition}[theorem]{Definition}
\newtheorem{lemma}[theorem]{Lemma}
\newtheorem{example}[theorem]{Example}
\newtheorem{remark}[theorem]{Remark}

\newenvironment{proof}[1][{Proof:}]{\begin{trivlist}
\item[\hskip \labelsep {\bfseries #1}]}{\hfill$\blacksquare$ \end{trivlist}}

\selectfont

\author[1]{Jonathan Ansari\thanks{Corresponding author: \texttt{jonathan.ansari@plus.ac.at}}}
\author[1]{Sebastian Fuchs \thanks{\texttt{sebastian.fuchs@plus.ac.at}}}

\affil[1]{University of Salzburg (Austria)}
\title{An ordering for the strength of functional dependence}

\begin{document}

\maketitle

\begin{abstract}
We introduce a new dependence order, termed the \emph{conditional convex order}, whose minimal and maximal elements characterize independence and perfect dependence. Moreover, it characterizes conditional independence, satisfies information monotonicity, and exhibits several invariance properties.
Consequently, it is an ordering for the strength of functional dependence of a random variable \(Y\) on a random vector \(\XX\,.\) 
As we show, various recently studied dependence measures---including Chatterjee’s rank correlation, Wasserstein correlations, and rearranged dependence measures---are increasing in this order and inherit their fundamental properties from it.
We characterize the conditional convex order by the Schur order and by the concordance order, and we verify it in settings such as additive error models, the multivariate normal distribution, 
and various copula-based models.
Our results offer a unified perspective on the behavior of dependence measures across statistical models.

\end{abstract}

\maketitle

\section{Introduction and main results}\label{secintro}

A fundamental problem in regression analysis is to predict a response \(Y\) given a random vector \(\XX\) of input variables. 
To construct parsimonious models, it is important to identify the relevant predictors.
Recent literature has focused on model-free variable selection methods based on novel dependence measures \(\kappa(Y,\XX)\) with the following remarkable properties \cite{Azadkia-2025,chatterjee2023,deb2020b,strothmann2022,wiesel2022}: 
\(\kappa(Y,\XX)\) takes values in \([0,1]\), where \(0\) characterizes independence of \(Y\) and \(\XX\) (\emph{zero-independence}), and \(1\) characterizes perfect dependence of \(Y\) on \(\XX\) (\emph{max-functionality}), i.e., there exists a measurable function \(f\) such that \(Y = f(\XX)\) almost surely. 
Moreover, \(\kappa\) satisfies information monotonicity, that is, \(\kappa(Y,\XX)\leq \kappa(Y,(\XX,\ZZ))\), and it characterizes conditional independence in the sense that \(\kappa(Y,\XX) = \kappa(Y,(\XX,\ZZ))\) if and only if \(Y\) and \(\ZZ\) are conditionally independent given \(\XX\).
A prominent example is Chatterjee's rank correlation \(\xi\) defined in Equation \eqref{Tuniv} below.
While the values \(0\) and \(1\) have a clear interpretation, the intermediate values remain elusive.
An open question concerns the identification of a dependence order that underlies the above type of dependence measures and ranks their values across various models.
In this paper, we fill this gap by introducing a new dependence order that allows us to answer the following question in the affirmative.

Is there a dependence order that
\begin{itemize}
    \item characterizes independence through its minimal elements and perfect (functional) dependence through its maximal elements,
    \item satisfies information monotonicity and characterizes conditional independence,
    \item underlies Chatterjee's rank correlation and related dependence measures,
    \item has a simple and interpretable form,
    \item can be verified in a variety of models, and
    \item applies to a large class of distributions? \pagebreak
\end{itemize}

Classical rank correlations such as Kendall's \(\tau\), Spearman's \(\varrho\), or Gini's \(\gamma\) quantify the degree of \emph{positive} or \emph{negative} dependence among random variables. 
There is a large literature on \emph{positive} dependence orderings underlying these rank correlations; see \cite{Bauerle-1997,Fang-Joe-1992,Joe-1990,Kimeldorf-1978,Kimeldorf-1987,Kimeldorf-1989,Scarsini-1996,Schriever-1987} and the references therein. 
The certainly most common \emph{positive} dependence order is the concordance order \cite{Tchen-1980,Yanagimoto-1969}, defined by a pointwise comparison of distribution and survival functions; see \eqref{defconcordance}. Its minimal and maximal elements correspond to countermonotonocity and comonotonocity, i.e., to
perfect negative and positive monotone dependence, respectively. 


In contrast, measures of regression dependence like \(\xi\) behave fundamentally differently, since they range from independence to perfect (not necessarily monotone) dependence.
Apart from our preliminary work \cite{Ansari-Rockel-2023,Ansari-LFT-2023},
research on \emph{global} dependence orderings is conducted in
\cite{Dabrowska-1981,siburg2013,Joe-1987,Scarsini-1990,Shaked-2012,Shaked-2013}. 
However, none of the dependence orderings in these references is able to describe, even in the simplest models, the behavior of dependence measures that characterize both independence \emph{and} perfect dependence.
For example, consider the additive error model
\begin{align}\label{eqadderrmod}
    Y = f(\XX) + \sigma \varepsilon,
\end{align}
where \(\sigma\geq 0\) is a deterministic parameter scaling the noise \(\varepsilon\) that is assumed to be independent of \(f(\XX)\,.\)
Since Chatterjee's \(\xi\) quantifies the strength of functional dependence \cite{chatterjee2023}, it should be decreasing in the noise parameter \(\sigma\).
Analogous monotonicity properties are expected to hold for related dependence measures such as Wasserstein correlations \cite{wiesel2022}, rearranged dependence measures \cite{strothmann2022}, kernel partial correlations \cite{deb2020b}, and measures of sensitivity \cite{Ansari-LFT-2023}. However, to the best of our knowledge, general monotonicity properties for such dependence measures have not been studied in the literature so far.


In this paper, we resolve this issue and introduce a new global dependence order---the \emph{conditional convex order}---that describes the strength of functional dependence of a random variable \(Y\) on a random vector \(\XX\). 
This order reflects the fundamental properties of a large class of dependence measures, and we verify it in settings like the additive error model \eqref{eqadderrmod}, the multivariate normal distribution, and various copula-based models. 
Before we propose eight natural axioms for a global dependence ordering underlying \(\xi\), recall the population version of Chatterjee's rank correlation defined by
\begin{align} \label{Tuniv}
  \xi(Y,\XX)
	:= \frac{\int_{\mathbb{R}} \Var (P(Y \geq y \, | \, \XX)) \; \mathrm{d} P^{Y}(y)}
					{\int_{\mathbb{R}} \Var (\mathds{1}_{\{Y \geq y\}}) \; \mathrm{d} P^{Y}(y)};
\end{align}
see \cite{chatterjee2021,chatterjee2020}.
It relies on the copula-based version in \cite{siburg2013} and on the sensitivity measures in \cite{Borgonovo-2021,Gamboa-Klein-Lagnoux-2018}. 
Dependence measures like $\xi$ capture the variability of conditional distributions in various ways and have attracted much attention in the past few years; see e.g.
\cite{auddy2021,Azadkia-2025,Figalli-2025,buecher2024,bickel2020,deb2020,bernoulli2021,deb2020b,han2022limit,shi2021normal,hörmann2025,strothmann2022,wiesel2022}.

For understanding the behavior of dependence measures in various models, a suitable dependence order is essential.
It is straightforward to see that \(\xi\) is not increasing in the concordance order. 
For example, let \(X\) and \(Y\) be standard normal with \(Y = \pm X.\) Then \((Y,X)\) is comonotone/countermonotone and thus maximal/minimal in the concordance order, but \(\xi(Y,X) = 1\) in both cases since \(Y\) is a deterministic function of \(X\,.\)

\subsection{Axioms for a global dependence ordering}

We propose the following axioms for a \emph{global dependence ordering} 
\(\prec\) 
that compares the strength of functional dependence of \(Y\) on \(\XX\) on a class \(\cR\) of random vectors defined on a common probability space \((\Omega,\cA,P)\). 
We denote by \(\eqd\) equality in distribution and write \(\equiv\) if \(\prec\) and \(\succ\) holds true.
\begin{enumerate}[(O1)]
\item \emph{Law-invariance}: \label{axdepord0} \((Y,\XX)\eqd (Y',\XX')\) implies \((Y,\XX) \equiv (Y',\XX')\).
\item \emph{Quasi-ordering}: \label{axdepord1}\(\prec\) is reflexive and transitive.
    \item \label{axdepord2} \emph{Characterization of independence}: \((Y,\XX)\prec (Y',\XX')\) for all \((Y',\XX')\)
    if and only if \(\XX\) and \(Y\) are independent.
    \item \label{axdepord3} \emph{Characterization of perfect dependence}: \((Y',\XX')\prec (Y,\XX)\) for all \((Y',\XX')\)
    if and only if \(Y\) is perfectly dependent on \(\XX\,.\)
    \item \label{axdepord4} \emph{Information monotonicity}: \((Y,\XX)\prec (Y,(\XX,\ZZ))\) for all \(\XX,\ZZ\) and \(Y\).
    \item \label{axdepord5} \emph{Characterization of conditional independence}: \((Y,\XX) \equiv (Y,(\XX,\ZZ))\) if and only if \(Y\) and \(\ZZ\) are conditionally independent given \(\XX\,.\)
    \item \label{axdepord6} \emph{Transformation invariance:} \((Y,\XX) \equiv (g(Y),{\bf h}(\XX))\) for every strictly increasing function \(g\) and for every bijective function \({\bf h}\,.\) 
    \item \label{axdepord7} \emph{Distributional invariance: 
    } 
    \((Y,\XX) \equiv (F_Y(Y),\XX)\) and \((Y,\XX) \equiv (Y,{\bf F}_\XX(\XX))\), where \({\bf F}_\XX = (F_{X_1},F_{X_2},\ldots)\) for \(\XX = (X_1,X_2,\ldots)\).
\end{enumerate}

\begin{remark}\label{reax}
    \begin{enumerate}[(a)]
    \item While \(Y\) and \(Y'\) are univariate random variables, \(\XX\) and \(\XX'\) are random vectors that may have arbitrary dimensions \(p,p'\in \N\), which are not assumed to be equal. 
    \item Axiom \ref{axdepord0} states that \(\prec\) is invariant under all versions of \((Y,\XX)\). Note that the term law-invariance is also used in the context of risk measures \cite{Foellmer-Schied-2016}.
    \item Axiom \ref{axdepord1} on reflexivity and transitivity is also considered in \cite{Dabrowska-1981,siburg2013,Shaked-2012}. In contrast to positive dependence orderings, a suitable global dependence ordering \(\prec\) cannot be expected to be antisymmetric (in the sense that \((Y,\XX)\equiv (Y',\XX')\) implies \((Y,\XX)\eqd(Y',\XX')\)) because it is desirable that \((Y,Y)\equiv (Y,-Y)\) due to Axiom \ref{axdepord3}.
        \item Axioms \ref{axdepord2} and \ref{axdepord3} 
        state that independent and perfectly dependent random vectors are minimal and maximal elements over the whole class \(\cR\), i.e., they are globally extreme elements. These axioms are significantly stronger than 
        \begin{itemize}
            \item \cite[Axioms O4 and O5]{Shaked-2012}, according to which any bivariate random vector \((Y,X)\) is comparable with an independent version \((Y^\perp,X^\perp)\) and with a perfectly dependent version\footnote{For a bivariate random vector \((Y,X)\), an independent resp. perfectly dependent version is a random vector \((Y^\perp,X^\perp)\) resp. \((Y^\top,X^\top)\) with \( Y^\perp \eqd Y^\top \eqd Y\) and \(X^\perp \eqd X^\top \eqd X\) such that \(Y^\perp\) and \(X^\perp\) are independent resp. \(Y^\top\) is perfectly dependent on \(X^\top\). We generally assume that \((\Omega,\cA,P)\) is non-atomic and thus admits an independent and perfectly dependent version.
            } \((Y^\top,X^\top)\) from the same Fr\'{e}chet class;
            \item \cite[Axiom O2 and O3]{siburg2013}, according to which \((Y^\top,X^\top)\prec (Y,X)\) implies that \(Y\) perfectly depends on \(X\), and \((Y,X)\prec (Y^\perp,X^\perp)\) implies that \(X\) and \(Y\) are independent;
            \item \cite[Axioms O4 and O5]{Dabrowska-1981} where minimal and maximal elements satisfy \(\E[Y|\XX]=\E Y\) a.s. and \(\E[Y|\XX] = Y\) a.s., respectively.
        \end{itemize}
        \item \label{reax5} Axioms \ref{axdepord4} and \ref{axdepord5} state that additional information provided by further predictor variables increases the strength of functional dependence, and remains unchanged precisely when the additional information is irrelevant given the existing information. For dependence measures, information monotonicity is studied in \cite{Figalli-2025}. On the level of dependence orderings, both axioms are new to the best of our knowledge. They yield a data processing inequality for \(\prec\) in the sense that \((Y,\ZZ)\prec (Y,\XX)\) whenever \(Y\) and \(\ZZ\) are conditionally independent given \(\XX\). This implies \((Y,{\bf h}(\XX))\prec (Y,\XX)\) for all measurable \({\bf h}\).
        \item Axiom \ref{axdepord6} is stronger than the invariance properties in \cite[Axiom O3]{Dabrowska-1981} and \cite[Axiom O3]{Shaked-2013} where \(g\) is assumed to be (increasing) linear. In \cite{siburg2013}, distributions from the same Fr\'{e}chet class are transformed to the copula setting and no further invariance properties are studied.
        \item Axiom \ref{axdepord7} on distributional invariance implies that \((Y,\XX)\prec (Y',\XX')\) only depends on the underlying copulas\footnote{A \(d\)-variate copula is the distribution function of a random vector \((U_1,\ldots,U_d)\) with \(U_i\sim U(0,1)\) for all \(i\). By Sklar's theorem, every \(d\)-variate distribution function \(F\) can be decomposed into \(F = C\circ (F_1,\ldots,F_d),\) where \(F_i\) is the \(i\)th marginal distribution function and \(C\) is a copula that is unique on \(\Ran(F_1)\times \cdots \times \Ran(F_d)\).} \(C_{Y,\XX}\) and \(C_{Y',\XX'}\) evaluated on the closed Cartesian products \(\overline{\Ran(F_Y)}\times \overline{\Ran(F_{X_1})}\times \cdots \overline{\Ran(F_{X_p})}\) and \(\overline{\Ran(F_{Y'})}\times \overline{\Ran(F_{X_1'})}\times \cdots \overline{\Ran(F_{X_{p'}'})}\) of the respective marginal distribution functions. 
    \end{enumerate}
\end{remark}

\subsection{A new global dependence order}\label{sec12}


To explain the idea how to construct a suitable dependence order that satisfies all the axioms above, let us have a closer look at the following representation of Chatterjee's rank correlation:
\begin{align}\label{repxi}
    \xi(Y,\XX) 
    = \alpha \, \int_{\mathbb{R}^{1+p}}  P(Y\geq y \mid \XX=\xx) ^2 
            \; \mathrm{d} (P^{Y} \otimes P^{\XX}) (y,\xx) -\beta,
\end{align}
where \(\alpha\) and \(\beta\) are positive constants depending only on \(\overline{\Ran(F_Y)}\).
Note that \(Y\) and \(\XX\) are generally dependent while the measure integrated against is the product measure as a consequence of the definition of \(\xi\).
It is straightforward to see that a pointwise ordering of the integrand in \eqref{repxi} is too strong to define a suitable dependence order.
If we split the integral in \eqref{repxi} into a double integral, there are two natural approaches for an ordering of conditional survival functions in squared mean. 


The first approach considers the variability of the conditional survival probability  \(P(Y\geq y \mid \XX=\xx)\) in \(y\) for each fixed \(\xx\,.\) This might be seen as the more natural and more intuitive approach supported by a variety of literature on stochastic orderings for conditional distributions;\footnote{For example, in optimal transport problems, optimal martingale couplings are based on a comparison of \(P^{Y\mid X \in (-\infty,x]}\) in convex order for all \(x\in \R\); see \cite[Theorem 1.8]{Beiglboeck-2016}} see \cite{Block-1988,Leskela-2017,Rueschendorf-1991,Whitt-1980,Whitt-1985}. In the setting of bivariate copulas, this approach is studied in \cite{siburg2013} where the dilation order and the dispersive order are used to construct \emph{regression dependence orders}, which imply an ordering of \(\xi\). However, as we discuss in Section \ref{subsecdispdil}, these orders are limited because they cannot be verified even in the simplest settings.

The second approach, which we explore in this paper, examines the variability of \(P(Y\geq y \mid \XX = \xx)\) in \(\xx\) for each fixed \(y\,.\) 
This perspective relates to the theory of random measures \cite{Kallenberg-2017}, and the works closest to our paper are \cite{Shaked-2012,Shaked-2013} where the variability of regression functions is studied.
As it turns out, our approach is much more powerful as it yields an ordering that satisfies all the axioms above and thus reflects the fundamental properties of \(\xi\). Further, this order can be verified for various standard models as we show in Theorem \ref{theadderrmod} and Section \ref{sec52}.

Pursuing the second approach, we introduce below a new dependence order that we denote as \emph{conditional convex order}. It compares conditional survival probabilities with respect to the convex order---the most common variability order studied in the literature; see \cite[Section 3]{Shaked-2007}. To this end, recall that, for bounded random variables \(S\) and \(T\), the convex order \(S\leq_{cx} T\) is defined by \(\E \varphi(S) \leq \E \varphi(T)\) for all convex functions \(\varphi\colon \R\to \R\). We denote by \(q_Y(v) := \inf\{y \mid F_Y(y)\geq v\},\) \(v\in (0,1)\), the left-continuous quantile function of \(Y\), also known as generalized inverse of \(F_Y\).

\begin{definition}[Conditional convex order]\label{defccx}~\\
    Let \((Y,\XX)\) and \((Y',\XX')\) be random vectors. Then we define the \emph{conditional convex order} \((Y,\XX)\preccurlyeq_{ccx} (Y',\XX')\) by \(P(Y\geq q_Y(v) \mid \XX) \leq_{cx} P(Y'\geq q_{Y'}(v) \mid \XX')\) for all \(v\in (0,1)\).
\end{definition}

\noindent We write \((Y,\XX)=_{ccx} (Y',\XX')\) if \((Y,\XX)\preccurlyeq_{ccx} (Y',\XX')\) and \((Y,\XX)\succcurlyeq_{ccx} (Y',\XX')\). 

Before we present our main results, let us discuss why the conditional convex order is a suitable candidate for a global dependence order underlying \(\xi\).
First note that a necessary condition for \((Y,\XX)\preccurlyeq_{ccx} (Y',\XX')\) is \(\overline{\Ran(F_Y)} = \overline{\Ran(F_{Y'})}\). This follows from the definition of the conditional convex order using that \(S\leq_{cx} T\) implies \(\E S = \E T\). 
We refer to this condition as \emph{marginal constraint}. 
In this case, the denominator of \(\xi\) in \eqref{Tuniv} is equal for \((Y,\XX)\) and \((Y',\XX')\).

Second, the numerator of \(\xi\) is increasing in the conditional convex order. This is a direct consequence of the fact that \(S\leq_{cx} T\) implies \(\Var(S) \leq \Var(T)\). It immediately follows that Chatterjee's rank correlation is increasing in \(\preccurlyeq_{ccx}\), i.e., 
\begin{align}\label{eqccxxi}
    (Y,\XX) \preccurlyeq_{ccx} (Y',\XX') \quad \text{implies} \quad \xi(Y,\XX) \leq \xi(Y',\XX').
\end{align}
We provide this statement in a more general framework in Theorem \ref{theconsccx}.

Third, as studied in \cite{Figalli-2025}, convexity is essential for information monotonicity and, in particular, for zero-independence and max-functionality of dependence measures.
It is straightforward to verify that minimal and maximal elements in conditional convex order characterize zero-independence and max-functionality: On the one hand, \(Y\) and \(\XX\) are independent if and only if, for all \(v\in (0,1)\), \(P(Y\geq q_Y(v) \mid \XX)\) is a.s. constant and thus minimal in convex order. On the other hand, \(Y\) perfectly depends on \(\XX\) if and only if, for all \(v\in (0,1)\), \(P(Y\geq q_Y(v) \mid \XX)\) attains a.s.~values in \(\{0,1\}\) and thus is maximal in convex order within the class of \([0,1]\)-valued random variables with mean \(P(Y\geq q_Y(v))\).
More generally, the conditional convex order allows a simple interpretation: 
Little/Large variability of the conditional probabilities \(P(Y\geq q_Y(v) \mid \XX)\) 
 in convex order indicates low/strong dependence of \(Y\) on \(\XX\)---with the extreme cases of independence and perfect dependence discussed before.

\begin{remark}\label{remccx}
    \begin{enumerate}[(a)]
        \item \label{remccxa}
        To derive ordering results for models such as \eqref{eqadderrmod}, we allow in Definition \ref{defccx} a comparison of random vectors with different marginal distributions. If \(Y\) and \(Y'\) have the same distribution, then \((Y,\XX)\preccurlyeq_{ccx} (Y',\XX')\) simplifies to \(P(Y\geq y\mid \XX) \leq_{cx} P(Y'\geq y\mid \XX')\) for all \(y\in \R\). For bivariate stochastically increasing distributions from the same Fr\'{e}chet class, the conditional convex order coincides with the concordance order, as we show in Proposition \ref{rembivSchur}. 
        Recall that a bivariate distribution (function) is \emph{stochastically increasing (SI)} if there exists a bivariate random vector \((W_1,W_2)\) with this distribution (function) such that \(W_1\) is SI in \(W_2\), i.e., \(\E[f(W_1)\mid W_2 = t]\) is increasing in \(t\) for all increasing functions \(f\) such that the expectations exist. 
        \item \label{remccx2} The conditional convex order is defined by comparing conditional survival probabilities \(P(Y\geq q_Y(v) \mid \XX)\) in convex order for \emph{all} \(v\in (0,1)\). Equivalent versions can be achieved for strict survival probabilities and cumulative probabilities. For example, \((Y,\XX)\preccurlyeq_{ccx} (Y',\XX')\) is equivalent to
        \begin{align}\label{eqequiccx}
            P(Y\leq q_Y(v)\mid \XX)\leq_{cx} P(Y'\leq q_{Y'}(v) \mid \XX') \quad \text{for } \lambda\text{-\emph{almost all }} v\in (0,1);
        \end{align}
        see Proposition \ref{propversccx}. The exceptional null set with respect to the Lebesgue measure \(\lambda\) in \eqref{eqequiccx} is required due to possibly different continuity properties of the transformations \(F_Y\circ q_Y\) and \(F_{Y'}\circ q_{Y'}\) at the jump points of \(F_Y\) and \(F_{Y'}\), respectively; see Example \ref{exdiscY} and Proposition \ref{Prop.Inverse}\,\ref{Prop.Inverse2} and \ref{Prop.Inverse3} for details. In contrast, the transformed survival functions \(t \mapsto P(Y\geq q_Y(t))\) and \(t\mapsto P(Y'\geq q_{Y'}(t))\) are left-continuous and, hence, they coincide under the marginal constraint for all \(t\). This motivated us to define the conditional convex order through survival probabilities.
    \end{enumerate}
\end{remark}

\subsection{Main results}\label{secmainresults}

The following main result shows that the conditional convex order in Definition \ref{defccx} satisfies all the proposed axioms for a global dependence ordering.
Therefore, we consider for a closed set \(A\subseteq [0,1]\) the class 
\begin{align}\label{defR_A}
    \cR_A := \{(Y,\XX) \colon (\Omega,\cA,P) \to (\mathbb{R}^{1+p},\cB(\R^{1+p})) \mid  \overline{\Ran(F_Y)} = A\,, \, p\in \N \}
\end{align}
of random vectors.
 
\begin{theorem}[Fundamental properties of \(\preccurlyeq_{ccx}\)]\label{thefundpropS}
 For \(\cR = \cR_A\), the conditional convex order \(\preccurlyeq_{ccx}\) in Definition \ref{defccx} satisfies the Axioms \ref{axdepord0} -- \ref{axdepord7}.
\end{theorem}


\begin{remark}\label{remSchurOrder} 
\begin{enumerate}[(a)]
    \item The class \(\cR_A\) in Theorem \ref{thefundpropS} consists of all random vectors \((Y,\XX)\) that are at least two-dimensional and satisfy the marginal constraint with respect to the set \(A\). 
    For example, if \(A=[0,1]\), then \(\cR_A\) includes all random vectors \((Y,\XX)\) such that \(Y\) has a continuous distribution function; note that there is no assumption on the distribution or dimension of \(\XX\).
    The set \(\cR_A\) includes the extreme elements in \(\preccurlyeq_{ccx}\), for example, the comonotone random vector \((Y,Y)\) as maximal element and the independent random vector \((Y,c)\), \(c\in \R\), as minimal element.
    \item Under the marginal constraint, independent and perfectly dependent random vectors are \emph{globally} extreme elements in \(\preccurlyeq_{ccx}\) and are thus comparable to all random vectors in \(\cR_A\).
    Such a property does not hold for the regression dependence orders in \cite{siburg2013} where independent random variables may not be comparable to other random vectors even if they have the same marginal distributions.
    \item The conditional convex order is \emph{not} symmetric, i.e., for bivariate random vectors \((Y,X)\) and \((Y',X')\), the relation \((Y,X)\preccurlyeq_{ccx}(Y',X')\) does not imply \((X,Y)\preccurlyeq_{ccx}(X',Y')\). This is in line with the asymmetric notion of perfect dependence, where \(Y\) being a function of \(X\) does not imply that \(X\) is a function of \(Y\).
    \item \label{remSchurOrder4}  As a consequence of information monotonicity and characterization of conditional independence, the conditional convex order satisfies the data processing inequality
\((Y,\ZZ)\preccurlyeq_{ccx} (Y,\XX)\) whenever \(Y\) and \(\ZZ\) are conditionally independent given \(\XX\,.\) 
It follows that \((Y,\bold{h}(\XX)) \preccurlyeq_{ccx} (Y,\XX)\) for all measurable functions \(\bold{h}\).
Further, the conditional convex order is self-equitable in the sense that \((Y,\bold{h}(\XX)) =_{ccx} (Y,\XX)\) for all measurable functions \(\bold{h}\) such that \(Y\) and \(\XX\) are conditionally independent given \(\bold{h}(\XX)\,.\)
\end{enumerate}
\end{remark}

In the next theorem, our second main result, we give two characterizations of the conditional convex order---both in terms of \emph{bivariate} SI random vectors. These random vector are compared with respect to the concordance order \(\leq_c\), the \emph{positive} dependence ordering mentioned above; see also Section \ref{sec_cccx}. 
The first characterization relies on conditional comonotonicity---a dependence concept that underlies Wasserstein correlations.
The second characterization is based on a dimension reduction principle that we study in detail in Section \ref{secdimredprin}. 
Regarding the notation, we consider for \((Y,\XX)\) the rearranged quantile transform \((0,1)^2 \ni (t,u) \mapsto q_{Y;\XX}^{\uparrow u}(t)\) defined in Equation \eqref{defyu} below. 
Then,
for independent \(U,V\sim U(0,1)\),
the \emph{bivariate} random vector \((F_Y\circ q_{Y;\XX}^{\uparrow U}(V),U)\) is equivalent to \((Y,\XX)\) in conditional convex order, that is,
\begin{align}\label{eqredrv}
    (F_Y\circ q_{Y;\XX}^{\uparrow U}(V),U) =_{ccx} (Y,\XX);
\end{align}
see Theorem \ref{lemtrafSI}.
We refer to such a bivariate random vector as \emph{reduced random vector}.
As we study in Section \ref{secdepmea}, each of the equivalent versions in the following theorem generates a class of dependence measures that are \(\preccurlyeq_{ccx}\)-increasing and inherit all their fundamental properties from the conditional convex order.


\begin{theorem}[Characterization of \(\preccurlyeq_{ccx}\) by bivariate concordance order]\label{propcharS}~\\
For \(U,V\sim U(0,1)\) independent, the following statements are equivalent:
    \begin{enumerate}[(i)]
        \item \label{propcharS1} \((Y,\XX)\preccurlyeq_{ccx} (Y',\XX')\,,\)
        \item \label{propcharS3} \((F_Y \circ F_{Y|\XX}^{-1}(V), V) \geq_c (F_{Y'}\circ F_{Y'|\XX'}^{-1}(V),V)\)  with \(V\) independent of \(\XX\) and \(\XX'\),
        \item \label{propcharS2} \((F_Y \circ q_{Y;\XX}^{\uparrow U}(V),U)\leq_{c} (F_{Y'} \circ q_{Y';\XX'}^{\uparrow U}(V),U)\,.\) 
    \end{enumerate}
    In particular, we have \((Y,\XX)=_{ccx} (Y',\XX')\) if and only if \((F_Y \circ F_{Y|\XX}^{-1}(V), V) \eqd (F_{Y'}\circ F_{Y'|\XX'}^{-1}(V),V)\) if and only if \((F_Y \circ q_{Y;\XX}^{\uparrow U}(V),U)\eqd ( F_{Y'} \circ q_{Y';\XX'}^{\uparrow U}(V), U)\).
\end{theorem}

\begin{remark}\label{remcharccx}
\begin{enumerate}[(i)]
\item Recall from \eqref{eqccxxi} that the conditional convex order underlies Chatterjee’s rank correlation. In Theorem \ref{theconsccx}, we study classes of related dependence measures that are \(\preccurlyeq_{ccx}\)-increasing and constructed via conditional distribution functions.
\item By Theorem \ref{propcharS}\,\ref{propcharS3}, \(\preccurlyeq_{ccx}\) is characterized via a comparison of \emph{bivariate} random vectors in concordance order. These random vectors are conditionally comonotone and SI.
Conditional comonotonicity follows from the representation
\(P(F_{Y|\XX}^{-1}(V) \leq y, V\leq y') = \int_{\R^p} \min\{F_{Y|\XX=\xx}(y) , y'\} \de P^\XX(\xx)\). Note that concavity of \(u\mapsto \min\{u,y\}\) explains the reversed inequality sign in \ref{propcharS3}.
Wasserstein correlations \cite{wiesel2022} are based on conditional comonotonicity and, as we show in Theorem \ref{theOT}, they are \(\preccurlyeq_{ccx}\)-increasing. While \(\preccurlyeq_{ccx}\) is defined via distribution/survival functions, the characterization in \ref{propcharS3} is a quantile-based version. Independence of \((F_{Y|\XX}^{-1}(V),V)\) characterizes perfect dependence of \(Y\) on \(\XX\). Conversely, comonotonicity of \((F_{Y|\XX}^{-1}(V),V)\) corresponds to independence of \(\XX\) and \(Y\). 
\item \label{remcharccx3} Rearranged dependence measures \cite{strothmann2022} are constructed via measures of concordance and reduced random vectors; see \eqref{deffunccon}.
As a consequence of Theorem \ref{propcharS}\,\ref{propcharS2}, such measures are \(\preccurlyeq_{ccx}\)-increasing; see Theorem \ref{cordepmeas}. 
The reduced random vectors \((F_Y\circ q_{Y;\XX}^{\uparrow U}(V),U)\) are SI.
They are constructed via the dimension-reduction principle in Section \ref{secdimredprin}, and, notably, they are \(\preccurlyeq_{ccx}\)-equivalent to \((Y,\XX)\); see \eqref{eqredrv}. Independence/Comonotonicity of \((q_{Y;\XX}^{\uparrow U}(V),U)\) characterizes independence/perfect dependence of \(Y\) and/on \(\XX\).
\item The concordance order requires equal marginal distributions. Consequently, we have \(F_Y(F_{Y|\XX}^{-1}(V))\eqd F_{Y'}(F_{Y'|\XX'}^{-1}(V))\) and \(F_Y(q_{Y;\XX}^{\uparrow U}(V)) \eqd F_{Y'}(q_{Y';\XX'}^{\uparrow U}(V))\) in Theorem \ref{propcharS}\,\ref{propcharS3} and \ref{propcharS2}. Either condition is equivalent to the marginal constraint \(\overline{\Ran(F_Y)} = \overline{\Ran(F_{Y'})}\) resulting from the definition of the conditional convex order. 
\end{enumerate}
\end{remark}

The next theorem, which constitutes our third main result, provides new insights into the behavior of dependence measures like \(\xi\) beyond the boundary values \(0\) and \(1\). While their construction is based on the zero-independence and max-functionality, these measures typically satisfy further useful properties such as information monotonicity; see \cite{Figalli-2025}.
However, as we discuss in Remark \ref{remadderrmod} below,
information monotonicity alone does not explain monotonicity of dependence measures, for instance, in the noise level of additive error models.
In Theorem \ref{theadderrmod}, we show that the additive error model \eqref{eqadderrmod} is decreasing in the parameter \(\sigma\) with respect to the conditional convex order, and hence so are all \(\preccurlyeq_{ccx}\)-increasing dependence measures.


\begin{theorem}[Additive error models]\label{theadderrmod}~\\
Consider \(Y = f(\XX) + \sigma \varepsilon\) and \(Y' = f(\XX) + \sigma'\varepsilon\) for \(0\leq \sigma < \sigma'.\) Assume that \(F_Y\) and \(F_{Y'}\) are continuous and that \(\varepsilon\) is independent from \(f(\XX)\,.\) 
    Then we have \((Y,\XX) \succcurlyeq_{ccx} (Y',\XX).\)
\end{theorem}

\begin{remark}\label{remadderrmod}
    \begin{enumerate}[(a)]
        \item Besides the continuity assumption on \(Y\) and the independence condition between \(f(\XX)\) and \(\varepsilon\), there are no assumptions on the distributions of the underlying random variables in Theorem \ref{theadderrmod}. In particular, the error \(\varepsilon\) is \emph{not} assumed to follow a symmetric distribution, to be infinitely divisible, or to have zero mean.
 Extensions to non-additive error models are studied in Section \ref{secsufdimred}.
\item \label{remadderrmod3} Consider \(Y_{\sigma} := X + \sigma \varepsilon\) where \(X\) and \(\varepsilon\) are standard normal and independent. Then, for \(0<\sigma' < \sigma''\), we have \((Y_{\sigma'},X)\succcurlyeq_{ccx} (Y_{\sigma''},X)\) by Theorem \ref{theadderrmod} and, obviously, \((Y_{\sigma'},X)\not =_{ccx} (Y_{\sigma''},X)\). Since the \(\sigma\)-algebras generated by \((Y_{\sigma'},X)\) and \((Y_{\sigma''},X)\) coincide, there is no information gain for \(Y_\sigma\) in the sense of Axiom \ref{axdepord4} when decreasing the parameter \(\sigma\). Thus, 
information monotonicity is not sufficient for a comparison in \(\preccurlyeq_{ccx}\).
    \end{enumerate}
\end{remark}

\subsection{Our contribution to the literature and structure of the paper}

We propose Axioms \ref{axdepord0}--\ref{axdepord7} as natural and desirable axioms for a global dependence ordering that describes the strength of functional dependence of a random variable \(Y\) on a random vector \(\XX\). 
    As shown in Theorem \ref{thefundpropS}, our first main result, the conditional convex order introduced in Definition \ref{defccx} satisfies all these axioms.
    To the best of our knowledge, there is no work that addresses all these axioms, and we are not aware of any dependence ordering that satisfies all these fundamental properties. 
    We refer to 
    \cite{Figalli-2025,Kimeldorf-1989,Mori-Szekely-2019,Renyi-1959,Scarsini-1984} for various axioms at the level of measures of association, and to \cite{Joe-2014,Muller-2002,Shaked-2007} for an overview of dependence orderings.

In Theorem \ref{propcharS}, our second main result, we present two characterizations of \(\preccurlyeq_{ccx}\) based on the bivariate concordance order. The first characterization in Theorem \ref{propcharS}\,\ref{propcharS3} relies on conditionally comonotone random vectors.
The second characterization in Theorem \ref{propcharS}\,\ref{propcharS2} is based on a dimension-reduction principle using the concept of rearrangements, which we study in detail in Section \ref{secsuffsch}.
More precisely, we use the Hardy-Littlewood-Polya theorem to rearrange conditional distribution functions and to characterize \(\preccurlyeq_{ccx}\) in terms of the Schur order; see Theorem \ref{thecharSchur}. This characterization is particularly useful for verifying \(\preccurlyeq_{ccx}\) in several models. As a first example, we verify \(\preccurlyeq_{ccx}\) for bivariate Bernoulli distributions in Proposition \ref{Prop.Bernoulli}. 
The dimension reduction principle is formalized in Theorem \ref{lemtrafSI} where we construct, for a given random vector \((Y,\XX)\), a \(\preccurlyeq_{ccx}\)-equivalent bivariate SI random vector.
In Example \ref{exdimrednormal}, we illustrate the dimension-reduction procedure for the multivariate normal distribution.

Section \ref{sec52} builds on the dimension reduction procedure and provides sufficient conditions for verifying \(\preccurlyeq_{ccx}\) under stochastic monotonicity assumptions.
In particular, we verify \(\preccurlyeq_{ccx}\) for copula-based models in Theorem \ref{propmvSchur}. This theorem is also used to prove Theorem \ref{theadderrmod} on additive error models.
Somewhat surprisingly, Theorem \ref{propschurnormal} shows that multivariate normal distributions are always comparable in conditional convex order, and
the strength of functional dependence is solely encoded in the Schur complements of the covariance matrices.

In Section \ref{secdepmea}, we show that the conditional convex order generates several classes of dependence measures recently studied in the statistics literature. 
Due to Theorem \ref{theconsccx}, \(\preccurlyeq_{ccx}\) underlies Chatterjee's rank correlation and related dependence measures constructed via conditional distribution functions.
As we show in Theorem \ref{theOT}, \(\preccurlyeq_{ccx}\) also underlies Wasserstein correlations  \cite{wiesel2022}.
In Theorem \ref{cordepmeas}, we prove that rearranged dependence measures \cite{strothmann2022} are \(\preccurlyeq_{ccx}\)-increasing. Similar results hold true for the integrated \(R^2\) \cite{Azadkia-2025} (see Remark \ref{remaza}\,\ref{remaza2}) and the kernel partial correlation \cite{deb2020b} (see Remark \ref{remnorm}\,\ref{remnorm2}).

In Section \ref{seccompS}, we compare the conditional convex order with related positive and global dependence orderings from the literature \cite{siburg2013,Shaked-2012,Shaked-2013
}. In particular, we show that none of the existing orderings satisfies Axioms \ref{axdepord0}–\ref{axdepord7}, nor do they describe the behavior of Chatterjee’s rank correlation and related dependence measure across various models.

Some useful properties of generalized inverses, convex order, and transformations of random vectors are given in Appendix \ref{appA}. All proofs are deferred to Appendix \ref{App:Proofs}. 
Throughout the paper,
boldface symbols \(\xx\) and \(\XX\) are used for vectors and random vectors, respectively, and standard symbols \(y\) and \(Y\) for real numbers and real-valued random variables.






\section{Characterizations of \(\preccurlyeq_{ccx}\)}\label{secsuffsch}

Recall Theorem \ref{propcharS} where two characterizations of the conditional convex order are given. 
In Theorem \ref{thecharSchur} below, we establish a third characterization of \(\preccurlyeq_{ccx}\) in terms of the Schur order, which is based on rearrangements and can be verified in several settings. 
The link to the \(\preccurlyeq_{ccx}\)-order is provided by the Hardy--Littlewood--P\'{o}lya theorem (Lemma \ref{lemhlpt}), which characterizes the (standard) convex order in terms of the Schur order.
The dimension-reduction procedure in Section \ref{secdimredprin} relies on rearrangements and is  of central importance for our paper.


\subsection{Decreasing rearrangements}

In this subsection, we characterize the conditional convex order through rearrangements. This concept plays a key role for verifying \(\preccurlyeq_{ccx}\) in several settings. Large parts of the paper, in particular, Sections \ref{secdimredprin}, \ref{sec52}, and \ref{secrearrdepmea}, as well as Theorem \ref{propcharS}\,\ref{propcharS2} and Theorem \ref{theadderrmod}, are based on rearrangements.

For 
\(d,d'\in \N\,,\) let \(f\colon (0,1)^d\to \R\) and \(g\colon (0,1)^{d'}\to \R\) be integrable functions. Then the \emph{Schur order} \(f\prec_S g\) is defined by
\begin{align}\label{defSchurfun}
\int_0^x f^*(t)\de t &\leq \int_0^x g^*(t)\de t ~~~\text{for all } x\in (0,1) \text{ and} ~~~
\int_0^1 f^*(t)\de t = \int_0^1 g^*(t)\de t\,,
\end{align}
where \(h^*\colon (0,1)\to \R\) denotes the \emph{decreasing rearrangement} of an integrable function \(h\colon (0,1)^k\to \R\,,\) \(k\in \N\,,\) i.e., the essentially (with respect to the Lebesgue measure) uniquely determined decreasing function \(h^*\) such that \(\lambda(h^*\geq w)=\lambda^k(h\geq w)\) for all \(w\in \R\,;\) see \cite{Chong-1974}. Here, \(\lambda^k\) denotes the Lebesgue measure on \((0,1)^k\).

The Schur order is characterized by the following multivariate Hardy-Littlewood-Polya theorem; see \cite[Theorem 2.5]{Chong-1974}. 

\begin{lemma}[Hardy--Littlewood--Polya theorem]\label{lemhlpt}
The following are equivalent:
\begin{enumerate}[(i)]
    \item \(f\prec_S g\)
    \item \label{lemhlpt2} \(\int_{(0,1)^d} \varphi\circ f \de \lambda^d \leq \int_{(0,1)^{d'}} \varphi\circ g \de \lambda^{d'}\) for all convex functions \(\varphi\colon \R \to \R\) such that the integrals exist.
\end{enumerate}
\end{lemma}

To apply the Hardy-Littlewood-Polya theorem for characterizing \(\preccurlyeq_{ccx}\), we use quantile transformations to get back to the unit interval/hypercube. Recall that \(q_Y\) is the left-continuous quantile function of \(Y\).
We also need the multivariate quantile transform \({\bf q}_\XX\colon (0,1)^p \to \R^p\) defined in \eqref{MQT1}. It satisfies \(\XX \eqd {\bf q}_\XX(\UU)\) for \(\UU\) uniform on \((0,1)^p\).
Now, define for 
 \(v\in (0,1)\) the function \(\eta_{Y|\XX}^v\colon (0,1)^p \to \R\) by
\begin{align}\label{desetay}
    \eta_{Y|\XX}^v(\uu):=&P(Y\leq q_Y(v) \mid \XX = {\bf q}_\XX(\uu))\,.
\end{align}

Due to the following result, the conditional convex order can be characterized by a comparison of the functions \(\eta_{Y|\XX}^v\) and \(\eta_{Y'|\XX'}^v\) in Schur order. 

\begin{theorem}[Characterization of \(\preccurlyeq_{ccx}\) by the Schur order]\label{thecharSchur}~\\
    The following statements are equivalent:
    \begin{enumerate}[(i)]
        \item \label{thecharSchur1} \((Y,\XX) \preccurlyeq_{ccx} (Y',\XX')\)
        \item \label{thecharSchur2} \(\eta_{Y|\XX}^v \prec_S \eta_{Y'|\XX'}^v\) for \(\lambda\)-almost all \(v\in (0,1)\).
    \end{enumerate}
\end{theorem}

\begin{remark}
    \begin{enumerate}[(a)]
    \item In our preliminary work \cite{Ansari-Rockel-2023, Ansari-LFT-2023}, we introduced the so-called \emph{Schur order for conditional distributions} through the characterization in Theorem \ref{thecharSchur}\,\ref{thecharSchur2} and derived ordering results for Chatterjee’s rank correlation within bivariate copula models. In this paper, we consider the conditional convex order, which is equivalent to the Schur order for conditional distributions but conceptually simpler. 
    \item Global dependence orderings based on the Schur order in \eqref{defSchurfun} are also studied in \cite{Joe-1987,Scarsini-1990}. In the first reference, densities are compared in the Schur order to analyze the strength of dependence \emph{within} multivariate distributions. The second reference studies so-called Lorentz curves (which are related to the Schur order) to compare likelihood ratios between probability measures with finite support. In contrast, in our approach, we compare conditional distribution functions and analyze the strength of dependence of \(Y\) on \(\XX\). Note that the definition of \(\preccurlyeq_{ccx}\) does not require assumptions like absolute continuity or distributions with finite support.
    \end{enumerate}
\end{remark}

To give a first illustration how to verify \(\preccurlyeq_{ccx}\) by Theorem \ref{thecharSchur}, we construct the decreasing rearrangement of \(\eta_{Y|\XX}^v\) in \eqref{desetay} for bivariate Bernoulli distributions. 

\begin{example}[Bivariate Bernoulli distribution] \label{Ex.Bernoulli}
Consider a bivariate Bernoulli distributed random vector $(Y,X)$ taking on values $(0,0), (0,1), (1,0)$ and $(1,1)$. Denote by $p_{ji} := P(Y=j,X=i)$, $j,i \in \{0,1\}$, the corresponding probability of occurrence such that $q := P(Y=1) = p_{10} + p_{11} \in (0,1)$ and
$p := P(X=1) = p_{01} + p_{11} \in (0,1)$. 
With this notation, the bivariate Bernoulli distribution has in total \(6\) parameters with \(3\) degrees of freedom, where \(q\) determines the distribution of \(Y\) and where \(\alpha := p_{00}/(1-p)\) and \(\beta := p_{01}/p\) determine the conditional probabilities \(P(Y=0\mid X=0)\) and \(P(Y=0\mid X=1)\). 
The conditional distribution function $\eta_{Y|X}^v$ in \eqref{desetay} and its decreasing rearrangement $(\eta_{Y|X}^v)^\ast$ are given by
\begin{align}
    \eta_{Y|X}^v(u) &= \begin{cases}
    \alpha\, \1_{(0,1-p]}(u) + \beta \, \1_{(1-p,1]}(u) & \text{if } v\in (0,1-q],\\
        1 & \text{if } v\in (1-q,1],
    \end{cases}\\
  \label{eqdecrrearrBd}  (\eta_{Y|X}^v)^*(u) &= \begin{cases}
    (\alpha \vee \beta)\, \1_{(0,z]}(u) + (\alpha \wedge \beta)\, \1_{(z,1]}(u) & \text{if } v\in (0,1-q],\\
        1 & \text{if } v\in (1-q,1],
    \end{cases}
\end{align}
where \(z := 1-p\) if \(\alpha \geq \beta\) and \(z := p\) otherwise. Here, \(\vee\) and \(\wedge\) denote the maximum and minimum of two real numbers.
The \emph{integrated} decreasing rearrangement $x \mapsto \int_{0}^x (\eta_{Y|X}^v)^\ast(u) \de u$ in \eqref{defSchurfun} is piecewise linear with a breakpoint at $z$ and slope $\alpha\vee \beta $ on $[0,z]$ and slope $\alpha \wedge \beta $ on $(z,1]$; see Figure \ref{figExBernoulli} for an illustration.
\begin{figure}
    \centering
    \includegraphics[width=0.55\linewidth, trim = 0 25 0 0]{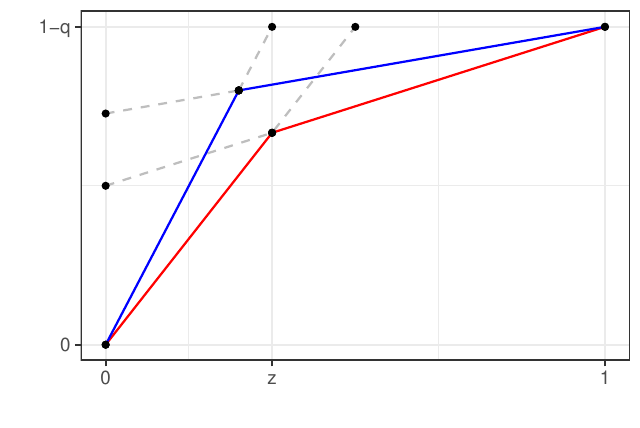}
    \caption{The red curve depicts the piecewise linear \emph{integrated} decreasing rearrangement $x \mapsto \int_{0}^x (\eta_{Y|X}^v)^\ast(u) \de u$, \(v\in (0,1-q]\), of the Bernoulli-distributed random vector $(Y,X)$ discussed in Example \ref{Ex.Bernoulli}.
    For any Bernoulli-distributed random vector $(Y',X')$ with $(Y,X) \preccurlyeq_{ccx} (Y',X')$ or, equivalently, \(\eta_{Y|X}^v \prec_S \eta_{Y'|X'}^v\) for Lebesgue almost all \(v\in (0,1-q]\),
    its integrated decreasing rearrangement (blue curve) lies above the red curve and hence must have a slope greater than $\alpha\vee \beta$ in the interval starting at $0$ and a slope less than $\alpha\wedge \beta$ in the interval ending at $1$.
    Due to the marginal constraint both the red and blue functions equal $1-q$ at point $1$.}
    \label{figExBernoulli}
\end{figure}
\end{example}

To characterize \(\preccurlyeq_{ccx}\) for bivariate Bernoulli distributions in terms of Theorem \ref{thecharSchur}, we need to compare integrals of the decreasing rearrangements in \eqref{eqdecrrearrBd}; see Figure \ref{figExBernoulli}. The following result further illustrates the different concepts of monotone and functional dependence.

\begin{proposition}[\(\preccurlyeq_{ccx}\)-ordering of bivariate Bernoulli distributions] \label{Prop.Bernoulli}~\\
Let $(Y,X)$ and $(Y',X')$ be Bernoulli distributed random vectors with parameters as in Example \ref{Ex.Bernoulli}. Then the following two statements are equivalent:
\begin{enumerate}[(i)]
    \item \label{Prop.Bernoulli1} $(Y,X) \preccurlyeq_{ccx} (Y',X')$,
    \item \label{Prop.Bernoulli2} $q=q^\prime$ as well as \(\alpha'\wedge \beta'\leq \alpha \wedge \beta\) and \(\alpha \vee \beta \leq \alpha'\vee \beta'\).
\end{enumerate}
Further, we have that
\begin{enumerate}[(i)]\setcounter{enumi}{2}
    \item \label{Prop.Bernoulli3} \(X\) and \(Y\) are independent if and only if \(\alpha = \beta\);
    \item \label{Prop.Bernoulli4} \(Y\) perfectly depends on \(X\) if and only if \(\alpha \wedge \beta = 0\) and \(\alpha \vee \beta = 1\); 
    \item \label{Prop.Bernoulli5} \(X\) and \(Y\) are comonotone if and only if \(\alpha = 
    \tfrac{1-q}{1-p}
    \wedge 1\);
    \item \label{Prop.Bernoulli6} \(X\) and \(Y\) are countermonotone if and only if \(\alpha = 0 \vee 
    (1-\tfrac{q}{1-p})
    \).
\end{enumerate}
\end{proposition}

\begin{remark} Consider the setting in Proposition \ref{Prop.Bernoulli}.
\begin{enumerate}[(a)]
    \item In \ref{Prop.Bernoulli2}, the condition \(q=q'\) follows from the marginal constraint of the conditional convex order, i.e., \(\{0,1-q,1\} = \overline{\Ran(F_Y)} = \overline{\Ran(F_{Y'})} = \{0,1-q',1\}\). 
    Large and small values of \(\alpha\) and \(\beta\) yield strong couplings between \(Y\) and \(X\), and hence a stronger functional dependence in conditional convex order than moderate values of \(\alpha\) and \(\beta\).
    \item There are several alternatives to characterize independence, for example, \(\alpha = 1-q\).
    \item In the characterization of perfect dependence in \ref{Prop.Bernoulli4}, the condition \(\alpha\wedge \beta = 0\) \emph{and} \(\alpha\vee \beta = 1\) corresponds to the case \(Y = X\) \emph{or} \(Y = 1-X\) almost surely. 
    In these cases, it is \(p=q\) and \(p=1-q\), respectively. Due to \ref{Prop.Bernoulli5} and \ref{Prop.Bernoulli6}, perfect dependence of \(Y\) on \(X\) implies comonotonicity or countermonotonicity of \((Y,X)\); however, the converse is not true because \((Y,X)\) being comonotone (resp. countermonotone) only means that \(Y = q_Y(U)\) and \(X = q_X(U)\) (resp. \(X = q_X(1-U)\)) a.s. for \(U\sim U(0,1)\). Consequently, the extreme cases of $\leq_c$ and $\preccurlyeq_{ccx}$ generally differ; see also Section \ref{sec_cccx}.
    \item \(Y\) is SI in \(X\) if and only if \(\alpha\geq \beta\). If, additionally, \(p=p'\), \(q=q'\), and \(\alpha'\geq \beta'\), then we have \((Y,X)\preccurlyeq_{ccx} (Y',X')\) \(\Longleftrightarrow\) \(\alpha\leq \alpha'\) \(\Longleftrightarrow\) \(p_{00}\leq p_{00}'\) \(\Longleftrightarrow\) \((Y,X)\leq_c (Y',X')\). This confirms Proposition \ref{rembivSchur} in the case of Bernoulli distributions.
\end{enumerate}
\end{remark}

\subsection{A dimension reduction principle}\label{secdimredprin}

In this subsection, we reduce \((Y,\XX)\) to a \(\preccurlyeq_{ccx}\)-equivalent bivariate SI random vector using the concept of decreasing rearrangements. The dimension reduction principle underlies Theorem \ref{propcharS}\,\ref{propcharS2} and Theorem \ref{theadderrmod} as well as the results in Sections \ref{sec52} and \ref{secrearrdepmea}.
For convenience, we work with cumulative probabilities and distribution functions instead of survival probabilities and survival functions.

To transform \((Y,\XX)\) into a \(\preccurlyeq_{ccx}\)-equivalent bivariate random vector, we denote by \([0,1]\ni u\mapsto F_u(y)\) the decreasing rearrangement of \(\uu \mapsto P(Y\leq y \mid \XX = q_\XX(\uu))\). Recall that \(u\mapsto F_u(y)\) is \(\lambda\)-almost surely uniquely determined. In the sequel, we will consider versions such that \(y\mapsto F_u(y)\) is a distribution function for all \(u\in (0,1)\); see Lemma \ref{lemrearrqt} for the existence of such rearrangements. 
Setting \(y = q_Y(v)\) for \(v\in (0,1)\), we observe that the function \(u \mapsto F_u(q_Y(v))\) is the decreasing rearrangement of \(\eta_{Y|\XX}^v\) in \eqref{desetay}, i.e.
\begin{align}\label{definreacdf}
     (\eta_{Y|\XX}^v)^*(u) = F_u(q_Y(v)) \quad \text{for all } u\in (0,1).
\end{align}
Using that \(F_u\) is a distribution function on \(\R\), we define its associated quantile function by
\begin{align}\label{defyu}
    q_{Y;\XX}^{\uparrow u}(t) := F_u^{-1}(t), \quad t\in (0,1),
\end{align}
which we refer to as \emph{rearranged quantile transform}. To explain the choice of the notation in \eqref{defyu}, let \(V\sim U(0,1)\).
First, note that \(q_{Y;\XX}^{\uparrow u}(V)\) is a random variable with distribution function \(F_u\).
Since, by construction, \(F_u(y)\) is decreasing in \(u\) for all \(y\in \R\), the family \(\{q_{Y;\XX}^{\uparrow u}(V)\}_{u\in [0,1]}\) of random variables is increasing in \(u\) with respect to the \emph{stochastic order}. This means that \(u_1\leq u_2\) implies \(\E f(q_{Y;\XX}^{\uparrow u_1}(V)) \leq \E f(q_{Y;\XX}^{\uparrow u_2}(V))\) for all increasing functions \(f\) such that the expectations exist.

\begin{figure}
    \centering
\begin{tikzpicture}[scale=1.2, >=stealth, every node/.style={anchor=center}]

\node (A) at (0, 2.2) {
            $\frac{1}{40}
            \begin{bmatrix}
            \makebox[1em]{3} & \makebox[1em]{4} & \makebox[1em]{1} & \makebox[1em]{2} \\
            \makebox[1em]{4} & \makebox[1em]{1} & \makebox[1em]{0} & \makebox[1em]{5} \\
            \makebox[1em]{1} & \makebox[1em]{5} & \makebox[1em]{3} & \makebox[1em]{1} \\
            \makebox[1em]{2} & \makebox[1em]{0} & \makebox[1em]{6} & \makebox[1em]{2}
            \end{bmatrix}$
        };

        \node (B) at (0, 0) {
            $\frac{1}{10}
            \begin{bmatrix}
            \makebox[1em]{10} & \makebox[1em]{10} & \makebox[1em]{10} & \makebox[1em]{10} \\
            \makebox[1em]{7} & \makebox[1em]{6} & \makebox[1em]{9} & \makebox[1em]{8} \\
            \makebox[1em]{3} & \makebox[1em]{5} & \makebox[1em]{9} & \makebox[1em]{3} \\
            \makebox[1em]{2} & \makebox[1em]{0} & \makebox[1em]{6} & \makebox[1em]{2}
            \end{bmatrix}$
        };

        \node (C) at (4, 0) {
            $\frac{1}{10}
            \begin{bmatrix}
            \makebox[1em]{10} & \makebox[1em]{10} & \makebox[1em]{10} & \makebox[1em]{10} \\
            \makebox[1em]{9} & \makebox[1em]{8} & \makebox[1em]{7} & \makebox[1em]{6} \\
            \makebox[1em]{9} & \makebox[1em]{5} & \makebox[1em]{3} & \makebox[1em]{3} \\
            \makebox[1em]{6} & \makebox[1em]{2} & \makebox[1em]{2} & \makebox[1em]{0}
            \end{bmatrix}$
        };

        \node (D) at (4, 2.2) {
            $\frac{1}{40}
            \begin{bmatrix}
            \makebox[1em]{1} & \makebox[1em]{2} & \makebox[1em]{3} & \makebox[1em]{4} \\
            \makebox[1em]{0} & \makebox[1em]{3} & \makebox[1em]{4} & \makebox[1em]{3} \\
            \makebox[1em]{3} & \makebox[1em]{3} & \makebox[1em]{1} & \makebox[1em]{3} \\
            \makebox[1em]{6} & \makebox[1em]{2} & \makebox[1em]{2} & \makebox[1em]{0}
            \end{bmatrix}$
        };

\draw[->, thick] (A) -- (B); 
\draw[->, thick] (B) -- (C); 
\draw[->, thick] (C) -- (D); 

\end{tikzpicture}
\caption{Example for the construction of the bivariate SI random vector \((q_{Y;X}^{\uparrow U}(V),U)\) in Theorem \ref{lemtrafSI}\,\ref{lemtrafSI1} when a discrete bivariate random vector \((Y,X)\) maps into a grid \(\{a_1,\ldots,a_4\}\times \{b_1,\ldots,b_4\}\): The top left matrix illustrates the mass distribution of \((Y,X)\); the bottom left matrix describes the associated conditional distribution functions \(\eta_{Y|X}^v\,;\) the bottom right matrix is the rearranged conditional distribution function \(v \mapsto F_u(q_Y(v))\) which is defined in \eqref{definreacdf} and decreases in \(u\); the top right matrix describes the mass distribution of the SI random vector \((q_{Y;X}^{\uparrow U}(V),U)\) defined by \eqref{defyu}.}
    \label{fig:enter-label}
\end{figure}

Using the rearranged quantile transform constructed above, the following result shows that every random vector \((Y,\XX)\) can be reduced to a \(\preccurlyeq_{ccx}\)-equivalent bivariate SI random vector.

\begin{theorem}[Dimension reduction to bivariate SI distributions]\label{lemtrafSI}~\\
Let \(U,V\sim U(0,1)\) be independent. Then the rearranged quantile transform defined by \eqref{defyu} has the following properties:
    \begin{enumerate}[(i)]
    \item \label{lemtrafSI0} \(q_{Y;\XX}^{\uparrow U}(V)\eqd Y\,,\)
        \item \label{lemtrafSI1} \((q_{Y;\XX}^{\uparrow U}(V),U)\) is a bivariate SI random vector,
        \item \label{lemtrafSI2} \((q_{Y;\XX}^{\uparrow U}(V), U) =_{ccx} (Y,\XX)\,.\) 
    \end{enumerate}
\end{theorem}

\begin{remark}\label{remtrafSI}
        Due to Theorem \ref{lemtrafSI}, the rearranged quantile transform relates a \((1+p)\)-dimensional random vector \((Y,\XX)\) to a bivariate SI random vector such that the dependence information in the sense of \(\preccurlyeq_{ccx}\) remains invariant. Consequently, the strength of functional dependence is preserved in the sense that all dependence measures which are monotone in \(\preccurlyeq_{ccx}\) are invariant under this transformation. 
        The reduced random vector in \eqref{eqredrv} is based on this dimension-reduction principle. Note that our transformation extends the rearrangement of bivariate copulas in \cite[Proposition 3.17]{Ansari-2021} and \cite[Theorem 2.3]{strothmann2022} to arbitrary marginal distributions and to any multivariate random vector \(\XX\).
\end{remark}


For a bivariate random vector with finite support, the construction of the rearranged quantile transform in \eqref{defyu} is explained in Figure \ref{fig:enter-label}.
For a multivariate normal random vector, we perform the dimension reduction and the construction of the bivariate SI random vector \((q_{Y;\XX}^{\uparrow U}(V),U)\) in the following example.
Note that we use the basic ideas from this example to study in Section \ref{secsufdimred} sufficient conditions for verifying \(\preccurlyeq_{ccx}\).

\begin{example}[Dimension reduction for multivariate normal distribution]\label{exdimrednormal}~\\
    Let \((Y,\XX)\sim N(\mathbf{0},\Sigma)\) be a \((1+p)\)-dimensional normal random vector. In the sequel, we construct the bivariate SI random vector \((q_{Y;\XX}^{\uparrow U}(V),U)\) in Theorem \ref{lemtrafSI} based on the dimenion reduction principle in \eqref{definreacdf} and \eqref{defyu}. Consider the decomposition of the covariance matrix into 
    \begin{align}\label{decompSigma}
        \Sigma =  \begin{pmatrix}
         \sigma_{Y}^2 & \Sigma_{\XX,Y} \\ \Sigma_{Y,\XX} & \Sigma_\XX
    \end{pmatrix}.
    \end{align}
    Define \(S := A \XX\) for \(A:= \Sigma_{Y,\XX}\Sigma_{\XX}^- / \sqrt{\Sigma_{Y,\XX}\Sigma_{\XX}^- \Sigma_{\XX,Y}}\,,\) where \(\Sigma_\XX^-\) denotes a generalized inverse of \(\Sigma_\XX\) such as the Moore-Penrose  inverse. Key for the dimension reduction principle and for the construction of the rearranged quantile transform \(q_{Y;\XX}^{\uparrow u}\) in \eqref{defyu} is the identity
    \begin{align}\label{eqexdimrednormal1}
        (Y\mid \XX = \xx) \eqd (Y\mid S = A\xx) \quad \text{for } P^\XX-\text{almost all } \xx\in \R^p\,;
    \end{align}
    see \cite[Proof of Proposition 2.7]{ansari2023MFOCI}. 
    For \(k:= \rank(\Sigma_\XX)\,,\)
    let \(B\) be a \(p\times k\) matrix such that \(\Sigma_\XX = BB^T.\) Then the transformed conditional distribution function \(\eta_{Y|\XX}^v\) in \eqref{desetay} is given by
    \begin{align}\label{eqexdimrednormal2}
    \begin{split}
        \eta_{Y|\XX}^v(\uu) &= P(Y\leq q_Y(v) \mid \XX = B {\bf \Phi}^{-1}(\uu))\\
        &= P(Y\leq \sigma_Y \Phi^{-1}(v) \mid S = AB {\bf \Phi}^{-1}(\uu)) \\
        &= P(Y\leq \sigma_Y \Phi^{-1}(v)\mid S = \Phi^{-1}(u))
    \end{split}
    \end{align}
    for \(u = \Phi(AB{\bf \Phi}^{-1}(\uu))\,,\)
    where \({\bf \Phi}^{-1}(\uu)\) denotes a column vector with the standard normal quantile function applied to each component of \(\uu\in(0,1)^k\,.\)
Note that \((Y,S)\) is bivariate normal with correlation
\begin{align}\label{eqexdimrednormal2ab}
    \Cor(Y,S) = \frac{\Cov(Y,A\XX)}{\sqrt{\Var(Y)\Var(A\XX)}} = \sqrt{\frac{\Sigma_{Y,\XX} \Sigma_\XX^- \Sigma_{\XX,Y}}{\sigma_Y^2}} \geq 0,
\end{align}
 which implies that \(Y\) is SI in \(S\,;\) see e.g. \cite{Rueschendorf-1981}.
Let \(\UU\) be a random vector that is uniform on \((0,1)^k\) and define \(U := \Phi(S) = \Phi(AB{\bf \Phi}^{-1}(\UU))\). Then \(U\) is uniform on \((0,1)\) because \(S\) is standard normal.
The decreasing rearrangement of \(\eta_{Y|\XX}^v\) in \eqref{eqexdimrednormal2} is given by
\begin{align}\label{eqexdimrednormal2b}
    (\eta_{Y|\XX}^v)^*(u) = P( Y\leq \sigma_Y \Phi^{-1}(v)\mid S = \Phi^{-1}(u))
\end{align}
since the right-hand side of \eqref{eqexdimrednormal2b} is decreasing in \(u\) and
\begin{align}\label{eqexdimrednormal3}
\begin{split}
    \lambda(\{u \mid (\eta_{Y|\XX}^v)^*(u) \geq z\}) 
    &= P(\{\omega \mid (\eta_{Y|\XX}^v)^*(U(\omega)) \geq z\}) \\
    &= P(\{\omega \mid \eta_{Y|\XX}^v(\UU(\omega)) \geq z\}) =  \lambda^p(\{\uu \mid \eta_{Y|\XX}^v(\uu) \geq z\}) 
\end{split}
\end{align}
for all \(z\in \R\,.\) The second equality in \eqref{eqexdimrednormal3} follows from \eqref{eqexdimrednormal2}. 
For the first equality in \eqref{eqexdimrednormal3}, we use that \(U\) is uniform on \((0,1)\); for the third equality we use that \(\UU\) is uniform on \((0,1)^k\,.\)
For fixed \(u\in (0,1)\,,\) the rearranged distribution function \(F_u\) in \eqref{definreacdf} is given by
\begin{align}\label{eqexdimrednormal4}
    F_u(y) =  P( Y\leq y \mid S = \Phi^{-1}(u)) =  (\eta_{Y|\XX}^v)^*(u)
\end{align}
for \(y = q_Y(v) = \sigma_Y \Phi^{-1}(v)\,.\)
Hence, the rearranged quantile transform defined in \eqref{defyu} is
\begin{align}
    q_{Y;\XX}^{\uparrow u}(t) = F_u^{-1}(t) = F_{Y|S=\Phi^{-1}(u)}^{-1}(t)\,, \quad t\in (0,1)\,.
\end{align}
Finally, the bivariate SI random vector in Theorem \ref{lemtrafSI}\,\ref{lemtrafSI1} is given by
\begin{align}\label{eqexdimrednormal6}
    (q_{Y;\XX}^{\uparrow U}(V),U) = (F_{Y|S=\Phi^{-1}(U)}^{-1}(V),U).
\end{align}
Note that the left-hand side in \eqref{eqexdimrednormal6} satisfies \((q_{Y;\XX}^{\uparrow U}(V),U) =_{ccx} (Y,\XX)\), which follows from \eqref{eqexdimrednormal2} and \eqref{eqexdimrednormal4} and verifies Theorem \ref{lemtrafSI}\,\ref{lemtrafSI2}.  
Due to \eqref{eqexdimrednormal1}, we know that \((Y,\XX) =_{ccx}  (Y,S)\). 
Hence, the reduced random vectors in Theorem \ref{propcharS} \ref{propcharS2} satisfy 
\begin{align}\label{eqexdimrednormal7}
    (F_Y\circ q_{Y;\XX}^{\uparrow U}(V),U) \eqd (F_Y\circ Y,\Phi(S))
\end{align}
in the case of the multivariate normal distribution.
\end{example}

\section{Verifying \(\preccurlyeq_{ccx}\)}\label{sec52}

In this section, we provide sufficient conditions for verifying the conditional convex order. These conditions are based on the dimension reduction principle from Section \ref{secdimredprin} and on stochastic monotonicity assumptions which simplify a verification of the Schur order criterion in Theorem \ref{thecharSchur}. 
In the first part of this section, we characterize the \(\preccurlyeq_{ccx}\)-order for multivariate normal distributions.
In the second part, we verify the conditional convex order for various copula-based models that generalize the additive error model in \eqref{eqadderrmod}.



\subsection{Multivariate normal distribution}\label{secmultnorm}

The following result characterizes \(\preccurlyeq_{ccx}\) for the multivariate normal distribution. Its proof is based on the dimension reduction performed in Example \ref{exdimrednormal}. 

\begin{theorem}[Conditional convex order for multivariate normal distribution]\label{propschurnormal}~\\
    Let \((Y,\XX)\sim N(\boldsymbol{\mu},\Sigma)\) and \((Y',\XX')\sim N(\boldsymbol{\mu}',\Sigma')\) be multivariate normal random vectors with non-degenerate components \(Y\) and \(Y'.\) For \(\Sigma\) and \(\Sigma'\) decomposed as in \eqref{decompSigma}, we have 
    \begin{align}\label{propschurnormal1}
        (Y,\XX) \preccurlyeq_{ccx} (Y',\XX') \quad \Longleftrightarrow \quad  \,\Sigma_{Y,\XX}\Sigma_\XX^- \Sigma_{\XX,Y} / \sigma_{Y}^2 \,
        \leq  \, \Sigma_{Y',\XX'}\Sigma_{\XX'}^- \Sigma_{\XX',Y'} / \sigma_{Y'}^2.
    \end{align}
\end{theorem}

\begin{remark}
\begin{enumerate}[(a)]\label{remnorm}
\item 
Due to Theorem \ref{propschurnormal}, multivariate normal random vectors \((Y,\XX)\) and \((Y',\XX')\) are always comparable in \(\preccurlyeq_{ccx}\), provided \(\sigma_Y\) and \(\sigma_{Y'}\) are positive. It is neither assumed that the random vectors have the same dimension nor that the covariance matrices are positive definite. No conditions are imposed on \(\boldsymbol{\mu}\) and \(\boldsymbol{\mu}'\) either.
 Interestingly, by \cite[Proposition 2.7]{ansari2023MFOCI}, either condition in \eqref{propschurnormal1} is equivalent to \(\xi(Y,\XX)\leq \xi(Y',\XX')\).
\item \label{remnorm2} In Theorem \ref{propschurnormal}, if \(\sigma_Y = \sigma_{Y'}\), the conditional convex order reduces to comparing the Schur complements of the covariance matrices.
Since \(\E[Y\mid \XX = \xx] = \Sigma_{Y,\XX}\Sigma_\XX^- \xx\), \(\preccurlyeq_{ccx}\) in \eqref{propschurnormal1} also underlies the fraction of explained variance \(\Var(\E[Y\mid \XX])/\Var(Y) = \Sigma_{Y,\XX}\Sigma_\XX^- \Sigma_{\XX,Y} / \sigma_{Y}^2\) and the kernel partial correlation for a linear kernel \cite[Remark 18]{deb2020b}.
\end{enumerate}
\end{remark}

In Example \ref{exmultnormsch} below, we verify the ordering condition \eqref{propschurnormal1} for the equicorrelated normal distribution and the multivariate normal distribution with independent predictor variables.

\subsection{Sufficient dimension reduction}\label{secsufdimred}

To derive sufficient conditions for verifying \(\preccurlyeq_{ccx}\),
we generalize the dimension reduction principle discussed for the multivariate normal distribution in Example \ref{exdimrednormal}. To this end, we extend the identity in \eqref{eqexdimrednormal1} to general transformations and denote a function \(g\colon \R^p\to \R\) as \emph{sufficient} for (the conditional distribution) \(Y|\XX\) if 
\begin{align}\label{eqsuffdimred}
    (Y\mid \XX=\xx) \eqd (Y \mid g(\XX) = g(\xx)) \quad \text{for } P^\XX\text{-almost all } \xx\in \R^p\,.
\end{align}
The terminology of \(g\) being sufficient for \(Y|\XX\) coincides with the classical concept of a sufficient statistic \(g\) for the family of conditional distributions \(\{P^{Y|\XX=\xx} \mid \xx\in \R^p\}\).
For a linear function \(g,\) the concept in \eqref{eqsuffdimred} is also referred to as \emph{sufficient for dimension reduction} in \cite{Chiaromonte-2002,Li-1991,Li-2007,Li-2009}.
An example of a sufficient function for \(Y| \XX\) is the function \(f\) in the additive error model \eqref{eqadderrmod} or the linear transformation \(g(\xx) = A \xx\) in Example \ref{exdimrednormal}. Note that \eqref{eqsuffdimred} is equivalent to conditional independence of \(Y\) and \(\XX\) given \(g(\XX)\); see \cite[Proposition 6.6]{Kallenberg-2002}.

In the following theorem, we provide simple conditions for verifying \(\preccurlyeq_{ccx}\) under stochastic monotonicity assumptions.
We write \(Y\uparrow_{st} Z\) if \(Y\) is SI in \(Z\), and similarly, \(Y\downarrow_{st} Z\) if \(Y\) is stochastically decreasing in \(Z\) (i.e., \(-Y\uparrow_{st} Z\)).



\begin{theorem}[\(\preccurlyeq_{ccx}\)-order for copula-based models]\label{propmvSchur}~\\
    Assume that \(g\colon \R^p\to \R\) is sufficient for \(Y\mid \XX\) and that \(h\colon \R^{p'}\to \R\) is sufficient for \(Y'\mid \XX'\,.\) Then the following statements hold true:
    \begin{enumerate}[(i)]
        \item \label{propmvSchur1} If \((Y, g(\XX))\preccurlyeq_{ccx} (Y', h(\XX'))\,,\) then \((Y, \XX)\preccurlyeq_{ccx} (Y',\XX')\,.\)
        \item \label{propmvSchur2} Assume that \(\overline{\Ran(F_Y)} = \overline{\Ran(F_{Y'})}\,,\) 
        \(Y\uparrow_{st} g(\XX)\), and \(Y'\uparrow_{st} h(\XX')\,.\) Then \(C_{Y,g(\XX)}\leq_{c} C_{Y',h(\XX')}\) implies \((Y,\XX)\preccurlyeq_{ccx} (Y',\XX')\,.\)
        \item \label{propmvSchur3} Assume that \(\overline{\Ran(F_Y)} = \overline{\Ran(F_{Y'})}\,,\) 
        \(Y\downarrow_{st} g(\XX)\), and \(Y'\downarrow_{st} h(\XX')\,.\) Then \(C_{Y,g(\XX)}\leq_{c} C_{Y',h(\XX')}\) implies \((Y,\XX)\preccurlyeq_{ccx} (Y',\XX')\,.\)
    \end{enumerate}
\end{theorem}

\begin{remark}\label{remverccx}
\begin{enumerate}[(a)]
\item 
The condition \(Y\uparrow_{st} g(\XX)\) is satisfied if there exists a representation \(Y = f(g(\XX),\varepsilon)\) a.s. for some componentwise increasing function \(f\) and an error \(\varepsilon\) that is independent of \(g(\XX)\); see e.g. \cite[Lemma 3.1]{Muller-2001}. Without loss of generality, \(\varepsilon\) may be uniform on \((0,1)\) or standard normal. The additive error model \eqref{eqadderrmod} admits such a representation. Note that the proof of Theorem \ref{theadderrmod} applies Theorem \ref{propmvSchur}\,\ref{propmvSchur1} but not the criteria on pointwise comparison of copulas in parts \ref{propmvSchur2} and \ref{propmvSchur3} because the distributions of \(g(\XX)\) and \(\varepsilon\) in the additive error model are not given and thus the copula \(C_{Y,g(\XX)}\) is not known.
\item The stochastic monotonicity assumptions in Theorem \ref{propmvSchur} facilitate the verification of the conditional convex order because, for bivariate SI random vectors from the same Fr\'{e}chet class, \(\preccurlyeq_{ccx}\) and \(\leq_c\) are equivalent; see Proposition \ref{rembivSchur}. For the latter concordance order, numerous comparison results are available for parametric families of distributions; see Example \ref{excop} below. Note that Proposition \ref{Prop.Bernoulli} on Bernoulli distributions does not require stochastic monotonicity assumptions as in Theorem \ref{propmvSchur}.
\end{enumerate}
\end{remark}

In the next example, we use Theorem \ref{propmvSchur} 
to verify for two specific families of multivariate normal distributions the \(\preccurlyeq_{ccx}\)-ordering criterion in \eqref{propschurnormal1}.

\begin{example}[Multivariate normal distribution]\label{exmultnormsch}~\\
Assume that \((Y,\XX)\sim N(\mathbf{0},\Sigma)\) follows a \((1+p)\)-dimensional normal distribution.
\begin{enumerate}[(a)]
\item Consider the case where the covariance matrix \(\Sigma = (\sigma_{ij})\) is equicorrelated with \(\sigma_{ij}=1\) for \(i=j\) and \(\sigma_{ij}=\varrho\) else for some \(\varrho\in [-1/p,1]\,.\)
    Standard calculations show that 
    \begin{align}\label{eqexmultnormsch1}
        (Y\mid \XX = \xx) \sim N\Big(\frac{\varrho}{1+(p-1)\varrho} \sum x_i\,, ~ 1-\frac{p\varrho^2}{1+(p-1)\varrho}\Big)\,.
    \end{align}
    The conditional distribution in \eqref{eqexmultnormsch1} depends only on \(g(\xx)=\sum x_i\,,\) so \(g\) is sufficient for \(Y|\XX\,.\) The random vector \((Y,\sum X_i)\) follows a bivariate normal distribution with zero mean and covariance matrix
    \begin{align*}
        \Sigma = \begin{pmatrix}
             1 & p\varrho\\ p\varrho & p(1+(p-1)\varrho)
        \end{pmatrix}\,.
    \end{align*}
    The copula of \((Y,\sum X_i)\) is the Gaussian copula with correlation \(\varrho\,.\) Since the Gaussian copula family is \(\leq_{c}\)-increasing in its parameter and since \(Y\uparrow_{st} \sum X_i\) for \(\varrho\geq 0\), we obtain from Theorem \ref{propmvSchur}\,\ref{propmvSchur2} that \((Y,\XX)\) is increasing in \(\varrho\) with respect to the conditional convex order.
Using radial symmetry of the bivariate normal distribution (i.e., \((-Y,-\sum X_i) \eqd (Y,\sum X_i)\)), using \(Y\downarrow_{st} \sum X_i\) for \(\varrho\leq 0\), and applying Theorem \ref{propmvSchur}\,\ref{propmvSchur3}, we obtain that \((Y,\XX)\) is increasing eventually in \(|\varrho|\) with respect to \(\preccurlyeq_{ccx}\). Note that, by \eqref{propschurnormal1}, this means that \(\Sigma_{Y,\XX}\Sigma_\XX^-\Sigma_{\XX,Y}\) is increasing in \(|\varrho|\).
    \item \label{exmultnormsch1} Consider the case where \(\Var(Y) = 1\), the components of \(\XX = (X_1,\ldots,X_p)\) are independent, and \(\Cor(Y,X_i) = \varrho\) for all \(i\). 
    Then the covariance matrix \(\Sigma\) is positive semi-definite if and only if \(\varrho\in [-p^{-1/2},p^{-1/2}]\).
    Again, the conditional distribution 
\begin{align}\label{eqexmultnormsch}
    (Y \mid \XX = \xx) \sim N\big(\varrho \sum x_i\,,~ 1-p\varrho^2\big)\,,
\end{align}
depends only on \(g(\xx)=\sum x_i\,.\) 
Straightforward calculations show that 
\begin{align*}
    \big(Y,\sum X_i\big)\sim N\big(0,\left(\begin{smallmatrix}
    1 & p\varrho \\ p\varrho & p\end{smallmatrix}\right)\big)\,.
\end{align*}
Similar as above, 
we obtain from Theorem \ref{propmvSchur}\,\ref{propmvSchur2} and \ref{propmvSchur3} that 
\((Y,\XX)\) as well as the Schur complements \(\Sigma_{Y,\XX}\Sigma_\XX^-\Sigma_{\XX,Y}\) are increasing in \(|\varrho|\) with respect to \(\preccurlyeq_{ccx}\).
\end{enumerate}
\end{example}


In the following example, we outline various well known SI copula families that are increasing in the concordance order. This yields families of models \(Y = f(g(\XX),\varepsilon)\) that fulfill the assumptions of Theorem \ref{propmvSchur}\,\ref{propmvSchur2} and thus are \(\preccurlyeq_{ccx}\)-comparable; see Remark \ref{remverccx}. Symmetry arguments yield similar results for stochastically decreasing copulas.

\begin{example}[Copula-based models]\label{excop}
    Let \(Y = f(g(\XX),\varepsilon)\) with copula \(C_1 = C_{Y,g(\XX)}\) and \(Y' = f(g(\XX'),\varepsilon')\) with copula \(C_2 = C_{Y',g(\XX')}\). Assume that \(\overline{\Ran(F_Y)} = \overline{\Ran(F_{Y'})}\).
    Then, we can verify \((Y,\XX)\preccurlyeq_{ccx} (Y',\XX')\) by Theorem \ref{propmvSchur}\,\ref{propmvSchur2} for Archimedean, extreme-value, and elliptical copulas \(C_1\) and \(C_2\) as follows.
    \begin{enumerate}[(a)]
        \item Assume that \(C_i\) is an \emph{Archimedean copula}, i.e., \(C_i\) admits a representation \(C_i(u,v) = \psi_i(\varphi_i(u),\varphi_i(v))\) for some decreasing convex function \(\varphi_i\colon [0,1]\to [0,\infty)\) with \(\varphi(1) = 0\), where \(\psi_i = \varphi_i^{[-1]}\) denotes the (pseudo-)inverse of \(\varphi_i\). If \(-\psi_i'\) is log-convex, then \(C_i\) is SI \cite[Theorem 2.8]{Mueller-2005}; further, if \(\varphi_1\circ \psi_2\) is subadditive, then \(C_1\leq_c C_2\) \cite[Theorem 4.4.2]{Nelsen-2006}. 
        Various Archimedean copula families that meet these conditions are given in \cite[Table 3]{Ansari-Rockel-2023}.
        \item Assume that \(C_i\) is an \emph{extreme-value copula}, i.e., \(C_i\) admits a representation \linebreak\(C_i(u,v) = (\exp(\log(uv)) A_i(\log(v)/(\log(u) + \log(v))\) where the \emph{Pickands dependence function} \(A_i\) is convex and satisfies the constraints \(\max\{t,1-t\}\leq A_i(t) \leq 1\) for all \(t\in [0,1]\). Extreme-value copulas are always SI; moreover, \(A_1(t)\geq A_2(t)\) for all \(t\in (0,1)\) implies \(C_1\leq_c C_2\); see \cite[Theorem 3.4]{Ansari-Rockel-2023}. Various well known extreme-value copula families that meet the conditions are given in \cite[Table 4]{Ansari-Rockel-2023}.
        \item Assume that \(C_i\) is an \emph{elliptical copula}, i.e., it is the copula (implicitly obtained by Sklar's theorem) of an elliptically distributed bivariate random vector \(\ZZ_i\eqd R A_i\UU\) for a non-negative random variable \(R\) that is independent of the bivariate random vector \(\UU\) which is uniformly distributed on the \(2\)-sphere. Here, \(A_i\) is a \(2\times 2\)-matrix such that \(\Sigma_i := A_iA_i^T = \begin{pmatrix}
            1 & \rho_i \\ \rho_i & 1
        \end{pmatrix}\) for \(\rho_i\in [-1,1]\). Sufficient conditions on \(R_i\) and \(\rho_i\) for \(C_i\) being SI are given in \cite[Proposition 1.2]{Abdous-2005}. Further, if \(\rho_1\leq \rho_2\) then \(C_1\leq_c C_2\). Elliptical SI copula families that meet these conditions are given in \cite[Table 5]{Ansari-Rockel-2023}.
    \end{enumerate}
\end{example}

\section{Dependence measures generated by \(\preccurlyeq_{ccx}\)}\label{secdepmea}

In this section, we use the conditional convex order to construct dependence measures based on conditional distribution functions, Wasserstein distances, and monotone rearrangements. These three classes correspond to the three equivalences in Theorem \ref{propcharS}.
The resulting dependence measures inherit information monotonicity and invariance properties from \(\preccurlyeq_{ccx}\). 
The zero-independence and max-functionality property require some regularity conditions based on convexity; see \cite{Ansari-LFT-2023,Figalli-2025}. We further provide simple conditions under which these dependence measures characterize conditional independence. 
As noted in Remark \ref{remnorm}\,\ref{remnorm2}, \(\preccurlyeq_{ccx}\) also underlies the kernel partial correlation \cite{deb2020b} for a linear kernel.

Throughout, we assume \(Y\) to be non-degenerate so that the proposed dependence measures attain at least two distinct values; if \(Y\) were degenerate, it would be both independent of \(\XX\) and perfectly dependent on \(\XX\), contradicting zero-independence and max-functionality.


\subsection{Dependence measures based on conditional distribution functions}

As a first class of dependence measures, we consider functionals of the form
\begin{align}
 \nonumber   \xi_\varphi(Y,\XX) :&= \alpha_\varphi^{-1} \int_{\R^{p}} \int_{\R} \varphi(F_{Y|\XX=\xx}(y)-F_{Y}(y)) \de P^Y(y) \de P^{\XX}(\xx)\quad \text{and}\\
 \label{eq:Gamma.delta.phi} 
    \Lambda_\varphi(Y,\XX) :&= \beta_\varphi^{-1} \int_{\R^{2p}}\int_{\R} \varphi(F_{Y|\XX=\xx_1}(y)-F_{Y|\XX=\xx_2}(y)) \de P^Y(y) \de (P^{\XX}\otimes P^{\XX})(\xx_1,\xx_2)
\end{align}
with normalizing constants 
\(\alpha_\varphi  := \int_{\R} \int_\R \varphi\left(\1_{\{y_1\leq y\}} - F_{Y}(y)\right) \de P^Y(y) \de P^Y (y_1)\) and \linebreak \(\beta_\varphi
   := \int_{\R^{2}}\int_\R  \varphi\left(\1_{\{y_1\leq y\}} - \1_{\{y_2\leq y\}}\right) \de P^Y(y) \de (P^Y\otimes P^Y)(y_1,y_2) \,.\)
   While \(\xi_\varphi\) compares conditional distribution functions with unconditional ones, \(\Lambda_\varphi\) measures the sensitivity of conditional distribution functions in the conditioning variable.
   Note that, for \(\varphi(x) = x^2\), the functionals \(\xi_\varphi\) and \(\Lambda_\varphi\) reduce to Chatterjee's rank correlation \(\xi\) defined in \eqref{Tuniv}.

By the following result, which is based on the Hardy–Littlewood–Polya theorem (see Lemma \ref{lemhlpt}), convexity conditions on \(\varphi\) imply that \(\xi_\varphi\) and \(\Lambda_\varphi\) are dependence measure that are \(\preccurlyeq_{ccx}\)-increasing and characterize (conditional) independence and perfect dependence.

\begin{theorem}[Dependence measures generated by convexity]\label{theconsccx}~\\
    Assume that \(\varphi: [-1,1] \rightarrow \R\) is convex and strictly convex at \(0\) with \(\varphi(0)=0\). 
    Then \((Y,\XX)\preccurlyeq_{ccx} (Y',\XX')\) implies 
\begin{align}\label{theconsccx0}
    \xi_\varphi(Y,\XX) \leq \xi_\varphi(Y',\XX')
    \qquad \textrm{ and } \qquad 
    \Lambda_\varphi(Y,\XX) \leq \Lambda_\varphi(Y',\XX')\,.
\end{align}
Further, the following statements hold true:
\begin{enumerate}[(i)]
    \item\label{theconsccx1} \(\xi_\varphi\) and \(\Lambda_\varphi\) attain values only in \([0,1]\).
    \item\label{theconsccx2} \(\xi_\varphi\) and \(\Lambda_\varphi\) have the zero-independence and max-functionality property.
    \item \(\xi_\varphi\) and \(\Lambda_\varphi\) satisfy information monotonicity, i.e. \(\xi_\varphi(Y,\XX) \leq \xi_\varphi(Y,(\XX,\ZZ))\) and \(\Lambda_\varphi(Y,\XX) \leq \Lambda_\varphi(Y,(\XX,\ZZ))\).
    \item\label{theconsccx3} If \(\varphi\) is strictly convex on \([-1,1]\), then \(\Lambda_\varphi(Y,\XX) = \Lambda_\varphi(Y,(\XX,\ZZ))\) if and only if \(\xi_\varphi(Y,\XX) = \xi_\varphi(Y,(\XX,\ZZ))\) if and only if \(Y\) and \(\ZZ\) are conditionally independent given \(\XX\).
\end{enumerate}
\end{theorem}

\begin{remark} Suppose the assumptions of Theorem \ref{theconsccx}.
\begin{enumerate}[(a)]
    \item The normalizing constants \(\alpha_\varphi\) and \(\beta_\varphi\) in \eqref{eq:Gamma.delta.phi} are positive. Note that the marginal constraint implies that also \(Y'\) is non-degenerate.
    \item \(\xi_\varphi\) and \(\Lambda_\varphi\) are dependence measures that inherit all invariance properties from \(\preccurlyeq_{ccx}\), i.e. \(\xi_\varphi(Y,\XX) = \xi_\varphi(g(Y),\mathbf{h}(\XX))\) for all strictly increasing functions \(g\) and bijective functions \(\mathbf{h}\), and \(\xi_\varphi(Y,\XX) = \xi_\varphi(F_Y(Y),\mathbf{F}_\XX(\XX))\), similarly for \(\Lambda_\varphi\).
    \item Strict convexity of \(\varphi\) at \(\varphi(0) = 0\) is sufficient for the values \(0\) and \(1\) to characterize independence and perfect dependence, respectively.
    However, strict convexity at \(0\) is not sufficient for \(\xi_\varphi\) to characterize conditional independence.
    As a counterexample, consider \(\varphi\colon t \mapsto |t|\) which is convex and strictly convex at \(\varphi(0) = 0\).
    In this case, \(\xi_\varphi(Y,\XX) = \xi_\varphi(Y,(\XX,\ZZ))\) does not imply conditional independence of \(Y\) and \(\ZZ\) given \(\XX\); see \cite[Example 5.10]{fgwt2021}.
    
\end{enumerate}
\end{remark}

\begin{remark}\label{remaza}
\begin{enumerate}[(a)]
    \item The zero-independence and max-functionality property as well as information monotonicity for \(\xi_\varphi\) can be traced back to convexity properties for measures of statistical association as studied in \cite{Figalli-2025}.  In Theorem \ref{theconsccx}, we also provide a characterization of \emph{conditional} independence and show that \(\xi_\varphi\) is \(\preccurlyeq_{ccx}\)-increasing. 
    \item \label{remaza2} A variant of Chatterjee's rank correlation is the \emph{integrated \(R^2\)} in \cite{Azadkia-2025}, defined by 
    \begin{align*}
    \nu(Y,\XX)
	:= \int_{\mathbb{R}} \frac{\Var (P(Y \leq y \, | \, \XX))}{\Var (\mathds{1}_{\{Y \leq y\}})} \; \mathrm{d} \tilde{\mu}(y)
    \quad \textrm{for } 
    \tilde{\mu}(B) := \frac{\mu(B\cap \tilde{S})}{\mu(\tilde{S})}\,,
    \end{align*} 
    where \(\mu\) is the law of \(Y\) having support \(S\), and where \(\tilde{S} := S\setminus \{s_{\max}\}\) if \(S\) attains a maximum \(s_{\max}\), and \(\tilde{S}:= S\) otherwise.
    Note that, in comparison to \(\xi\) in \eqref{Tuniv}, the integral is outside the fraction. 
    Using the representation 
    \begin{align*}
    \nu(Y,\XX)
	   = \int_{q_Y^{-1}(\tilde{S})} 
        \frac{\Var (P(Y \leq q_Y(v) \, | \, \XX))}{\Var (\mathds{1}_{\{Y \leq q_Y(v)\}}) \, (1-P^Y(\{s_{\rm max}\}))} \; \mathrm{d} \lambda(v)
    \end{align*}
    and the version of \(\preccurlyeq_{ccx}\) for cumulative probabilities (see Remark \ref{remccx}\,\ref{remccx2}), it follows that also \(\nu\) is increasing in the conditional convex order.
\end{enumerate}
\end{remark}

\subsection{Optimal transport-based dependence measures}

As a second class of dependence measures generated by \(\preccurlyeq_{ccx}\), we consider optimal transport-based Wasserstein correlations as studied in \cite{wiesel2022}.
Therefore, let \(\nu\) and \(\nu'\) be distributions on \(\R\) and \(c\colon \R\times \R \to \R\) be a continuous cost function satisying \(c(y,y')\geq a(y)+b(y')\) for some \(a\in L^1(\nu)\) and \(b\in L^1(\nu')\). Then Kantorovich's mass transportation problem 
consists of finding an optimal coupling between \(\nu\) and \(\nu'\) that attains the minimal costs 
\begin{align}\label{eqopttransprob}
    \cW_c(\nu,\nu') := \inf_{\gamma\in \Pi(\nu,\nu')} \int c(y,y') \de \gamma(y,y'),
\end{align}
where \(\Pi(\nu,\nu')\) is the set of couplings between \(\nu\) and \(\nu'\); see \cite{Villani-2009}. Now, the \emph{Wasserstein correlation coefficient} in \cite{wiesel2022} adapted to our setting is defined by
\begin{align}\label{defWcc}
    \dW_c(Y,\XX) := \frac{\int_{\R^p} \cW_c(\pi_\xx,\nu) \de \mu(\xx)}{\int_\R \int_\R c(y,y') \de \nu(y) \de \nu(y')}.
\end{align}
Here \(\nu\), \(\mu\), and \(\pi_\xx\) denote the distribution of \(F_Y(Y)\), \(\XX\), and \(F_Y(Y)\mid \XX = \xx\), respectively. 
For \emph{submodular}\footnote{A function \(f\colon \R^2 \to \R\) is said to be submodular if \(f(\xx) + f(\yy) \geq f(\xx\wedge \yy)+f(\xx\vee \yy)\) for all \(\xx,\yy\in \R^2\), 
where \(\wedge\) and \(\vee\) denote the componentwise minimum and maximum, respectively.} 
cost functions \(c\), it is well known that the comonotone coupling \((F_{\pi_\xx}^{-1}(V),F_\nu^{-1}(V))\) attains the minimal cost \(W_c(\pi_\xx,\nu)\) in the numerator of \eqref{defWcc}, i.e.
\begin{align}\label{eqrepW_c}
    \cW_c(\pi_\xx,\nu) = \E c( F_{F_Y(Y)|\XX=\xx}^{-1}(V), F_{F_Y(Y)}^{-1}(V));
\end{align}
see \cite[Theorem 3.1.2]{Rachev-Rueschendorf-1998}. Important examples of submodular cost functions are \emph{convex costs} of the type \(c(y,y') = h(y'-y)\) for \(h\) convex. These functions include the standard costs \(c(y,y') = |y'-y|^p\) for \(p\geq 1\). 
For the following result, we use the characterization \ref{propcharS3} in Theorem \ref{propcharS} to show that, for convex costs, the Wasserstein correlation \(\dW_c(Y,\XX)\) is \(\preccurlyeq_{ccx}\)-increasing and satisfies all the desired properties of a dependence measure that quantifies the strength of functional dependence.

\begin{theorem}[Optimal transport-based dependence measures]\label{theOT}~\\
    Assume a cost function \(c(y,y') = h(y'-y)\) where \(h\colon \R\to \R\) is convex and strictly convex at \(0 = h(0)\). 
    Then 
    \begin{align}\label{eqtheOT}
        (Y,\XX)\preccurlyeq_{ccx} (Y',\XX') \quad \text{implies} \quad \dW_c(Y,\XX) \leq \dW_c(Y',\XX').
    \end{align}
    Further, the following statements hold true.
    \begin{enumerate}[(i)]
        \item \label{theOT1} \(\dW_c\) attains values only in \([0,1]\).
        \item \label{theOT2} \(\dW_c\) has the zero-independence and max-functionality property.
        \item \label{theOT3} \(\dW_c\) satisfies information monotonicity, i.e. \(\dW_c(Y,\XX) \leq \dW_c(Y,(\XX,\ZZ))\).
        \item \label{theOT4} If \(h\) is strictly convex, then \(\dW_c(Y,\XX) = \dW_c(Y,(\XX,\ZZ))\) if and only if \(Y\) and \(\ZZ\) are conditionally independent given \(\XX\).
    \end{enumerate}
\end{theorem}

\begin{remark}
    \begin{enumerate}[(a)]
        \item The function \(h\) in Theorem \ref{theOT} can attain negative values. Hence, also the conditional optimal cost \(\cW_c(\pi_\xx,\nu)\) may be negative. However, by convexity of \(h\) and \(h(0) = 0\), the average cost \(\int_{\R^p} \cW_c(\mu_\xx,\nu) \de \mu(\xx)\) is non-negative.
        \item As for the dependence measures \(\xi_\varphi\) and \(\Lambda_\varphi\) in Theorem \ref{theconsccx}, the Wasserstein correlation \(\dW_c\) inherits invariance properties from from the conditional convex order. 
        More precisely, \(\dW_c(Y,\XX) = \dW_c(g(Y),{\bf h}(\XX))\) for all strictly increasing functions \(g\) and bijective functions \(\bf{h}\), and \(\dW_c(Y,\XX) = \dW_c(F_Y(Y),{\bf F}_\XX(\XX))\).
        \item The Wasserstein correlation in \cite{wiesel2022} is defined with respect to a metric cost function. Apart from the metric \(d(y,y') = |y'-y|\), metrics \(d\) on \(\R\) are generally not submodular (since metrics are subadditive). Consequently, the representation of \(\cW_c\) in terms of conditional comonotonicity in \eqref{eqrepW_c} typically does not apply for metric costs \(c=d\). 
    \end{enumerate}
\end{remark}

\subsection{Rearrangement-based dependence measures}\label{secrearrdepmea}

For a third class of dependence measures generated by \(\preccurlyeq_{ccx}\), we use the characterization of \(\preccurlyeq_{ccx}\) via the bivariate concordance order in Theorem \ref{propcharS}\,\ref{propcharS2}. 


Therefore, we denote by \(\cR^\uparrow := \{(T,S)\mid T,S\sim U(0,1), T\uparrow_{st} S \}\) the class of bivariate SI random vectors with marginals that are uniform on \((0,1)\). 
Let \(\mu\) be a functional on \(\cR^\uparrow\) that is normalized and increasing in the concordance order in the sense that \(\mu(T,S) = 0\) if and only if \(T\) and \(S\) are independent, \(\mu(T,S) = 1\) if and only if \(T\) and \(S\) are comonotone, and \((T,S)\leq_c (T',S')\) implies \(\mu(T,S)\leq \mu(T',S')\).
We say that \(\mu\) is \emph{strictly increasing} on \(\cR^\uparrow\) if \(\leq_c\) is strict on \(\cR^\uparrow\), i.e. \((T,S)\leq_c (T',S')\) with \((T,S)\not\eqd (T',S')\) implies \(\mu(T,S) < \mu(T',S')\). 
For \((Y,\XX)\in \cR_{[0,1]}\) (recall \eqref{defR_A} for the definition of this class)
and for \(V,U\) i.i.d. uniform on \((0,1)\), we define the \emph{rearranged dependence measure} \(\sR_\mu\) by
\begin{align}\label{deffunccon}
    \sR_\mu(Y,\XX) := \mu(F_Y\circ q_{Y;\XX}^{\uparrow U}(V),U),
\end{align}
where \((F_Y\circ q_{Y;\XX}^{\uparrow U}(V),U)\) is a reduced random vector associated with \((Y,\XX)\). 
Since \((Y,\XX)\in \cR_{[0,1]}\), we have that \(F_Y\) is continuous. Hence, 
the reduced random vector \((F_Y\circ q_{Y;\XX}^{\uparrow U}(V),U)\) has marginals that are uniform on \((0,1)\). Since it is also SI, it is in \(\cR^\uparrow\), and \(\sR_\mu(Y,\XX)\) in \eqref{deffunccon} is well-defined.
Recall the identity \((Y,\XX) =_{ccx} (F_Y\circ q_{Y;\XX}^{\uparrow U}(V),U)\) in \eqref{eqredrv}.
Hence, \(\sR_\mu\) depends on the strength of functional dependence of \(Y\) on \(\XX\) in the sense of the conditional convex order. Equivalently, by Theorem \ref{propcharS}, \(\sR_\mu\) depends on the strength of \emph{positive} dependence of \((F_Y\circ q_{Y;\XX}^{\uparrow U}(V),U)\) in the sense of the concordance order.


\begin{theorem}[Dependence measures generated by monotone rearrangements]\label{cordepmeas}
    For the functional \(\sR_\mu\) in \eqref{deffunccon}, we have that 
    \begin{align}\label{cordepmeas1}
        (Y,\XX)\preccurlyeq_{ccx} (Y',\XX') \quad \text{implies} \quad \sR_\mu(Y,\XX) \leq \sR_\mu(Y',\XX').
    \end{align}
    Further, the following statements hold true.
    \begin{enumerate}[(i)]
    \item \label{cordepmeas2a} \(\sR_\mu\) takes values in \([0,1]\).
        \item \label{cordepmeas2} \(\sR_\mu\) has the zero-independence and max-functionality property. 
        \item \label{cordepmeas4} \(\sR_\mu\) satisfies information monotonicty, i.e. \(\sR_\mu(Y,\XX) \leq \sR_\mu(Y,(\XX,\ZZ))\).
        \item \label{cordepmeas5} If, additionally, \(\mu\) is strictly increasing on \(\cR^\uparrow\), then \(\sR_\mu(Y,\XX) = \sR_\mu(Y,(\XX,\ZZ))\) if and only if \(Y\) and \(\ZZ\) are conditionally independent given \(\XX\).
    \end{enumerate}
\end{theorem}

Considering the functional form of \(\sR_\mu\) in \eqref{deffunccon}, we can extend Wasserstein correlations to non-linear functionals by defining
\begin{align}\label{defS_mu}
    \mathsf{S}_\mu(Y,\XX) = 1 - \mu(F_Y \circ F_{Y|\XX}^{-1}(V), V)
\end{align}
for \(V\sim U(0,1)\) independent of \(\XX\).
Then the following result is a direct consequence of Theorem \ref{propcharS} and Theorem \ref{cordepmeas}.

\begin{corollary}[Dependence measures generated by conditional comonotonicity]~\\
\(\mathsf{S}_\mu\) in \eqref{defS_mu} is a dependence measure that has the same properties as \(\sR_\mu\) in Theorem \ref{cordepmeas}.
\end{corollary}

\pagebreak

\begin{remark}
\begin{enumerate}[(a)]
    \item Our approach on rearranged dependence measures in Theorem \ref{cordepmeas} is studied in \cite{strothmann2022} for the case of bivariate random vectors with continuous marginal distribution functions. Various examples of measures of concordance that are increasing on \(\cR^\uparrow\), such as Kendall's \(\tau\) and Gini's \(\gamma\), or even strictly increasing, such as Spearman's \(\rho\) or the Schweizer-Wolff measure, are given in \cite[Section 2.2]{strothmann2022}. 
    \item For the functional \(\mathsf{S}_\mu\) in \eqref{defS_mu}, note that the bivariate random vector \((F_Y \circ F_{Y|\XX}^{-1}(V), V)\) is \(\leq_c\)-decreasing in \(\preccurlyeq_{ccx}\) by equivalence of Theorem \ref{propcharS}\,\ref{propcharS1} and \ref{propcharS3}. If \(\mu\) is a linear functional of the form \(\E c(T,S)\), then \(\mathsf{S}_\mu\) reduces to the Wasserstein correlation in \eqref{defWcc}.
\end{enumerate}
\end{remark}

\section{Comparison of \(\preccurlyeq_{ccx}\) with related dependence orderings}\label{seccompS}

In this section, we compare the conditional convex order with well-established dependence orderings from the literature. 
While Theorem \ref{propcharS} reveals strong connections between the conditional convex order and the concordance order in terms of dimension-reduced bivariate SI random vectors, the two dependence orderings are fundamentally different; see Figure \ref{Fig2} and Table \ref{tabpropkpct}.
As we show in Proposition \ref{rembivSchur}, the two orderings coincide for SI random vectors.
We then discuss related global dependence orderings studied in \cite{siburg2013,Shaked-2012,Shaked-2013} and demonstrate that \(\preccurlyeq_{ccx}\) provides a more natural and powerful framework for ordering the strength of functional dependence.

\subsection{Comparison of \(\preccurlyeq_{ccx}\) with the concordance order}\label{sec_cccx}

The certainly most popular \emph{positive} dependence order is the concordance order. For bivariate random vectors \((Y,X)\) and \((Y',X'),\) it is defined by
\begin{align}\label{defconcordance}
    (Y,X)\leq_c (Y',X') \quad : \,\Longleftrightarrow \quad (Y,X)\leq_{lo} (Y',X') \quad \text{and} \quad (Y,X)\leq_{uo} (Y',X'),
\end{align}
where the \emph{lower orthant} and \emph{upper orthant} order on the right-hand side of \eqref{defconcordance} are defined by \(P(Y\leq y,X\leq x)\leq P(Y'\leq y,X'\leq x)\) and \(P(Y > y,X > x)\leq P(Y'>  y,X' > x)\) for all \(x,y\in \R\), respectively, see e.g. \cite[Definition 2.4]{Joe-2014}.
By definition, \(\leq_c\) is a pure dependence order since \((Y',X')\leq_c (Y',X')\) implies \(X\eqd X'\) and \(Y\eqd Y'\). Note that for equal marginal distributions, the bivariate concordance order is equivalent to the lower (and also to the upper) orthant order; see Proposition \ref{propconcordequiv}.

To relate the concordance order and the conditional convex order, let us consider for the moment bivariate random vectors \((V,U)\) and \((V',U')\) with \(U(0,1)\) marginals. Then we are in the setting of bivariate copulas.
Using the notation \(\eta_{V|U}^v(t) = P(V\leq v\mid U=t)\) in \eqref{desetay} and disintegration, we observe that the concordance order can be characterized by
\begin{align}\label{repconcordanceorder}
    (V,U)\leq_c (V',U') \quad \Longleftrightarrow \quad \int_0^u \eta_{V|U}^v(t) \de t \leq \int_0^u \eta_{V'|U'}^v(t) \de t \quad \text{for all } u,v\in [0,1].
\end{align}
The conditional convex order has a similar representation due to its characterization via the Schur order in Theorem \ref{thecharSchur}. The integrand now consists of the decreasing rearrangements of \(\eta_{V|U}^v\) and \(\eta_{V'|U'}^v\,,\) respectively, i.e., 
\begin{align*}
    (V,U)\preccurlyeq_{ccx} (V',U') \quad \Longleftrightarrow \quad \int_0^u (\eta_{V|U}^v)^*(t) \de t \leq \int_0^u (\eta_{V'|U'}^v)^*(t) \de t \quad \text{for all } u,v\in [0,1].
\end{align*}
The integrated decreasing rearrangements above coincide with the integrals in \eqref{repconcordanceorder} if and only if the functions \(\eta_{V|U}^v\) and \(\eta_{V'|U'}^v\) are decreasing for all \(v\,.\) The latter is equivalent to the property that \(V\) is SI in \(U\), and \(V'\) is SI in \(U'\). 

An extension of the above reasoning to arbitrary marginals yields the following result. It states that \(\preccurlyeq_{ccx}\) and \(\leq_c\) coincide for bivariate SI random vectors from the same Fr\'{e}chet class.

\begin{figure}
\begin{center}
    \begin{minipage}{0.4\textwidth}
        \begin{tikzpicture}
            \draw[ultra thick] (0,0) circle (2cm);
            
    \tikzset{>={Stealth[length=8pt, width=5pt]}}

		\draw[->, thick] (0,0) ellipse (1.5cm and 2cm); 
    \draw[->, thick] (0,0) ellipse (1cm and 2cm);  
    \draw[->, thick] (0,0) ellipse (0.5cm and 2cm);    
    \draw[-, thick] (0,-2) -- (0,2);
    
    \draw[->,thick] (-1.5,-0.01) -- (-1.5,0.01);
     \draw[->,thick] (-1.0,-0.01) -- (-1.0,0.01);
      \draw[->,thick] (-0.5,-0.01) -- (-0.5,0.01);
       \draw[->,thick] (0,-0.01) -- (0,0.01);
         \draw[->,thick] (0.5,-0.01) -- (0.5,0.01);
         \draw[->,thick] (1.0,-0.01) -- (1.0,0.01);
          \draw[->,thick] (1.5,-0.01) -- (1.5,0.01);

            \filldraw (90:2cm) circle (2pt);
            \node[above] at (90:2cm) {\(M\)};
            
            \filldraw (270:2cm) circle (2pt);
            \node[below] at (270:2cm) {\(W\)};
            
            \filldraw (0,0) circle (2pt);
            \node[below right] at (0,0) {$\Pi$};
        \end{tikzpicture}
    \end{minipage}
    \hspace{0.1cm} 
    \begin{minipage}{0.4\textwidth}
        \begin{tikzpicture}
         \tikzset{>={Stealth[length=8pt, width=5pt]}}
            \draw[ultra thick] (0,0) circle (2cm);
            
            \foreach \angle in {0, 30,60, 90,120,150,180,210,240,270,300,330} {
                \draw[->, thick] (0,0) -- (\angle:2cm);
            }

            \filldraw (90:2cm) circle (2pt);
            \node[above] at (90:2cm) {\(M\)};
            
            \filldraw (270:2cm) circle (2pt);
            \node[below] at (270:2cm) {\(W\)};
            
            \filldraw (0,0) circle (2pt);
            \node[below right] at (0,0) {$\Pi$};
        \end{tikzpicture}
    \end{minipage}
\end{center}
\caption{Visualization of the concordance order (left panel) and the conditional convex order (right panel) for bivariate copulas. Here, \(W\) and \(M\) denote the lower and upper Fr\'{e}chet copula and \(\Pi\) the independence copula. While \(W\) and \(M\) are the uniquely determined minimal and maximal elements in concordance order, the conditional convex order admits the independence copula as a global minimal element and all perfectly dependent copulas as global maximal elements.}
    \label{Fig2}
\end{figure}
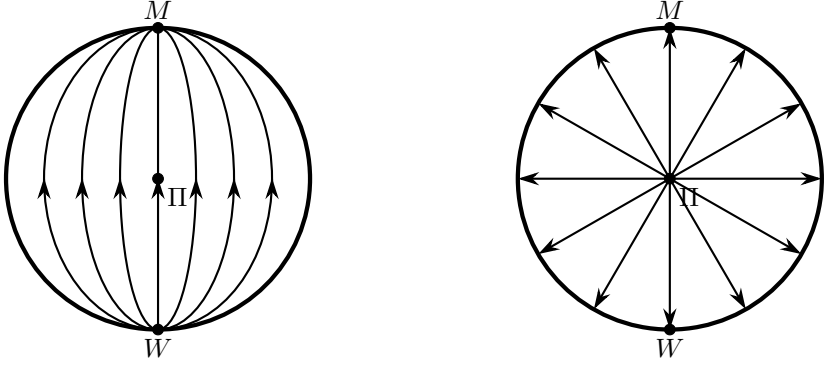

\begin{proposition}[Characterization of \(\preccurlyeq_{ccx}\) for bivariate SI random vectors]\label{rembivSchur}~\\
    Let \((Y,X)\) and \((Y',X')\) be bivariate SI random vectors with \(Y\eqd Y'\) and \(X\eqd X'\,.\) Then the conditional convex order and the concordance order are equivalent, i.e., 
    \begin{align}\label{eqrembivSchur}
        (Y,X)\preccurlyeq_{ccx} (Y',X') \quad \Longleftrightarrow \quad (Y,X)\leq_{c} (Y',X')\,.
    \end{align}
    In particular, \((Y,X) =_{ccx} (Y',X')\) is equivalent to \((Y,X) \eqd (Y',X')\,.\)
\end{proposition}

\begin{table}[t]
\centering
\scalebox{0.95}{
\begin{tabular}{|l||c|c|}
\hline
\textbf{Dependence orders} & \textbf{\((Y,X)\leq_c (Y',X')\)} & \textbf{\((Y,\XX)\preccurlyeq_{ccx} (Y',\XX')\)}  \\ \hline \hline
   Domain  &  Fr\'echet class &  \(\cR_A=\{(Y,\XX)\mid \overline{\Ran(F_Y)} = A\}\)  \\ \hline
   Transitivity & \cmark &  \cmark \\
   Reflexivity & \cmark & \cmark  \\
   Antisymmetry & \cmark & \xmark   \\
   Symmetry & \cmark & \xmark   \\\hline
   Lower bound & 
   Countermonotonicity & Independence\\
   Upper bound & Comonotonicity  & Perfect dependence \\ \hline
   Invariance under increasing transformations & \((a(Y),X)\leq_c (a(Y'),X')\) & \((a(Y),\XX)\preccurlyeq_{ccx} (a(Y'),\XX')\) \\
   Invariance under decreasing transformations & \((b(Y),X)\geq_c (b(Y'),X')\) & \((b(Y),\XX)\preccurlyeq_{ccx} (b(Y'),\XX')\) \\
   Invariance under bijective transformations & \xmark & \((Y,{\mathbf{h}}(\XX))\preccurlyeq_{ccx} (Y',\mathbf{h}(\XX'))\) 
   \\ \hline
\end{tabular}}
\vspace{5mm}
\caption{Basic properties of the bivariate concordance order and the conditional convex order, where \(A\) is a closed subset of \([0,1]\), \(a\) and \(b\) are strictly increasing/decreasing functions, and \(\mathbf{h}\) is a bijective function}
\label{tabpropkpct}
\end{table}

\subsection{Comparison of \(\preccurlyeq_{ccx}\) with the dilation order and dispersive order}\label{subsecdispdil}

Recall that the conditional convex order is defined through the variability of conditional survival functions in the \emph{conditioning} variable.
As discussed at the beginning of Section \ref{sec12}, another approach of constructing a dependence order for \(\xi\) is to study the variability of the conditional survival function in the \emph{conditioned} variable, i.e., to study the variability of the random variable
\begin{align}\label{eqvarY}
    P(Y\geq q_Y(V) \mid \XX = {\bf q}_\XX(\uu)) 
\end{align}
for Lebesgue-almost all \(\uu\in [0,1]^p.\) Here, \(V\sim U(0,1)\) is independent from \(\XX\). 
Similar as for \(\preccurlyeq_{ccx}\), maximal variability in \eqref{eqvarY} (in convex order) corresponds to perfect dependence of \(Y\) on \(\XX\); see \cite[Proposition 3\,(ii)]{siburg2013}. However, a drawback of this approach is that the random variable in \eqref{eqvarY} is not constant for independent \(\XX\) and \(Y\). Hence, independence of \(\XX\) and \(Y\) cannot be identified with minimal elements in convex order. 
Moreover, a comparison of the random variables in \eqref{eqvarY} with respect to the convex order requires equal means, which is typically not the case for conditional distributions.
To overcome the latter shortcoming, the idea in \cite{siburg2013} is to center the conditional distributions by subtracting conditional means. This leads to the concept of \emph{dilation order} defined by
    \begin{align}
        \label{defdil}
        S & \leq_{dil} T \quad \text{if} \quad S - \E S \leq_{cx} T - \E T\,.
    \end{align}
    A stronger order, that implies the dilation order, is the \emph{dispersive order}, also discussed in \cite{siburg2013}.
        The following result can be extended to arbitrary random vectors \(\XX\) and \(\XX'\).

    \begin{proposition}[Dilation order criterion {\cite[Theorem 2]{siburg2013}}]\label{propdisp}~\\
        Let \((Y,X)\) and \((Y',X')\) be bivariate random vectors with continuous marginal distribution functions.
        Then 
        \begin{align}\label{eqpropdisp}
        (F_Y(Y)|X = q_X(u)) \leq_{dil} (F_{Y'}(Y')|X' = q_{X'}(u)) \quad \text{for } \lambda\text{-almost all } u\in (0,1) 
    \end{align}
    implies \(\xi(Y,X) \leq \xi(Y',X')\,.\)
    \end{proposition}
    To verify the dilation order in \eqref{eqpropdisp}, one needs to substract the conditional expectations \(\E[F_Y(Y)|X = q_X(u)]\) and \(\E[F_{Y'}(Y')|X' = q_{X'}(u)],\) respectively.
    However, as we discuss in the sequel, the centered conditional distributions are even in the simplest models not comparable in convex order, so that the dilation order in \eqref{eqpropdisp} cannot be verified; see Figure \ref{figDilord}.
To this end, we use the following necessary condition for convex order; see \cite[Equation (3.A.12)]{Shaked-2007}. We denote by \(\supp(S)\) the support of a random variable \(S\). Similarly, \(\supp(S\mid T=t)\) is the support of the conditional distribution of \(S\) given \(T=t\).

\begin{lemma}
    Let \(S\) and \(T\) be random variables whose supports are intervals. Then \(S\leq_{cx} T\) implies \(\supp(S) \subseteq \supp(T)\,.\)
\end{lemma}
As a consequence of the above lemma, \(S\) and \(T\) are not comparable in convex order if 
\begin{align}\label{eqsuppcondncx}
    \inf(\supp(S)) < \inf(\supp(T)) \quad \text{and} \quad \sup(\supp(S)) < \sup(\supp(T))\,.
\end{align}
Using \eqref{eqsuppcondncx}, we obtain the following result, which limits the use of the dilation order criterion for conditional distributions in Proposition \ref{propdisp}.

\begin{proposition}\label{propnodil}
    Let \((V,U)\) and \((V',U')\) be bivariate random vectors. Assume that \(\E[V\mid U=u] \ne \E[V'\mid U'=u]\) and \(\supp(V\mid U=u) = \supp(V'\mid U'=u)\,.\) Then,  neither \((V\mid U =u) \leq_{dil} (V'\mid U'=u)\) nor \((V\mid U =u) \geq_{dil} (V'\mid U'=u)\,.\)
\end{proposition}

As studied in Section \ref{sec52}, the conditional convex order can be verified in various settings, including the additive error model \eqref{eqadderrmod}. 
In contrast, the approach via the dilation order criterion in Proposition \ref{propdisp} is too strong even in standard settings, as the following example shows.

\begin{example}[Dilation order]\label{exdisp}
     Consider the additive error models \(Y = X + \sigma \varepsilon\) and \(Y' = X + \sigma' \varepsilon\) for \(0 < \sigma < \sigma'\) where \(\varepsilon,X\) are standard normal and independent.
     To verify \eqref{eqpropdisp}, we obtain for \(V\sim U(0,1)\) independent of \(X\) that
     \begin{align}\label{eqdisp1}
         F_{F_Y(Y)|X = q_X(u)}^{-1}(V) &= \Phi\left(\frac{(1+\sigma) \Phi^{-1}(V) - \Phi^{-1}(u)}{\sigma}\right) \quad \text{and} \\
         \label{eqdisp2}
         F_{F_{Y'}(Y')|V = q_{X}(u)}^{-1}(V) &= \Phi\left(\frac{(1+\sigma') \Phi^{-1}(V) - \Phi^{-1}(u)}{\sigma'}\right)
     \end{align}
     for \(\lambda\)-almost all \(u\in (0,1)\,.\) It is not difficult to see that 
     \(\E[F_Y(Y) \mid X = q_X(u)] \ne \E[F_{Y'}(Y') \mid X = q_{X}(u)]\) for \(\lambda\)-almost all \(u\in (0,1)\,.\) Since \(\supp(F_Y(Y) \mid X = q_X(u)) = \supp(F_{Y'}(Y') \mid X = q_{X}(u)) = [0,1]\) for all \(u\in (0,1)\,,\) we obtain from Proposition \ref{propnodil} for \(\lambda\)-almost all \(u\in (0,1)\) that \((F_Y(Y) \mid X = q_X(u))\) and \((F_{Y'}(Y') \mid X' = q_{X'}(u))\) are not comparable in the dilation order (and thus neither in the dispersive order), see Figure \ref{figDilord}.
     Hence, Proposition \ref{propdisp} cannot be applied to prove \(\xi(Y,X) \geq \xi(Y',X)\,.\) Nonetheless, this inequality for \(\xi\) follows from Theorem \ref{theadderrmod} where we have shown that \((Y,X)\succcurlyeq_{ccx} (Y',X)\).
    
\end{example}

\begin{figure}
    \centering
    \includegraphics[width=0.8\linewidth, trim = 0 40 0 0, clip]{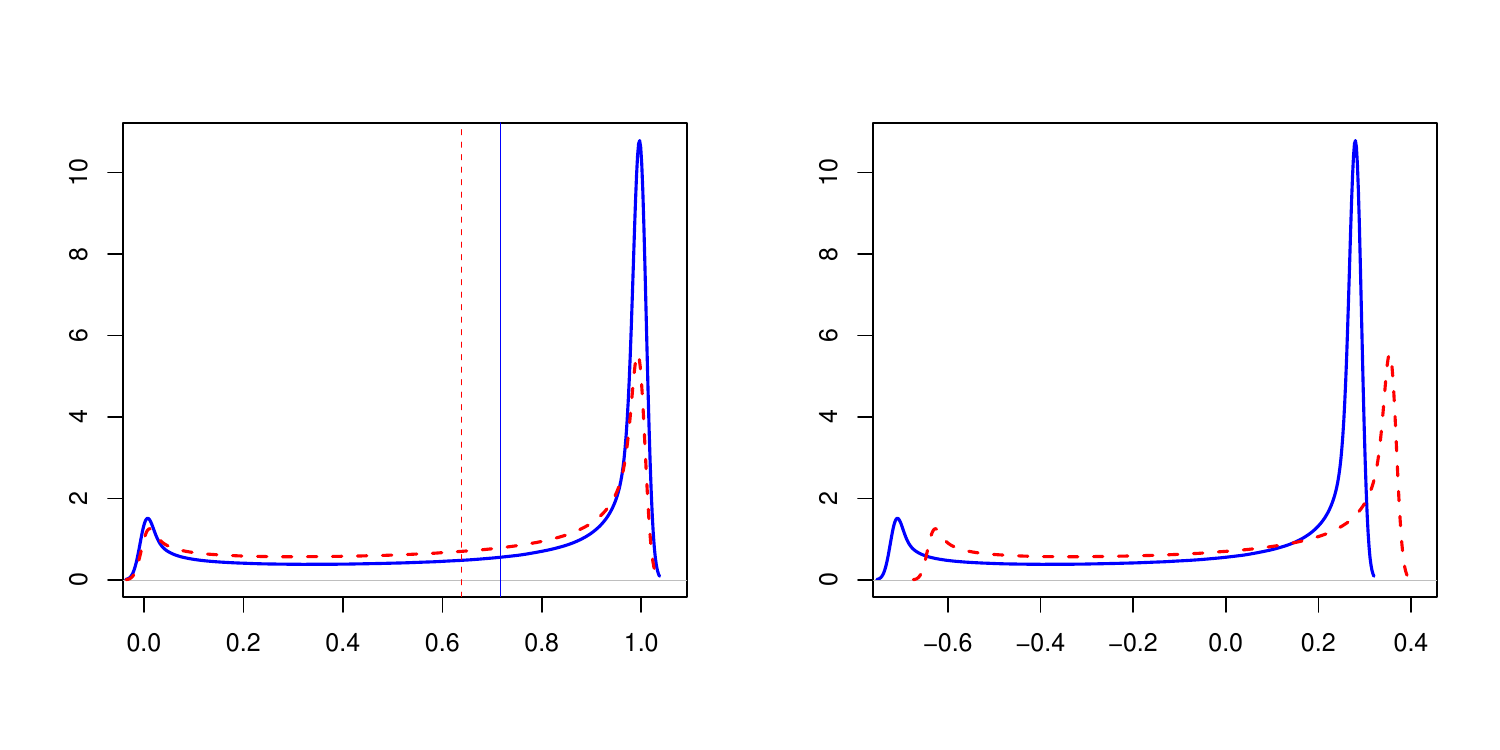}
    \caption{Left plot: Densities and means of the conditional distributions \(F_{Y}(Y)|X = q_{X}(0.1)\) for \(\sigma = 1\) (dashed) and \(\sigma = 2\) (solid). 
    Right plot: Densities from the left plot shifted by the mean. Since the supports are non-nested intervals satisfying \eqref{eqsuppcondncx}, they are not comparable in convex order.}
    \label{figDilord}
\end{figure}

\subsection{Variability of regression functions and magnitude of error curves}

The global dependence orderings studied in \cite{Shaked-2012,Shaked-2013} compare conditional expectations and conditional variances. As we discuss in the sequel, they are not sufficient for ordering results on dependence measures like \(\xi\).

Therefore, let \((Y,X)\) be a bivariate random vector. Then there exists a representation \(Y= f(X,U)\) a.s. for a measurable function \(f\) and a random variable \(U\) that is uniform on \((0,1)\) and independent of \(X\). The random variable \(U\) can be interpreted as noise. Informally, a weak (strong) influence of \(U\) on \(Y\) indicates strong (weak) dependence of \(Y\) on \(X\). 

The idea in \cite{Shaked-2012} is to define global dependence orderings for the regression function \(m(x) = \E[Y| X=x]\) and for the error function \(e(x) = \Var(Y| X=x)\).
For ordering the magnitude of \(e(X)\), the authors propose the (standard) stochastic order.
For \(m(X)\), they consider variability orderings such as the convex order, the dilation order, and the dispersive order. 
This approach comes closest to what we have been able to find in the literature on our approach, which compares the variability of the regression functions \(P(Y\geq q_Y(v)\mid X) = \E[\1_{\{\{Y\geq q_Y(v)\}\}} \mid X]\) for \emph{all} \(v\in (0,1)\). This is crucial for a dependence order underlying \(\xi\). 

In \cite{Shaked-2013}, the idea is to analyze the variability of the conditional expectation \(\tilde{m}(u) = \E[Y|U=u]\) and the magnitude of the conditional variance \(\tilde{e}(u) = \Var(Y|U=u)\). Similar to the approach in \cite{Shaked-2012}, this does not yield ordering results for dependence measures such as Chatterjee's rank correlation.

\makeatletter
\renewcommand{\appendix}{%
  \par
  \setcounter{section}{0}%
  \setcounter{subsection}{0}%
  \renewcommand{\thesection}{\Alph{section}}%
}
\makeatother

\appendix

\makeatletter
\renewcommand{\theequation}{\thesection.\arabic{equation}}
\setcounter{equation}{0}
\setcounter{theorem}{0} 
\renewcommand{\thetheorem}{\thesection.\arabic{theorem}}
\makeatother

\section*{\centering \huge Appendix}
\section{Auxiliary results and preparation for proofs in Appendix \ref{App:Proofs}}\label{appA}

For easier access to the proofs in Appendix \ref{App:Proofs}, we provide several useful results on generalized inverses, the convex order, the distributional transform, and the quantile transform.

\subsection{Generalized inverses}

For a distribution function \(G\), we denote by \(G^{-1}(v) := \inf\{y \mid G(y) \geq v\},\) \(v\in (0,1)\), the left-continuous generalized inverse of \(G\). We refer to \cite{Embrechts2013} and \cite[Chapter 21]{VdV2013} for various properties of generalized inverses and will frequently use the following specific properties. Therefore, we denote by \(G^-\) the left-continuous version of \(G\) and
we abbreviate \(\iota_G := G \circ G^{-1}\) and \(\iota_G^- := G^- \circ G^{-1}\).

\begin{proposition}\label{Prop.Inverse}
\begin{enumerate}[(i)]
  \item \label{Prop.Inverse1} \(G^{-1}(v) \leq y\) if and only if \(v \leq G(y)\),
  \item \label{Prop.Inverse1b} \(G^{-1} \circ G \circ G^{-1}(t) = G^{-1}(t)\) for all \(t\in (0,1)\),
  \item \label{Prop.Inverse2}\(\iota_G(t) = t = \iota_G^-(t)\) if and only if \(G\) is continuous at \(G^{-1}(t)\),
  \item \label{Prop.Inverse3}\(\iota_G^-\) is left-continuous while, in general, \(\iota_G\) is neither left- nor right-continuous, 
  \item \label{Prop.Inverse4}\(\iota_G(t) = t = \iota_G^-(t)\) for all \(t\in (0,1)\) if and only if \(G\) is continuous.
\end{enumerate}
\end{proposition}

Denote by \(y\mapsto \overline{F}_Y(y) := P(Y\geq y)\) the survival function of \(Y\).
Recall that we use the shorter notation $q_Y := F_Y^{-1}$ for the generalized inverse of a random variable \(Y\) with distribution function \(F_Y\).
The following example characterizes the marginal constraint in the definition of \(\preccurlyeq_{ccx}\).

\begin{lemma}[Marginal constraint]\label{lemmargconstraint}
    For real-valued random variables \(Y\) and \(Y'\), the following statements are equivalent:
    \begin{enumerate}[(i)]
        \item \label{lemmargconstraint1} \(\overline{\Ran(F_Y)} = \overline{\Ran(F_{Y'})}\),
        \item \label{lemmargconstraint2} \(\overline{F}_Y(q_Y(v)) = \overline{F}_{Y'}(q_{Y'}(v))\) for all \(v\in (0,1)\),
        \item \label{lemmargconstraint3} \(F_Y(q_Y(v)) = F_{Y'}(q_{Y'}(v))\) for \(\lambda\)-almost all \(v\in (0,1)\).
    \end{enumerate}
\end{lemma}

\subsection{Convex order}
For the proofs of several results,
we use the following properties of convex order; see \cite[Example 1.10.5]{Muller-2002} and \cite[Section 3.A.2]{Shaked-2007}.

\begin{lemma}[Some properties of convex order]\label{lemminmaxcx}
Let \(S_n,T_n,S,T\) be random variables with values in \([0,1]\). Then the following statements hold true:
\begin{enumerate}[(i)]
    \item \label{lemminmaxcx1} Minimal elements: \(\E S \leq_{cx} S\).
    \item \label{lemminmaxcx2} Maximal elements: If \(\E S = \E T\) and if \(T\) takes values only in \(\{0,1\}\), then \(S\leq_{cx} T\).
    \item \label{lemminmaxcx5} Antisymmetry: \(S\leq_{cx} T\) and \(S\geq_{cx} T\) implies \(S\eqd T\).
    \item \label{lemminmaxcx3} Invariance under linear transformations: \(S\leq_{cx} T\) implies \(c\,S + d \leq_{cx} c\, T + d\) for all constants \(c,d\in \R\).
    \item \label{lemminmaxcx4} Stability under weak convergence: If \(S_n\xrightarrow[]{~d~} S\), \(T_n\xrightarrow{~d~} T\), \(\E S_n\to \E S\), \(\E T_n  \to \E T \), and \(S_n \leq_{cx} T_n\) for all \(n\), then \(S\leq_{cx} T\).
\end{enumerate}
\end{lemma}

The conditional convex order is defined by comparing conditional survival probabilities \(P(Y\geq q_Y(v) \mid \XX)\) in convex order for \emph{all} \(v\in (0,1)\). Equivalent versions can be achieved for strict survival probabilities and cumulative probabilities as follows. Due to continuity properties of the transformations \(\iota_G\) and \(\iota_G^{-}\) in Proposition \ref{Prop.Inverse}\,\ref{Prop.Inverse2} and \ref{Prop.Inverse3}, exceptional null sets have to be considered.

\begin{proposition}[Versions of conditional convex order]\label{propversccx}~\\
    The following statements are equivalent:
    \begin{enumerate}[(i)]
        \item \label{propversccx1} \((Y,\XX) \preccurlyeq_{ccx} (Y',\XX')\)
        \item \label{propversccx2} \(P(Y < q_Y(v) \mid \XX) \leq_{cx} P(Y' < q_{Y'}(v)\mid \XX')\) for all \(v\in (0,1)\).
        \item \label{propversccx3} \(P(Y\leq q_Y(v) \mid \XX) \leq_{cx} P(Y'\leq q_{Y'}(v)\mid \XX')\) for \(\lambda\)-almost all \(v\in (0,1)\).
        \item \label{propversccx4} \(P(Y > q_Y(v) \mid \XX) \leq_{cx} P(Y' > q_{Y'}(v)\mid \XX')\) for \(\lambda\)-almost all \(v\in (0,1)\).
    \end{enumerate}
\end{proposition}

The following example shows that the exceptional null set in the version of Proposition \ref{propversccx}\,\ref{propversccx3} is necessary.

\begin{example}\label{exdiscY}
    Consider a random variable \(Y\) with distribution function given by the left-hand graph in Figure \ref{figdefschur}. The right-hand graph is the distribution function of the random variable \(Y' = Y+2\) if \(Y\leq 0.5\,,\) and \(Y'= Y+1.5\) if \(Y>0.5\,.\) We observe that, for \(v = F_Y(0.5) = 0.4\), it is \(P(Y\leq q_Y(v)) = 0.4 \ne 0.8 = P(Y'\leq q_{Y'}(v))\).  Hence, \(P(Y\leq q_Y(v)\mid \XX)\) and \(P(Y'\leq q_{Y'}(v) \mid \XX')\) are not comparable in convex order, independent of the choice of \(\XX\). However, since \(\overline{\Ran(F_Y)} = \overline{\Ran(F_{Y'})}\) and since \(F_Y^-\circ q_Y\) and \(F_{Y'}^-\circ q_{Y'}\) are left-continuous, it follows that \(P(Y\geq q_Y(v)) = P(Y'\geq q_{Y'}(v))\) for \emph{all} \(v\in (0,1)\). Consequently, \((Y,\XX)\) and \((Y',\XX)\) are comparable in conditional convex order; more precisely, it is \((Y,\XX) =_{ccx} (F_Y(Y),\XX) =_{ccx} (F_{Y'}(Y'),\XX) =_{ccx} (Y',\XX)\) for any choice of \(\XX\) using distributional invariance of \(\preccurlyeq_{ccx}\) and \(F_{Y'}(Y') \eqd F_Y(Y)\).
\end{example}

\begin{figure}
    \centering
    \includegraphics[width=0.7\linewidth]{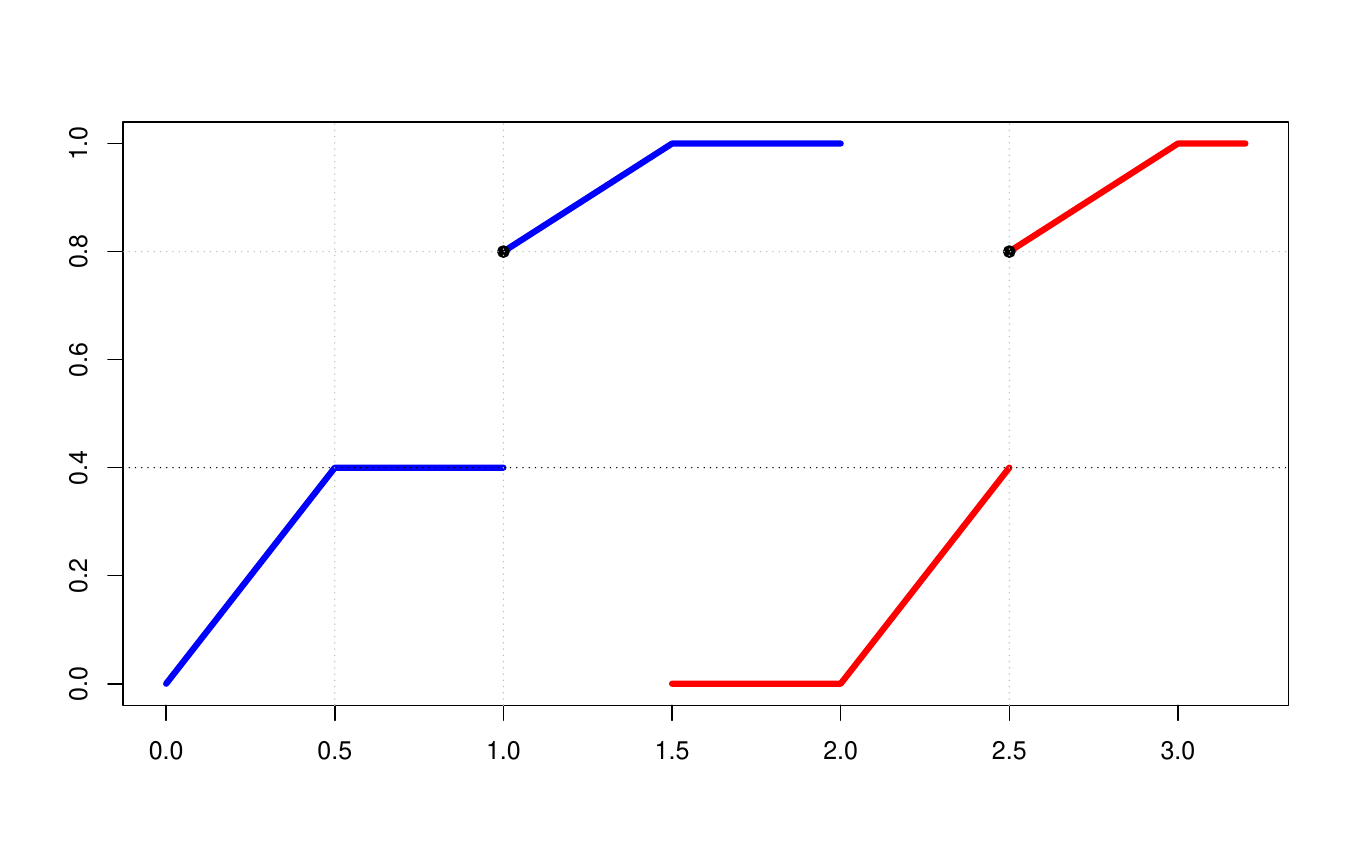}
    \caption{Left: the distribution function of a random variable \(Y\,;\) right: the distribution function of a random variable \(Y'\) which is a shift of \(Y\) given by \(Y' = Y+2\) if \(Y\leq 0.5\,,\) and \(Y'= Y+1.5\) if \(Y>0.5\,;\) see Example \ref{exdiscY}. Since \(Y\) and \(Y'\) have the same strength of functional dependence on a random vector \(\XX\) in terms of Chatterjee's rank correlation (i.e., \(\xi(Y|\XX) = \xi(Y'|\XX)\)), they should be equally ranked in a suitable global dependence order.}
    \label{figdefschur}
\end{figure}

\subsection{Distributional transform and quantile transform}

Let \(\VV=(V_1,\ldots,V_p)\) be a random vector that is independent of \(\XX = (X_1,\ldots,X_p)\) and uniform on \((0,1)^p\). Then
the \emph{multivariate distributional transform}  \(\tau_{\XX} (\XX,\VV)\) of \(\XX\) (also known as generalized Rosenblatt transform) is defined by  
\begin{align*}
    \tau_{\XX} (\xx,\boldsymbol{\lambda})
    := \big( F_1(x_1,\lambda_1), F_2(x_2,\lambda_2|x_1) \dots, F_p(x_p,\lambda_p|x_1,\dots,x_{p-1}) \big)
\end{align*}
for \(\xx=(x_1,\ldots,x_p) \in \mathbb{R}^p\) and \(\boldsymbol{\lambda}=(\lambda_1,\ldots,\lambda_p) \in [0,1]^p\), where
\begin{eqnarray*}
    F_1(x_1,\lambda_1) 
    & := & P(X_1 < x_1) + \lambda_1 \, P(X_1 = x_1)
    \\
    F_i(x_i,\lambda_i|x_1,\dots,x_{i-1}) 
    & := & P(X_i < x_i | X_1=x_1, \dots, X_{i-1}=x_{i-1}) 
    \\
    &    & +\, \lambda_i \, P(X_i = x_i | X_1=x_1, \dots, X_{i-1}=x_{i-1})\,, \quad i\in \{2,\ldots,p\}\,; 
\end{eqnarray*}
see \cite[Section 1.3]{Ru-2013}.
For \(p=1\) and for a random variable \(X\) with continuous distribution function \(F_X\,,\) the distributional transform \(\tau_X(X,V)\) simplifies to \(F_X(X)\,,\) which is uniform on \((0,1)\).

As an inverse transformation of \(\tau_\XX\,,\) we consider for a random vector \(\UU=(U_1,\ldots,U_p)\), uniformly on \((0,1)^p\) distributed, the \emph{multivariate quantile transform} \({\bf q}_\XX (\UU):=(\zeta_1,\ldots,\zeta_p)\). It is iteratively defined by 
\begin{align}
\begin{split}
    \zeta_1 &:= F_{X_1}^{-1}(U_1)\,, \label{MQT1}  \\
    \zeta_i &:= F^{-1}_{X_i|X_{i-1}=\xi_{i-1},\ldots,X_1=\xi_1}(U_i)~~~ \text{for all } i\in \{2,\ldots,p\}
\end{split}
\end{align}
where \(F_{W|\ZZ=\zz}\) and \(F^{-1}_{W|\ZZ=\zz}\) denote the conditional distribution function of \(W\) and its generalized inverse given \(\ZZ=\zz\).

According to \cite[Theorem 1.12]{Ru-2013},
\begin{align}\label{eqpropqt0}
\tau_\XX(\XX,\VV) &\text{ is a random vector that is uniformly on } (0,1)^p \text{ distributed,}\\
 \label{eqpropqt} {\bf q}_\XX(\UU) &\text{ is a random vector with distribution function } F_\XX,
\end{align}
and the multivariate quantile transform is inverse to the multivariate distributional transform,
i.e., 
\begin{align}\label{eqpropqt2}
  \XX = {\bf q}_\XX \big( \tau_{\XX} (\XX,\VV) \big) \qquad P\text{-almost surely.}
\end{align}

\section{Proofs} \label{App:Proofs}

\subsection{Proof of Theorem \ref{thefundpropS}}

\begin{proof}[Proof of Theorem \ref{thefundpropS}.] 
Axioms \ref{axdepord0} and \ref{axdepord1} follow from the definition of the conditional convex order (Definition \ref{defccx}). \\
To verify Axiom \ref{axdepord2}, first note that the deterministic random variable \(P(Y\geq q_Y(v))\) is minimal in convex order within the class of \([0,1]\)-valued random variables having mean \(P(Y\geq q_Y(v))\); see Lemma \ref{lemminmaxcx}. Therefore, if \(\XX\) and \(Y\) are independent, then \((Y,\XX)\preccurlyeq_{ccx} (Y',\XX')\) for all random vectors \((Y',\XX')\in \cR_A\).
For the reverse direction, assume that \((Y,\XX)\preccurlyeq_{ccx} (Y',\XX')\) for all random vectors
\((Y',\XX')\in \cR_A\,.\) 
This implies for deterministic \(\XX'\), say \(\XX' = \mathbf{0}\) a.s., that 
\begin{align*}
    P(Y\geq q_Y(v) \mid \XX) \leq_{cx} P(Y'\geq q_{Y'}(v) \mid \XX') = P(Y'\geq q_{Y'}(v)) = P(Y\geq q_{Y}(v))
\end{align*}
for all \(v\in [0,1]\).
Since the unconditional random variable \(P(Y\geq q_{Y}(v))\) is deterministic and thus minimal in convex order (see Lemma \ref{lemminmaxcx}), we have for all \(y\in \R\) that \(P(Y\geq y \mid \XX) \eqd P(Y\geq y)\). This implies for all \(y\in \R\) that \(P(Y\geq y \mid \XX = \xx) = P(Y\geq y)\) for \(P^\XX\)-almost all \(\xx\), and thus \(\XX\) and \(Y\) are independent.

To verify Axiom \ref{axdepord3}, first assume that \(Y\) perfectly depends on \(\XX\), i.e., there exists a measurable function \(f\) such that \(Y = f(\XX)\) almost surely. This implies for all \(v\in (0,1)\) that
\begin{align*}
    P(Y\geq q_Y(v) \mid \XX) 
    & = \mathds{1}_{\{f(\XX)\geq q_Y(v)\}}\,.
\end{align*}
The random variable \(P(Y\geq q_Y(v) \mid \XX)\) is a maximal element in convex order within the class of \([0,1]\)-valued random variables having mean \(P(Y\geq q_Y(v))\); see Lemma \ref{lemminmaxcx}. 
Hence, for \((Y',\XX')\in \cR_A\), using the marginal constraint, it follows that \(P(Y'\geq q_{Y'}(v)) = P(Y\geq q_Y(v))\) and thus \(P(Y'\geq q_{Y'}(v) \mid \XX') \leq_{cx} P(Y\geq q_Y(v) \mid \XX)\) for all \(v\in (0,1)\).
Therefore, \((Y',\XX') \preccurlyeq_{ccx} (Y,\XX)\) for all random vectors \((Y',\XX')\in \cR_A\).
For the reverse direction, assume that \((Y',\XX')\preccurlyeq_{ccx} (Y,\XX)\) for all \((Y',\XX')\in \cR_A\). This implies, in particular, \((Y,Y)\preccurlyeq_{ccx} (Y,\XX)\) and thus for all \(v\in (0,1)\) that
\begin{align*}
    P(Y\geq q_{Y}(v) \mid \XX) 
    \geq_{cx} P(Y\geq q_Y(v) \mid Y) 
    = \mathds{1}_{\{Y \geq q_Y(v)\}}\,.
\end{align*}
The right-hand side in the above inequality is a random variable that is maximal in convex order within the class of \([0,1]\)-valued random variables having mean \(P(Y\geq q_Y(v))\) (see Lemma \ref{lemminmaxcx}). This implies equality in convex order and, by antisymmetry of convex order, for all \(v\in (0,1)\) that \(P(Y\geq q_Y(v) \mid \XX) \eqd P(Y\geq q_Y(v) \mid Y) = \mathds{1}_{\{Y \geq q_Y(v)\}}\). 
It follows that, for all \(y\in \R\), \(P(Y\geq y \mid \XX = \xx ) \in \{0,1\}\) for \(P^\XX\)-almost all \(\xx\).
Hence, the conditional distribution \(Y\mid \XX = \xx\) is degenerate for \(P^\XX\)-almost all \(\xx\), which means that \(Y\) is a measurable function of \(\XX\); see the reasoning in \cite[Proof of Theorem 9.2]{chatterjee2021}.

To verify the information gain inequality in Axiom \ref{axdepord4}, we have
\begin{align*}
    \E \varphi\left( P(Y\geq q_Y(v) \mid \XX) \right) 
    & = \E\left[ \E \varphi\left( P(Y\geq q_Y(v) \mid \XX) \right) \mid (\XX,\ZZ) \right] \\
    & \leq \E \varphi\left( \E\left[P(Y\geq q_Y(v) \mid \XX) \mid (\XX,\ZZ)\right]\right) \\
    & = \E \varphi\left( P(Y\geq q_Y(v) \mid (\XX,\ZZ))\right)
\end{align*}
using Jensen's inequality for conditional expectation.

To verify Axiom \ref{axdepord5}, assume first that \((Y,\XX)=_{ccx} (Y,(\XX,\ZZ))\,.\) 
This implies \(\xi(Y,\XX)=\xi(Y,(\XX,\ZZ))\,.\) The latter means that \(Y\) and \(\ZZ\) are conditionally independent given \(\XX\,;\) see \cite[Lemma 11.2]{chatterjee2021}.
For the reverse direction, conditional independence of \(Y\) and \(\ZZ\) given \(\XX\) implies for all \(v\) that \(P(Y\geq q_Y(v)\mid \XX) = P( Y\geq q_Y(v) \mid (\XX,\ZZ))\) almost surely and thus \((Y,\XX) =_{ccx} (Y,(\XX,\ZZ))\). 

To verify Axiom \ref{axdepord6}, observe that \(\overline{\Ran(F_{g(Y)})} = \overline{\Ran(F_Y)} \). It follows that \(P(g(Y)\geq q_{g(Y)}(v)) = P(Y\geq q_Y(v))\) for all \(v\in (0,1)\,;\) see Lemma \ref{lemmargconstraint}.
Hence, the invariance property for \(Y\) follows from the definition of conditional convex order.
To show the invariance property for \(\XX\), let \(\bold{h}\) be bijective. This implies that the \(\sigma\)-algebras generated by \(\XX\) and \(\bold{h}(\XX)\) coincide. Hence, the statement follows from the definition of the conditional convex order.

Axiom \ref{axdepord7}: Since \(\overline{\Ran(F_{F_Y(Y)})} = \overline{\Ran(F_Y)}\), distributional invariance for \(Y\) follows similar to the proof of Axiom \ref{axdepord6}. To prove distributional invariance for \(\XX\),
    it is straightforward to verify that \(\xi(X_k,F_{X_k}(X_k)) = 1\). Hence, there exists a measurable function \(f_k\) such that \(X_k = f_k(F_{X_k}(X_k))\) almost surely. Applying the data processing inequality twice, we obtain
    \begin{align*}
        \big(Y, (F_{X_1}(X_1),\ldots,F_{X_p}(X_p))\big) &\preccurlyeq_{ccx} \big(Y,(X_1,\ldots,X_p)\big) \\
        &=_{ccx} \big(Y, (f_1(F_{X_1}(X_1)),\ldots,f_p(F_{X_p}(X_p)))\big) \\
        &\preccurlyeq_{ccx} \big(Y, (F_{X_1}(X_1),\ldots,F_{X_p}(X_p))\big),
    \end{align*}
    which proves the statement.
%
\end{proof}

\subsection{Proof of Theorem \ref{propcharS}}

The proof of Theorem \ref{propcharS} is based on several auxiliary results as well as on Theorem \ref{thecharSchur} and Theorem \ref{lemtrafSI} which we prove first.

\begin{proof}[Proof of Theorem \ref{thecharSchur}]
    The conditional convex order \((Y,\XX)\preccurlyeq_{ccx} (Y',\XX')\) is equivalent to \(P(Y\leq q_Y(v)\mid \XX)\leq_{cx} P(Y'\leq q_{Y'}(v)\mid \XX')\) for \(\lambda\)-almost all \(v\in (0,1)\); see Proposition \ref{propversccx}. However, the latter is by the Hardy-Littlewood-Polya theorem (see Lemma \ref{lemhlpt}) equivalent to \(\eta_{Y|\XX}^v \prec_S \eta_{Y'|\XX'}^v\) for \(\lambda\)-almost all \(v\in (0,1)\).
\end{proof}

For the proof of Theorem \ref{lemtrafSI}, we make use of the following result on rearrangements of conditional distribution functions.

\begin{lemma}[Rearranged conditional distribution function]\label{lemrearrqt}
Let \((Y,\XX)\) be a \((1+p)\)-dimensional random vector. Then there exists a function \(\R \times (0,1) \ni (y,u)\mapsto F_u(y)\) such that \(u\mapsto F_u(y)\) is a decreasing rearrangement of \(\uu\mapsto P(Y\leq y \mid \XX = q_\XX(\uu))\) for all \(y\in \R\) and \(y\mapsto F_u(y)\) is a distribution function for all \(u\in (0,1)\).
\end{lemma}

\begin{proof}
    For \(y\in \R\), denote by \(u\mapsto G_u(y)\) a decreasing rearrangement of \(\uu \mapsto P(Y\leq y \mid \XX=q_\XX(\uu))\). Since \(P(Y\leq y \mid \XX = q_\XX(\uu)) \to 0\) for all \(\uu\) as \(y\to - \infty\), also \(G_u(y) \to 0\) for all \(u\) as \(y\to -\infty\). Similarly, \(G_u(y)\to 1\) for all \(u\) as \(y\to +\infty\). For \(y\leq y'\), it is \(P(Y\leq y \mid \XX = q_\XX(\uu)) \leq P(Y\leq y'\mid \XX = q_\XX(\uu))\) for all \(\uu\). Hence, also the decreasing rearrangements satisfy \(G_u(y)\leq G_u(y')\) for all \(u\). Define \(H(u,y):= \int_0^u G_t(y) \de t\) for \((u,y)\in (0,1)\times \R\). It is straightforward to verify that \(H\) is a bivariate distribution function. 
    Then, by the regular version of conditional probability \cite[Theorem 6.3]{Kallenberg-2002}, there exists a family \((F_u)_{u\in [0,1]}\) of distribution functions such that
    \begin{align*}
        \int_0^u F_t(y) \de t = H(u,y) = \int_0^u G_t(y)\de t\quad \text{for all } (u,y)\in (0,1)\times \R.
    \end{align*}
    For fixed \(y\), it follows that \(F_t(y) = G_t(y)\) for \(\lambda\)-almost all \(t\in (0,1)\). This implies
    \begin{align*}
        \lambda(\{t \mid F_t(y)\geq w\}) = \lambda(\{t\mid G_t(y) \geq w\}) \quad \text{for all } w\in \R.
    \end{align*}
    Hence, also \(F_u\) is a decreasing rearrangement of \(\uu\mapsto P(Y\leq y\mid \XX=q_\XX(\uu))\). Since \(F_u\) is a distribution function for all \(u\), the statement is proven.
\end{proof}

\begin{proof}[Proof of Theorem \ref{lemtrafSI}.]
    To show statement \ref{lemtrafSI0}, we obtain for all \(y\in \R\) that
    \begin{align*}
        P\big( q_{Y;\XX}^{\uparrow U}(V) \leq y\big) &= \int_{[0,1]} P(F_u^{-1}(V)\leq y \mid U=u) \de \lambda(u) 
        = \int_{[0,1]} P(F_u^{-1}(V)\leq y ) \de \lambda(u) \\
        &= \int_{[0,1]} F_u(y) \de \lambda(u) 
        = \int_{[0,1]^p} P\big(Y\leq y \mid \XX = q_{\XX}(\uu)\big)  \de \lambda^p(\uu) \\
        &= \int_\Omega P(Y\leq y \mid \XX) \de P = P(Y\leq y),
    \end{align*}
    where the use for the first equality the definition of the rearranged quantile transform in \eqref{defyu}. For the second equality, we use independence of \(U\) and \(V\). The third equality follows from Proposition \ref{Prop.Inverse}\,\ref{Prop.Inverse1}, using the fact that \(F_u\) is a distribution function, and from the fact that \(V\) is uniform on \((0,1)\). The fourth equality holds by definition of \(F_u\) as a rearrangement of \(\uu \mapsto P(Y\leq y \mid \XX = q_{\XX}(\uu)) \). The fifth equality follows with \eqref{eqpropqt}, and the last equality is due to disintegration. \\
    To show statement \ref{lemtrafSI1}, recall that the family \(\{q_{Y;\XX}^{\uparrow u}(V)\}_{[0,1]}\) is SI in \(u\), that is, \(\E f(q_{Y;\XX}^{\uparrow u}(V)) = \E[f(q_{Y;\XX}^{\uparrow u}(V))\mid U=u] \) is increasing in \(u\) for all increasing functions \(f\) such that the expectations exist, where we use for the equality that \(U\) and \(V\) are independent. However, the latter means that \((q_{Y;\XX}^{\uparrow U}(V),U)\) is a bivariate SI random vector.\\
    To show statement \ref{lemtrafSI2}, let \(v\in (0,1)\). Then we obtain for all \(u\in [0,1]\) that
    \begin{align*}
        \int_0^u P\big( q_{Y;\XX}^{\uparrow U}(V)\leq q_{q_{Y;\XX}^{\uparrow U}}(v) \mid U = s\big) \de s &= \int_0^u P\big(q_{Y;\XX}^{\uparrow s}(V) \leq q_Y(v)\big) \de s \\
        &= \int_0^u F_s(q_Y(v)) \de s = \int_0^u \big(\eta_{Y|\XX}^{v}\big)^*(s) \de s,
    \end{align*}
    where we use for the first equality \(q_{Y;\XX}^{\uparrow U} \eqd Y\) and independence of \(U\) and \(V\). The second equality follows from \eqref{defyu} and Proposition \ref{Prop.Inverse}\,\ref{Prop.Inverse1}. The last equality is due to \eqref{definreacdf}. Since the integrand of the first integral is decreasing in \(s\), it coincides \(\lambda\)-almost surely with the decreasing rearrangement \((\eta_{Y|\XX}^v)^*\). Hence, the characterization of \(\preccurlyeq_{ccx}\) through the Schur order in Theorem \ref{thecharSchur} implies \((q_{Y;\XX}^{\uparrow U}(V),U) =_{ccx} (Y,\XX)\).
\end{proof}

We also use the following result, which gives several variants of the concordance order. 
Recall that \((Y,X)\leq_c (Y',X')\) is defined by \(P(Y\leq y, X\leq x)\leq P(Y'\leq y, X'\leq x)\) and \(P(Y > y, X > x) \leq P(Y' > y, X' > x)\) for all \(x,y\in \R\). 

\begin{proposition}[Concordance order]\label{propconcordequiv}
    Let \(X,Y,X',Y'\) be real-valued random variables. Then the following statements are equivalent:
    \begin{enumerate}[(i)]
        \item \label{propconcordequiv1} \((Y,X)\leq_c (Y',X')\),
        \item \label{propconcordequiv2} \(P(Y\leq y, X\leq x)\leq P(Y'\leq y, X'\leq x)\) and \(P(Y\geq y, X\geq x) \leq P(Y'\geq y, X'\geq x)\) for all \(x,y\in \R\).
        \item \label{propconcordequiv3}  \(Y\eqd Y'\), \(X\eqd X'\), and \(P(Y\leq y, X\leq x) \leq P(Y'\leq y, X'\leq x)\) for all \(x,y\in \R\).
        \item \label{propconcordequiv4} \(Y\eqd Y'\), \(X\eqd X'\),  and \(P(Y\geq y, X\geq x) \leq P(Y'\geq y, X'\geq x)\) for all \(x,y\in \R\).
    \end{enumerate}
\end{proposition}

\begin{proof} The proof follows from standard arguments like continuity of measures and the inclusion-exclusion principle.
\end{proof}

We also need the following variants of the pointwise order of distribution functions under the marginal constraint.

\begin{lemma}\label{lemconcordlfs}
Let \(Y,Y'\) be real-valued random variables, and let \(U\) be uniform on \((0,1).\)
    Under the marginal constraint \(\overline{\Ran(F_Y)} = \overline{\Ran(F_{Y'})}\), the following statements are equivalent:
    \begin{enumerate}[(i)]
        \item \label{lemconcordlfs1} \(P(Y\leq q_Y(v),U\leq u) \leq P(Y'\leq q_{Y'}(v),U\leq u)\) for \(\lambda\)-almost all \(v\in (0,1)\) and for all \(u\in (0,1),\)
        \item \label{lemconcordlfs2} \(P(F_Y(Y)\geq v, U\geq u) \leq P(F_{Y'}(Y')\geq v, U\geq u)\) for all \((u,v)\in (0,1)^2\).
    \end{enumerate}
\end{lemma}

\begin{proof}
    Assume \ref{lemconcordlfs1}. Denote by \(N_1\subset [0,1]\) 
    the exceptional \(\lambda\)-null set such that 
        \(P(Y\leq q_Y(v),U\leq u) \leq P(Y'\leq q_{Y'}(v),U\leq u)\)
     for all \(v\in (0,1)\setminus N_1\) and \(u\in (0,1)\). Noting that the marginal constraint implies \(P(Y\leq q_Y(v)) = P(Y'\leq q_{Y'}(v))\) for all \(v\in [0,1]\) outside a \(\lambda\)-null set \(N_2\) (see Lemma \ref{lemmargconstraint}), we obtain from the inclusion-exclusion principle that
    \begin{align}\label{prooflemconcordlfs1}
        P(Y > q_Y(v),U > u) \leq P(Y' > q_{Y'}(v),U > u)
    \end{align}
    for all \(v\in (0,1)\setminus N\) and for all \(u\in (0,1)\), where \(N:= N_1\cup N_2\). We aim show to that
    \begin{align}\label{prooflemconcordlfs2}
        P(Y \geq q_Y(v),U \geq u) \leq P(Y' \geq q_{Y'}(v),U \geq u)
    \end{align}
    for \(\lambda\)-almost all \(v\in (0,1)\) and for all \(u\in (0,1)\). Then, using the identities \(\{Y\geq q_Y(v)\} = \{F_Y(Y)\geq v\}\) and \(\{Y'\geq q_{Y'}(v)\} = \{F_{Y'}(Y')\geq v\}\) (see Proposition \ref{Prop.Inverse}\,\ref{Prop.Inverse1}), we obtain
    \begin{align}\label{prooflemconcordlfs3}
        P(F_Y(Y)\geq v, U\geq u) \leq P(F_{Y'}(Y')\geq v, U\geq u)
    \end{align}
    for \(\lambda\)-almost all \(v\in (0,1)\) and for all \(u\in (0,1)\). Finally, continuity of measures from above yields statement \ref{lemconcordlfs2}. To prove \eqref{prooflemconcordlfs2}, we consider two cases. For \(v\in \mathrm{int}(\overline{\Ran(F_Y)})\setminus N\), choose \(v_n\in \mathrm{int}(\overline{\Ran(F_Y)})\setminus N\) with \(v_n\uparrow v\). Then we have \(q_Y(v_n)\uparrow q_Y(v)\) and, using the marginal constraint, we also obtain \(q_{Y'}(v_n)\uparrow q_{Y'}(v)\). Hence, by continuity of measures from above, we obtain with \(u_n\uparrow u\) that 
    \begin{align}\label{prooflemconcordlfs4}
    \begin{split}
        P(Y \geq q_Y(v),U \geq u) 
        &= \lim_{n\to \infty} P(Y > q_Y(v_n),U > u_n) \\
        & \leq \lim_{n\to \infty} P(Y' > q_{Y'}(v_n),U > u_n) 
         = P(Y' \geq q_{Y'}(v),U' \geq u),
    \end{split}
    \end{align}
    where the inequality follows from \eqref{prooflemconcordlfs1}.
    For the second case, let \(v\in (0,1)\setminus (\overline{\Ran(F_Y)})\cup N)\) and define \(v_0:= F_Y^-(q_Y(v))\). Note that, by the marginal constraint and by left-continuity of the transformation \(\iota_{F_{Y}}^-=F_Y^-\circ F_Y^{-1}\) in Proposition \ref{Prop.Inverse}\,\ref{Prop.Inverse3}, we also have \(v_0 = F_{Y'}^-(q_{Y'}(v))\). Now, choose \(v_n \in [0,1]\setminus N\) with \(v_n\uparrow v_0\). Then 
    \begin{align}
        \{Y\geq q_Y(v)\} = \{Y\geq q_Y(v_0)\} = \bigcap_{n\in \N} \{Y > q_Y(v_n)\} \cap \{Y\notin [q_Y(v_0),q_Y(v))\}.
    \end{align}
    Using that \(\{Y\in [q_Y(v_0),q_Y(v))\}\) is a \(P\)-null set and using a similar reasoning for \(\{Y'\geq q_{Y'}(v)\}\), we obtain \eqref{prooflemconcordlfs4} also in the second case. This proves \eqref{prooflemconcordlfs2} and thus statement \ref{lemconcordlfs2}.

    For the reverse direction, assume that \eqref{prooflemconcordlfs3} holds for all \((u,v)\in (0,1
    )^2\). Using \(\{Y\geq q_Y(v)\} = \{F_Y(Y)\geq v\}\) and \(\{Y'\geq q_{Y'}(v)\} = \{F_{Y'}(Y')\geq v\}\), we obtain \eqref{prooflemconcordlfs2} for all \((u,v)\in (0,1)^2\). We aim to show \eqref{prooflemconcordlfs1} for \(\lambda\)-almost all \(v\in (0,1)\) and for all \(u\in [0,1]\), which implies statement \ref{lemconcordlfs1}. 
    To this end, consider first the case where \(v\in \mathrm{int}(\overline{\Ran(F_Y)})\). Choose \(v_n \downarrow v\) and \(u_n\downarrow u\). Then we have \(\bigcap_{n\in \N} \{Y\geq q_Y(v_n)\} = \{Y> q_Y(v) \} \setminus \{Y\in (q_Y(v),q_Y^+(v)]\}\) for \(q_Y^+(v) = \sup\{y\mid F_Y(y)\leq v\}\). Since \(\{Y\in (q_Y(v),q_Y^+(v)]\}\) is a \(P\)-null set and due to a similar reasoning for \(Y'\) (using the marginal constraint), we obtain \eqref{prooflemconcordlfs1}. 
    For the second case, let \(v\in (0,1)\setminus \overline{\Ran(F_Y)}\) and define \(v_0 := F_Y(q_Y(v)) = F_{Y'}(q_{Y'}(v))\). Then we have \(\bigcap_{n\in \N} \{Y\geq q_Y(v_n)\} = \{Y>q_Y(v_0)\} \setminus N_{v_0}\) for some \(Y\)-null set \(N_{v_0}\). Using a similar reasoning for an approximation of \(\{Y'>q_{Y'}(v)\}\) from above, we obtain \eqref{prooflemconcordlfs1} from \eqref{prooflemconcordlfs2}.
\end{proof}

The following lemma addresses distributional properties of quantile-based constructions.

\begin{lemma}\label{lemcondcomrv}
    Let \((W,\XX)\) be a \((1+p)\)-dimensional random vector. Let \(V\) be uniform on \((0,1)\) and independent from \(\XX\). Then we have
    \begin{enumerate}[(i)]
        \item \label{lemcondcomrv1} \(F_{W|\XX}^{-1}(V) \eqd W\),
        \item \label{lemcondcomrv2} \(P(F_{W|\XX}^{-1}(V) \leq y, V\leq y') = \int_{\R^p} \min\{F_{W|\XX=\xx}(y) , y'\} \de P^\XX(\xx)\) for all \(y,y' \in \R\).
    \end{enumerate}
\end{lemma}

\begin{proof}
The first statement is a special case of \eqref{eqpropqt}. For the second statement, we obtain from disintegration
\begin{align*}
    P\big(F_{W|\XX}^{-1}(V) \leq y, V\leq y'\big) 
    &= \int_{\R^p} P\big( F_{W|\XX = \xx}^{-1}(V) \leq y, V\leq y'  ) \de P^\XX(\xx) \\
    &= \int_{\R^p} \min\{F_{W|\XX=\xx}(y) , y'\} \de P^\XX(\xx),
\end{align*}
where we use independence of \(\XX\) and \(V\) for the first equality. The second equality follows with the transformation in Proposition \ref{Prop.Inverse}\,\ref{Prop.Inverse1} and the fact that \(V\) is uniform on \((0,1)\).
\end{proof}

We are now able to prove Theorem \ref{propcharS}.

\begin{proof}[Proof of Theorem \ref{propcharS}.] 
First, recall that a comparison in the concordance order implies identical marginal distributions; see Proposition \ref{propconcordequiv}. Thus \ref{propcharS3} and \ref{propcharS2} imply
\begin{align}\label{proofpropcharS10}
F_Y \circ F_{Y|\XX}^{-1}(V) &\eqd F_{Y'} \circ F_{Y'|\XX'}^{-1}(V) \quad \text{and} \\
    \label{proofpropcharS1} F_Y \circ q_{Y;\XX}^{\uparrow U}(V) &\eqd F_{Y'} \circ q_{Y';\XX'}^{\uparrow U}(V)\,,
\end{align}
respectively.
Second, \((Y,\XX)\preccurlyeq_{ccx} (Y',\XX')\) implies the marginal constraint \(\overline{\Ran(F_Y)} = \overline{\Ran(F_{Y'})}\) which is equivalent to
\begin{align}\label{proofpropcharS2}
    F_Y \circ q_Y(v) = F_{Y'}\circ q_{Y'}(v) \quad \text{for } \lambda\text{-almost all } v \in (0,1);
\end{align}
see Lemma \ref{lemmargconstraint}.
Since \(Y\eqd q_{Y;\XX}^{\uparrow U}(V)\) and \(Y'\eqd q_{Y';\XX'}^{\uparrow U}(V)\) due to Theorem \ref{lemtrafSI}\,\ref{lemtrafSI0} 
and \(F_{Y|\XX}^{-1}(V)\eqd Y\) and \(F_{Y'|\XX'}^{-1}(V)\eqd Y'\) due to Lemma \ref{lemcondcomrv}, we obtain that \eqref{proofpropcharS10}, \eqref{proofpropcharS1} and \eqref{proofpropcharS2} are equivalent.

To show '\ref{propcharS1} \(\Longrightarrow\) \ref{propcharS3}', assume that \((Y,\XX)\preccurlyeq_{ccx} (Y',\XX')\). Then, by Proposition \ref{propversccx}, there exists a Borel set \(A\subset (0,1)\) with \(\lambda(A)=1\) such that \(P(Y\leq q_Y(v)\mid \XX) \leq_{cx} P(Y'\leq q_{Y'}(v) \mid \XX')\) for all \(v\in A\). Since \(x \mapsto -\min\{x,\alpha\}\) is convex for all \(\alpha\in \R\),  we obtain for \(v \in A\) and \(v' \in (0,1)\) that
\begin{align}
    \nonumber P(F_{Y|\XX}^{-1}(V) \leq q_Y(v) , V\leq v') &= \int_{\R^p} \min\{F_{Y|\XX=\xx}(q_Y(v)), v'\} \de P^{\XX}(\xx) \\
    &\geq \int_{\R^{p'}} \min\{F_{Y'|\XX'=\xx}(q_Y(v)), v'\} \de P^{\XX'}(\xx) \\
    &= P(F_{Y'|\XX'}^{-1}(V) \leq q_Y(v), V\leq v'),
\end{align}
where the equalities are due to Lemma \ref{lemcondcomrv}\,\ref{lemcondcomrv2}. Note that \(F_{Y|\XX}^{-1}(V)\eqd Y\); see Lemma \ref{lemcondcomrv}\,\ref{lemcondcomrv1}. Then, applying Lemma \ref{lemconcordlfs} and Proposition \ref{propconcordequiv}, yields \((F_Y\circ F_{Y|\XX}^{-1}(V),V) \geq_c (F_Y\circ F_{Y'|\XX'}^{-1}(V),V)\).

To show '\ref{propcharS3} \(\Longrightarrow\) \ref{propcharS1}', assume \((F_Y\circ F_{Y|\XX}^{-1}(V),V) \geq_c (F_Y\circ F_{Y'|\XX'}^{-1}(V),V)\). This implies by \eqref{proofpropcharS10} and \eqref{proofpropcharS2} the marginal constraint \(\overline{\Ran(F_Y)} = \overline{\Ran(F_{Y'})}\), and by Proposition \ref{propconcordequiv} and Lemma \ref{lemconcordlfs} that \(P(F_{Y|\XX}^{-1}(V) \leq q_Y(v) , V\leq t) \geq P(F_{Y'|\XX'}^{-1}(V) \leq q_{Y'}(v) , V\leq t)\) for all \(t\in \R\) and \(\lambda\)-almost all \(v\in (0,1)\). 
Consider the stop-loss function \(x\mapsto (x-t)_+ = -\min\{x,t\} + x\). Then, for \(\lambda\)-almost all \(v\in (0,1)\), we obtain
\begin{align*}
    \E \left(P(Y\leq q_Y(v) \mid \XX) -t\right)_+ 
    &= - \E \min\{F_{Y|\XX}(q_Y(v)), t\} + \E(P(Y\leq q_Y(v) \mid \XX)) \\
    &=  - P(F_{Y|\XX}^{-1}(V)\leq q_Y(v) , V\leq t) + P(Y\leq q_Y(v)) \\
    &\leq  - P(F_{Y'|\XX'}^{-1}(V)\leq q_{Y'}(v) , V\leq t) + P(Y'\leq q_{Y'}(v)) \\
    &= \cdots = \E \left(P(Y'\leq q_{Y'}(v) \mid \XX') -t\right)_+
\end{align*}
for all \(t\in \R\), where the second equality is due to Lemma \ref{lemcondcomrv} and the inequality involves the marginal constraint.
Using the marginal constraint, we obtain from the characterization of the convex order by the stop-loss order (see \cite[Theorem 1.5.3 and Theorem 1.5.7]{Muller-2002}) that \(P(Y\leq q_Y(v) \mid \XX) \leq_{cx} P(Y'\leq q_{Y'}(v) \mid \XX')\) for \(\lambda\)-almost all \(v\in (0,1)\).

To show '\ref{propcharS1} \(\Longrightarrow\) \ref{propcharS2}', assume that \((Y,\XX)\preccurlyeq_{ccx} (Y',\XX')\,.\) 
    Then we obtain
    \begin{align}\label{proofpropcharS4}
    \begin{split}
        P(q_{Y;\XX}^{\uparrow U}(V) \leq q_Y(v),U\leq u)
        &= \int_0^u P\left(q^{\uparrow U}_{Y;\XX}(V)\leq q_Y(v) \mid U=s\right) \de s \\
        &= \int_0^u P\left(q^{\uparrow s}_{Y;\XX}(V)\leq q_Y(v) \right) \de s \\
        &= \int_0^u P\left(V\leq F_s(q_Y(v)) \right) \de s 
         = \int_0^u (\eta_{Y|\XX}^v)^*(s) \de s \\
        & \leq  \int_0^u (\eta_{Y'|\XX'}^v)^*(s) \de s \\
        &= \cdots  
        = P(q_{Y';\XX'}^{\uparrow U}(V) \leq q_{Y'}(v),U\leq u)
\end{split}
    \end{align}
    for \(\lambda\)-almost all \(v\in (0,1)\) and for all \(u\in [0,1]\) with equality for \(u=1,\)
    where we use for the inequality the characterization of \(\preccurlyeq_{ccx}\) by the Schur order; see Theorem \ref{thecharSchur}. The first equality is due to the disintegration theorem, the second equality holds by independence of \(U\) and \(V.\) For the third equality, we apply Proposition \ref{Prop.Inverse}\,\ref{Prop.Inverse1} using that \(F_s\) is a distribution function with associated quantile function \(q_{Y;\XX}^{\uparrow s}\); see \eqref{defyu}. The fourth equality follows with \eqref{definreacdf} noting that \(V\) is uniform on \((0,1)\).\\
    Using again \(q_{Y;\XX}^{\uparrow U}(V) \eqd Y\) and \(q_{Y';\XX'}^{\uparrow U}(V) \eqd Y'\) and using the marginal constraint, we apply Lemma \ref{lemconcordlfs} to prove that \eqref{proofpropcharS4} implies 
    \begin{align}\label{proofpropcharS48}
        P\big(F_Y(q_{Y;\XX}^{\uparrow U}(V))\geq v, U\geq u\big) \leq P\big(F_{Y'}(q_{Y';\XX'}^{\uparrow U}(V))\geq v, U\geq u\big)
    \end{align}
    for all \((u,v)\in (0,1)^2\). Now, statement \ref{propcharS2} follows from \eqref{proofpropcharS48} by Proposition \ref{propconcordequiv} where we use the identity in \eqref{proofpropcharS1}. 

    To prove '\ref{propcharS2} \(\Longrightarrow\) \ref{propcharS1}', assume that \((F_Y \circ q_{Y;\XX}^{\uparrow U}(V),U)\leq_{c} (F_{Y'} \circ q_{Y';\XX'}^{\uparrow U}(V),U)\). Applying Proposition \ref{propconcordequiv}, we obtain \eqref{proofpropcharS1} and the inequality \eqref{proofpropcharS48} for all \((u,v)\in [0,1]^2\).
    Using that \eqref{proofpropcharS1} is equivalent to \eqref{proofpropcharS2} which implies the marginal constraint, we obtain from \eqref{proofpropcharS48} by Lemma \ref{lemconcordlfs} that
    \begin{align}
        P\big(q_{Y;\XX}^{\uparrow U}(V) \leq q_Y(v),U\leq u\big) \leq P\big(q_{Y';\XX'}^{\uparrow U}(V) \leq q_{Y'}(v),U\leq u\big)
    \end{align}
    for \(\lambda\)-almost all \(v\in (0,1)\) and for all \(u\in (0,1)\), where \(U\) and \(V\) are independent and uniform on \((0,1)\). By a similar reasoning as in \eqref{proofpropcharS4}, it follows that
    \begin{align*}
         \int_0^u (\eta_{Y|\XX}^v)^*(s) \de s &= P\big(q_{Y;\XX}^{\uparrow U}(V) \leq q_Y(v),U\leq u\big) \\
         & \leq P\big(q_{Y';\XX'}^{\uparrow U}(V) \leq q_{Y'}(v),U\leq u\big) = \int_0^u (\eta_{Y'|\XX'}^v)^*(s) \de s 
    \end{align*}
    for \(\lambda\)-almost all \(v\in (0,1)\) and for all \(u\in (0,1)\).
    Due to the characterization of the conditional convex order by the Schur order in Theorem \ref{thecharSchur}, we obtain \((Y,\XX)\preccurlyeq_{ccx} (Y',\XX')\).
\end{proof}

\subsection{Proof of Theorem \ref{theadderrmod}}

The proof of Theorem \ref{theadderrmod} is based on Theorem \ref{thecharSchur}, which we have already proven, and on Theorem \ref{propmvSchur}, which we prove now. 

\begin{proof}[Proof of Theorem \ref{propmvSchur}.]
    \ref{propmvSchur1}: Due to sufficiency of \(g\) for \(Y|\XX\) and of \(h\) for \(Y'|\XX'\,,\) we obtain from self-equitability of \(\preccurlyeq_{ccx}\) (see Remark \ref{remSchurOrder}\,\ref{remSchurOrder4}) that \((Y,\XX) =_{ccx} (Y,g(\XX)) \preccurlyeq_{ccx} (Y',h(\XX')) =_{ccx} (Y',\XX')\,.\)\\
    \ref{propmvSchur2}: Since \(Y\uparrow_{st} g(\XX)\) and \(Y'\uparrow_{st} h(\XX')\,,\) the (on \(\Ran(F_Y)\times \Ran(F_{g(\XX)})\) resp. \(\Ran(F_{Y'})\times \Ran(F_{h(\XX)})\) uniquely determined) copulas \(C_{Y,g(\XX)}\) and  \(C_{Y',h(\XX')}\) can chosen to be SI on \([0,1]^2\). Hence, by Theorem \ref{propcharS}, \(C_{Y,g(\XX)}\leq_{c} C_{Y',h(\XX')}\) and \(\overline{\Ran(F_Y)} = \overline{\Ran(F_{Y'}})\) imply \((Y, g(\XX))\preccurlyeq_{ccx} (Y', h(\XX'))\,.\) Then, the statement follows from part \ref{propmvSchur1}.\\
    Statement \ref{propmvSchur3} follows from \ref{propmvSchur2} replacing \(g(\XX)\) and \(h(\XX')\) by \(-g(\XX)\) and \(-h(\XX')\,.\)
\end{proof}

For the proof of Theorem \ref{theadderrmod}, we also need the following simple lemma, whose proof is straightforward.

\begin{lemma}[Sign change criterion]\label{lemsignchange}
    Let \(f,g\colon (0,1)\to \R\) be decreasing, Lebesgue-integrable functions with \(\int f \de \lambda = \int g \de \lambda\,.\) Assume that there exists \(t_*\in (0,1)\) such that \(\{f > g\} \subseteq [0,t_*] \subseteq \{f\geq g\}\,.\) Then we have \(\int_0^x g(t) \de t \leq \int_0^x f(t) \de t\) for all \(x\in (0,1)\,.\)
\end{lemma}

\begin{proof}[Proof of Theorem \ref{theadderrmod}:] 
For \(\sigma = 0\,,\) the statement is trivial. Hence, assume \( \sigma > 0\,.\)
    For \(Z:= f(\XX)\) and for \(y,y'\in \R\,,\) we have
    \begin{align}
        P(Y \leq y \mid Z = q_Z(t)) = P( Z + \sigma \varepsilon \leq y \mid Z = q_Z(t)) = F_\varepsilon(\tfrac{y-q_Z(t)}{\sigma})\,,
    \end{align}
    where we use for the second equality independence of \(\varepsilon\) and \(f(\XX)\,.\) Similarly, it holds that \(P(Y'\leq y' \mid Z = q_Z(t)) = F_\varepsilon(\tfrac{y'-q_Z(t)}{\sigma'})\,.\) It follows with Proposition \ref{Prop.Inverse} \eqref{Prop.Inverse1} that
    \begin{align}
    \label{eqtheadderrmod4} 
    & & t &\leq F_Z(\tfrac{\sigma y'-\sigma'y}{\sigma-\sigma'}) \\
    &\Longleftrightarrow & q_Z(t) &\leq \frac{\sigma y'-\sigma'y}{\sigma-\sigma'} \\
    &\Longleftrightarrow & \frac{y-q_Z(t)}{\sigma} &\geq \frac{y'-q_Z(t)}{\sigma'} \\
        & \Longrightarrow &F_\varepsilon(\tfrac{y-q_Z(t)}{\sigma})&\geq F_\varepsilon(\tfrac{y'-q_Z(t)}{\sigma'}),
    \end{align}
    where we use \(\sigma < \sigma'\) for the second equivalence.
    Similarly, we have that \(t > F_Z(\tfrac{\sigma y'-\sigma'y}{\sigma-\sigma'})\) implies \(F_\varepsilon(\tfrac{y-q_Z(t)}{\sigma})\leq F_\varepsilon(\tfrac{y'-q_Z(t)}{\sigma'}).\)
    Using the marginal constraint \(\overline{\Ran(F_Y)} = \overline{\Ran(F_{Y'})},\) we obtain \(P(Y\leq q_Y(v)) = P(Y'\leq q_{Y'}(v))\) for \(\lambda\)-almost all \(v\in (0,1)\,.\) Hence, since \(F_\varepsilon(\tfrac{y-q_Z(t)}{\sigma})\) and \(F_\varepsilon(\tfrac{y'-q_Z(t)}{\sigma'})\) are decreasing in \(t\,,\) it follows with Lemma \ref{lemsignchange} for
    \(\lambda\)-almost all \(v\in (0,1)\) with \(y:= q_Y(v)\) and \(y':=q_{Y'}(v)\) that
    \begin{align}
        \int_0^x P(Y\leq q_Y(v) \mid Z = q_Z(t)) \de t \geq \int_0^x P(Y'\leq q_{Y'}(v) \mid Z = q_Z(t)) \de t \quad \text{for all } x\in [0,1]
    \end{align}
    with equality for \(x=1\,.\)
    Using Theorem \ref{thecharSchur}, it follows that \((Y,f(\XX)) \preccurlyeq_{ccx} (Y',f(\XX))\,.\) Since \(Y\mid f(\XX)\) is sufficient for \(Y\mid \XX\) and, similarly, \(Y'\mid f(\XX)\) is sufficient for \(Y'\mid \XX\), we obtain from Theorem \ref{propmvSchur}\,\ref{propmvSchur1} that \((Y,\XX)\preccurlyeq_{ccx} (Y',\XX')\).
\end{proof}

\subsection{Remaining proofs of Section \ref{secsuffsch}}

\begin{proof}[Proof of Proposition \ref{Prop.Bernoulli}.]
The equivalence of \ref{Prop.Bernoulli1} and \ref{Prop.Bernoulli2} is explained in detail in Example \ref{Ex.Bernoulli} and Figure \ref{figExBernoulli}. The condition \(q = q'\) thereby equals the marginal constraint.
\\
\ref{Prop.Bernoulli3}: \(X\) and \(Y\) are independent if and only if \(p_{00} = (1-p)(1-q)\) if and only if \(p_{01} = p(1-q)\) if and only if \(\alpha=\beta\).
\\
\ref{Prop.Bernoulli4}: \(Y\) perfectly depends on \(X\) if and only if (1) \(p_{01}=0\), \(p=q\), and \(p_{00} = 1-p\) or (2) \(p_{00}=0\), \(p=1-q\), and \(p_{01} = p\) if and only if \(\alpha \wedge \beta = 0\) and \(\alpha \vee \beta = 1\).
\\
\ref{Prop.Bernoulli5}: \(X\) and \(Y\) are comonotone if and only if \(p_{00} = (1-p) \wedge (1-q)\) if and only if \(\alpha = \frac{1-q}{1-p} \wedge 1\).
\\
\ref{Prop.Bernoulli6}: \(X\) and \(Y\) are countermonotone if and only if \(p_{00} = ((1-p)+(1-q)-1) \vee 0\) if and only if \(\alpha = 0 \vee \left(1- \frac{q}{1-p}\right)\).
\end{proof}

\subsection{Remaining proofs of Section \ref{sec52}}

\begin{proof}[Proof of Theorem \ref{propschurnormal}.]
    For \(S\) defined in Example \ref{exdimrednormal}, a reduced random vector of \((Y,\XX)\) is given by \((F_Y(Y),\Phi(S))\); see \eqref{eqexdimrednormal7}. Hence, for \(S'\) defined analogously, we obtain from Theorem \ref{propcharS} that
    \begin{align}\label{proofpropschurnormal2}
        (Y,\XX) \preccurlyeq_{ccx} (Y',\XX') \quad \Longleftrightarrow \quad (F_Y(Y),\Phi(S)) \leq_c (F_{Y'}(Y'),\Phi(S')).
    \end{align}
    Since \((Y,S)\) is bivariate normal with \(\var(S) = 1\), the distribution function of the random vector \((F_Y(Y),\Phi(S))\) is the Gaussian copula with correlation parameter \(\varrho=\sqrt{\Sigma_{Y,\XX}\Sigma_\XX^- \Sigma_{\XX,Y}/\sigma_Y^2}\,;\) see \eqref{eqexdimrednormal2ab}. Similarly, \((F_{Y'}(Y'),\Phi(S'))\) follows a Gaussian copula with parameter \(\varrho' = \sqrt{\Sigma_{Y',\XX'}\Sigma_{\XX'}^- \Sigma_{\XX',Y'}/\sigma_{Y'}^2}\).
    Since the Gaussian copula is increasing in its parameter with respect to the concordance order \cite{Ansari-Rockel-2023}, the right-hand side of \eqref{proofpropschurnormal2} is equivalent to \(\varrho \leq \varrho'\,.\) 
\end{proof}


\subsection{Proofs of Section \ref{secdepmea}}

\begin{proof}[Proof of Theorem \ref{theconsccx}:]
We first prove that \(\xi_\varphi\) is increasing in the conditional convex order. 
Due to the marginal constraint, \((Y,\XX)\preccurlyeq_{ccx} (Y',\XX')\) implies that the normalizing constants of \(\xi_\varphi(Y,\XX)\) and \(\xi_\varphi(Y',\XX')\) are equal.
Then, by the Hardy-Littlewood-Polya theorem (Lemma \ref{lemhlpt}) and the characterization of \(\preccurlyeq_{ccx}\) by the Schur order in Theorem \ref{thecharSchur}, it follows that
\begin{align*}
  \alpha_\varphi \, \xi_\varphi(Y,\XX)
  &   =  \int_{\R} \int_{\R^{p}} \varphi(F_{Y|\XX=\xx}(y)-F_{Y}(y)) \de P^\XX(\uu) \de P^Y(y)
  \\
  &   =  \int_{(0,1)} \int_{(0,1)^{p}} \varphi \big(F_{Y|\XX={\bf q}_\XX(\uu)}(q_Y(v))-F_{Y}(q_Y(v)) \big) \de \lambda^p(\uu) \de \lambda(v) 
  \\
  & = \int_{(0,1)} \int_{(0,1)^{p}} \varphi (\eta_{Y|\XX}^v(\uu) - F_Y(q_Y(v)) ) \de \lambda^p(\uu) \de \lambda(v) \\
  & \leq \int_{(0,1)} \int_{(0,1)^{p'}} \varphi (\eta_{Y'|\XX'}^v(\uu) - F_{Y'}(q_{Y'}(v)) ) \de \lambda^{p'}(\uu) \de \lambda(v) 
  \\
  & = \cdots   =  \alpha_\varphi \, \xi_\varphi(Y',\XX')
\end{align*}
where we use that \(z \mapsto \varphi(z-c)\) is convex whenever \(\varphi\) is convex and that \(\overline{\Ran(F_Y)} = \overline{\Ran(F_{Y'})}\).
By a similar reasoning,  \(\Lambda_\varphi\) is increasing in \(\preccurlyeq_{ccx}\).

To prove assertion \ref{theconsccx1}, we use that \((Y,\XX)\preccurlyeq_{ccx}(Y',\XX')\) implies \eqref{theconsccx0} as already shown. Since the conditional convex order satisfies Axiom \ref{axdepord2}, \((Y,\XX)\) is a minimal element in \(\preccurlyeq_{ccx}\) if \(Y\) and \(\XX\) are independent. In this case, we have \(\xi_\varphi(Y,\XX) = 0 = \Lambda_\varphi(Y,\XX)\). Similarly, the conditional convex order satisfies Axiom \ref{axdepord3}, and thus \((Y,\XX)\) is a maximal element in \(\preccurlyeq_{ccx}\) if \(Y\) perfectly depends on \(\XX\). In this case, we have \(\xi_\varphi(Y,\XX) = 1 = \Lambda_\varphi(Y,\XX)\), which is also immediate from the definition of the functionals.

The zero-independence property in \ref{theconsccx2} follows from the fact that \(Y\) and \(\XX\) are independent if and only if \(F_{Y|\XX=\xx} = F_{Y}\) for \(P^\XX\)-almost every \(\xx \in \R^p\). This condition is equivalent to requiring that the inner integrals in the definitions of \(\xi_\varphi\) and \(\Lambda_\varphi\) vanish \(P^\XX\)-almost everywhere and \(P^\XX \otimes P^\XX\)-almost everywhere, respectively, since \(\varphi\) is strictly convex at $0$ with $\varphi(0) = 0$.
\\
The max-functionality property of \(\Lambda_\varphi\) is a direct extension of \cite[Proof of Theorem 3.2]{Ansari-LFT-2023} to a multivariate vector \(\XX=(X_1,\ldots,X_p)\).
To prove max-functionality for \(\xi_\varphi\), it remains to show that \(\xi_\varphi(Y,\XX) = 1\) implies \(Y = f(\XX)\) almost surely. 
Therefore, define \(f_y(\xx) := F_{Y|\XX=\xx}(y)\) and \(a := F_{Y}(y)\) and set
\begin{align*}
  c_{1-a} & := f_y(\xx) - \min\{f_y(\xx), 1-f_y(\xx)\} \in [0,1]
  \\
  c_{- a}  & := (1-f_y(\xx)) - \min\{f_y(\xx), 1-f_y(\xx)\} \in [0,1]
  \\
  c_{b}   & := 2 \, \min\{f_y(\xx), 1-f_y(\xx)\} \in [0,1]\,.
\end{align*}
Then \(c_{1-a} + c_{-a} + c_{b} = 1\) and convexity of \(\varphi\) gives
\begin{align*}
  \varphi(f_y(\xx)-a)
  &   =  \varphi \left( (1-a) \, c_{1-a} + (-a) \, c_{-a} + \frac{1-2a}{2} c_{b} \right)  
  \\
  & \leq \varphi(1-a) \, c_{1-a} + \varphi(-a) \, c_{-a} + \varphi\left(\frac{1-2a}{2}\right) \, c_{b} 
  \\
  &   =  \varphi(1-a) \, f_y(\xx) + \varphi(-a) \, (1-f_y(\xx)) 
  \\ 
  & \qquad - \min\{f_y(\xx), 1-f_y(\xx)\} \, \left( \underbrace{\varphi(1-a) + \varphi(-a) - 2\, \varphi\left(\frac{1-2a}{2}\right)}_{=: c(a)} \right)
  \\
  & 
      =  \int_{\R} \varphi \left( \mathds{1}_{\{y_1 \leq y\}} - a \right) \de P^{Y|\XX=\xx}(y_1)
      - \min\{f_y(\xx), 1-f_y(\xx)\} \, c(a) \,,
\end{align*}
where \(c(a) > 0\) follows from strict convexity of \(\varphi\) at \(0\) with \(\varphi(0)=0\). 
Consequently, 
\begin{align*}
  1 
  &   =  \xi_\varphi(Y,\XX) 
      =  \alpha_\varphi^{-1} \, \int_{\R} \int_{\R^{p}} \varphi(f_y(\xx)-a) \de P^\XX(\xx) \de P^Y(y)
  \\
  & \leq \alpha_\varphi^{-1} \, \int_{\R} \int_{\R^{p}} \int_{\R} \varphi \left( \mathds{1}_{\{y_1 \leq y\}} - a \right) \de P^{Y|\XX=\xx}(y_1) \de P^\XX(\xx) \de P^Y(y)
  \\
  & \qquad - \alpha_\varphi^{-1} \, \int_{\R} \int_{\R^{p}} \underbrace{c(a)}_{> 0} \cdot \min\{f_y(\xx), 1-f_y(\xx)\} \de P^\XX(\xx) \de P^Y(y)
  \\
  & \leq \alpha_\varphi^{-1} \, \int_{\R} \int_{\R} \varphi \left( \mathds{1}_{\{y_1 \leq y\}} - a \right) \de P^{Y}(y_1) \de P^Y(y)
      =  1\,.
\end{align*}
Hence, there exists some Borel set \(B \subseteq \R^p\) with \(P^\XX(B) = 1\) such that \(\min\{f_y(\xx), 1-f_y(\xx)\} = 0\) for all \(\xx \in B\), implying that \(f_y(\xx) = F_{Y|\XX=\xx}(y) \in \{0,1\}\) for all \(\xx \in B\). Now, proceeding as in the proof of Theorem 3.2 in \cite{Ansari-LFT-2023} yields the claimed max-functionality property of \(\xi_\varphi\).


For verifying \ref{theconsccx3}, it remains to show that \(\Lambda_\varphi(Y,\XX) = \Lambda_\varphi(Y,(\XX,\ZZ))\) implies conditional independence of \(Y\) and \(\ZZ\) given \(\XX\).
Therefore, define again \(f_y(\xx):= F_{Y|\XX=\xx}(y)\) and \(g_y(\xx,\zz):= F_{Y|\XX=\xx,\ZZ=\zz}(y)\,.\)
Using \(\Lambda_\varphi(Y,\XX) = \Lambda_\varphi(Y,(\XX,\ZZ))\) and Jensen's inequality, it follows that
\begin{align*}
    & \beta_\varphi \Lambda_\varphi (Y,\XX) 
         = \int \int \varphi(f_y(\xx_1) - f_y(\xx_2)) \de P^\XX \otimes P^\XX(\xx_1,\xx_2) \de P^Y(y)\\
        & = \int \int \varphi\left(\int g_y(\xx_1,\zz_1)-g_y(\xx_2,\zz_2) \de P^{\ZZ|\XX=\xx_1}\otimes P^{\ZZ|\XX=\xx_2}(\zz_1,\zz_2) \right) \de P^\XX\otimes P^\XX(\xx_1,\xx_2) \de P^Y(y)\\
        & \leq \int \int \int \varphi( g_y(\xx_1,\zz_1) - g_y(\xx_2,\zz_2)) \de P^{\ZZ|\XX=\xx_1}\otimes P^{\ZZ|\XX=\xx_2}(\zz_1,\zz_2)  \de P^\XX\otimes P^\XX(\xx_1,\xx_2) \de P^Y(y)\\
        & = \beta_\varphi \Lambda_\varphi (Y,(\XX,\ZZ)) = \beta_\varphi \Lambda_\varphi (Y,\XX)\,,
\end{align*}
    and thus the inequality becomes an equality. 
    Define \(\ZZ_\xx^i:= q_{\ZZ|\XX=\xx}(\UU_i)\) for \(i\in \{1,2\}\) and \(\UU_1,\UU_2\) i.i.d. uniformly distributed on \([0,1]^r,\) where \(r\) is the dimension of \(\ZZ\).
    Hence, there exist a \(P^Y\)-null set \(N\) and for each \(y\in N^c\) some Borel set \(G_y\) with \(P^\XX\otimes P^\XX(G_y) = 1\) such that 
    \begin{align}
        \varphi(f_y(\xx_1)-f_y(\xx_2)) = \E \varphi\left( g_y(\xx_1,\ZZ_{\xx_1}^1)-g_y(\xx_2,\ZZ_{\xx_2}^2)\right)
    \end{align}
    for all \((\xx_1,\xx_2)\in G_y\) and \(y\in N^c\,.\) It follows that 
    \begin{align*}
        g_y(\xx_1,\ZZ_{\xx_1}^1) - g_y(\xx_2,\ZZ_{\xx_2}^2) = c \in [-1,1] \quad \text{a.s.}
    \end{align*}
    for such \((\xx_1,\xx_2)\) and \(y\) since \(\varphi\) is strictly convex. 
    It follows that 
    \begin{align*}
      c = \E(g_y(\xx_1,\ZZ_{\xx_1}^1) - g_y(\xx_2,\ZZ_{\xx_2}^2))
        = f_y(\xx_1)-f_y(\xx_2) \quad \text{a.s.}
    \end{align*}
    and hence 
    \begin{align}\label{prothelambdaphi5}
      f_y(\xx_1) - g_y(\xx_1,\ZZ_{\xx_1}^1) 
      = f_y(\xx_2) - g_y(\xx_2,\ZZ_{\xx_2}^2) \quad \text{a.s.}
    \end{align}
    for all \((\xx_1,\xx_2)\in G_y\) and \(y\in N^c\,.\) Integrating out in \eqref{prothelambdaphi5} then gives
    \begin{align*}
      f_y(\xx_1) - g_y(\xx_1,\zz_1) 
      = \int f_y(\xx) - g_y(\xx,\zz) \de P^{\ZZ|\XX=\xx}(\zz)
      = 0\,,
    \end{align*}
    so that \(F_{Y|\XX=\xx,\ZZ=\zz}(y) = g_y(\xx,\zz) = f_y(\xx) = F_{Y|\XX=\xx}(y)\) for \(P^{(\XX,\ZZ)}\)-almost all \((\xx,\zz)\) and \(P^Y\)-almost every $y$. In other words, \(Y\) and \(\ZZ\) are conditionally independent given \(\XX\).  
    Finally, a similar line of arguments yields the equivalence of \(\xi_\varphi(Y,\XX) = \xi_\varphi(Y,(\XX,\ZZ))\) and conditional independence of \(Y\) and \(\ZZ\) given \(\XX\).
\end{proof}

For the proof of Theorem \ref{theOT}, we make use of the \emph{supermodular order} \(\UU \leq_{sm} \VV\) which is defined for bounded, bivariate random vectors \(\UU= (U_1,U_2)\) and \(\VV = (V_1,V_2)\) by
\begin{align}
     \E f(\UU) \leq \E f(\VV) \quad \text{for all supermodular functions } f.
\end{align}
 Recall that \(f\colon \R^2\to \R\) is said to be \emph{supermodular} (\emph{submodular}) if \(f(\xx) + f(\yy) \leq\) \((\geq)\) \(f(\xx\wedge \yy) + f(\xx\vee \yy)\) for all \(\xx,\yy\in \R^2\). 
For bivariate random vectors, the supermodular order and the concordance order are equivalent. For bivariate random vectors from the same Fr\'{e}chet class, the supermodular order is even equivalent with the pointwise order of distribution functions as follows; see \cite[Theorem 2.5]{Muller-2000}.

\begin{lemma}[Characterization of supermodular order]\label{lemcharsm}
    If \(\UU = (U_1,U_2)\) and \(\VV = (V_1,V_2)\) have equal marginal distributions, then the following statements are equivalent. 
    \begin{enumerate}[(i)]
        \item \(\UU\leq_{sm} \VV\),
        \item \(P(U_1 \leq u_1, U_2\leq u_2) \leq P(V_1\leq u_1, V_2\leq u_2)\) for all \((u_1,u_2)\in \R^2\).
    \end{enumerate}
\end{lemma}

Next, recall that, for submodular cost functions \(c\), the comonotone coupling solves the optimal transport problem 
\begin{align}\label{eqopttransprob}
    \cW_c(\nu,\nu') := \inf_{\gamma\in \Pi(\nu,\nu')} \int c(y,y') \de \gamma(y,y')
\end{align}
as follows (see \cite[Theorem 3.1.2]{Rachev-Rueschendorf-1998}).

\begin{lemma}[Comonotone Coupling]\label{lemsolOT}
    Let \(\nu\) and \(\nu'\) be distributions on \(\R\), and let \(Z\sim \nu\) and \(Z'\sim\nu'\) be random variables. Assume that the cost function \(c\) is submodular. Then, for \(V\) uniform on \((0,1)\), the comonotone coupling \((F_Z^{-1}(V),F_{Z'}^{-1}(V))\) solves the optimal transport problem \eqref{eqopttransprob}, i.e. \(\cW_c(\nu,\nu') = \E c(F_Z^{-1}(V),F_{Z'}^{-1}(V))\).
\end{lemma}

We also make use of the following quantile-based information monotonicity.

\begin{lemma}[A variant of information monotonicity]\label{lemcondcomcx}
    For a random variable \(W\) and random vector \((\XX,\ZZ)\), let \(V\sim U(0,1)\) be independent of \((\XX,\ZZ)\). Then \(F_{W|\XX}^{-1}(V) - F_W^{-1}(V) \leq_{cx} F_{W|(\XX,\ZZ)}^{-1}(V) - F_W^{-1}(V)\).
\end{lemma}

\begin{proof}
    By information monotonicity of the conditional convex order, we have \((W,\XX)\preccurlyeq_{ccx} (W,(\XX,\ZZ))\), see Theorem \ref{thefundpropS}.
    Then we obtain from Lemma \ref{lemcondcomrv} for all \(y,y'\in \R\) that 
    \begin{align*}
       P(F_{W|\XX}^{-1}(V) \leq y, &F_W^{-1}(V)\leq y') = \E \min\{F_{W|\XX}(y),F_W(y')\} \\
       &\geq \E \min\{F_{W|(\XX,\ZZ)}(y),F_{W}(y')\} 
       = P(F_{W|(\XX,\ZZ)}^{-1}(V) \leq y, F_{W}^{-1}(V)\leq y')\},
    \end{align*}
    where the inequality follows from the version of the conditional convex order in Remark \ref{remccx}\,\ref{remccxa}.
    Since \(F_{W|\XX}^{-1}(V) \eqd F_{W|(\XX,\ZZ)}^{-1}(V) \eqd F_{W}^{-1}(V)\) by Lemma \ref{lemcondcomrv}, it follows from Lemma \ref{lemcharsm} that 
    \begin{align}
        (F_{W|\XX}^{-1}(V), F_{W}^{-1}(V)) \geq_{sm} (F_{W|(\XX,\ZZ)}^{-1}(V),F_{W}^{-1}(V)).
    \end{align}
    Now, the statement follows from \cite[Theorem 9.A.18]{Shaked-2007}.  
\end{proof}

\begin{proof}[Proof of Theorem \ref{theOT}.]
First we show that the denominator in \eqref{defWcc} is positive. Therefore, let \(W\) and \(W'\) be independent random variables both with distribution \(\nu\). Applying Jensen's inequality twice, we obtain
\begin{align*}
    \int_\R\int_\R c(y,y') \de \nu(y)\de \nu(y') = \E h(W'-W) \geq \int h(y' - \E W) \de \nu(y') > h(\E W' - \E W) = 0,
\end{align*}
where we use for the strict inequality on the one hand that \(h\) is strictly convex at \(0\). On the other hand, we use that \(Y\), and hence \(W \eqd F_Y(Y)\), is non-degenerate and that \(W\eqd W'\). The last equality follows from the assumption \(h(0) = 0\).

For the rest of the proof, we will only consider the numerator of \eqref{defWcc} because the denominators of \(\dW_c(Y,\XX)\) and \(\dW_c(Y',\XX')\) coincide under the marginal constraint.

To prove \eqref{eqtheOT}, we need to show that the numerator of \eqref{defWcc} is \(\preccurlyeq_{ccx}\)-increasing. We abbreviate \(W:= F_Y(Y)\) and \(W':=F_{Y'}(Y')\). Note that \(W\eqd W'\) by the marginal constraint. 
Since the conditional convex order satisfies Axiom \ref{axdepord7} on distributional invariance, \((Y,\XX) \preccurlyeq_{ccx} (Y',\XX')\) implies \((W,\XX) \preccurlyeq_{ccx} (W',\XX')\). From Theorem \ref{propcharS}, we then obtain \((F_W \circ F_{W|\XX}^{-1}(V), V) \geq_{c} (F_{W'} \circ F_{W'|\XX'}^{-1}(V),V)\). Using that the concordance order is invariant under increasing transformations of the components and using the invariance in Proposition \ref{Prop.Inverse}\,\ref{Prop.Inverse1b}, we obtain from Lemma \ref{lemcharsm} that 
\begin{align}\label{eqtheOT2}
    (F_{W|\XX}^{-1}(V),F_{W}^{-1}(V)) \geq_{sm} (F_{W'|\XX'}^{-1}(V),F_{W'}^{-1}(V)).
\end{align}
Since the function \((y,y')\mapsto - h(y'-y)\) is supermodular whenever \(h\) is convex, \eqref{eqtheOT2} 
yields 
\begin{align*}
\begin{split}
    \int_{\R^p} \cW_c(\pi_\xx,\nu) \de \mu(\xx) &= \E h(F_{W|\XX}^{-1}(V) - F_{W}^{-1}(V)) \\
    &\leq \E h(F_{W'|\XX'}^{-1}(V) - F_{W'}^{-1}(V)) = \int_{\R^{p'}} W_c(\pi'_\xx,\nu') \de \mu'(\xx),
\end{split}
\end{align*}
where we use \eqref{eqrepW_c} for the first and, similarly, for the second equality.
Here, \(\nu'\) and \(\mu'\) denote the distribution of \(W'\) and \(\XX'\), respectively, and \(\pi'_\xx\) refers to the conditional distribution \(W'\mid \XX' =\xx\).

To prove statement \ref{theOT1}, let \(Y\) and \(\XX\) be independent. Then \(\pi_\xx = \nu\) for \(\mu\)-almost all \(\xx\in \R^p\). Hence, \(\cW_c(\pi_\xx,\nu) = \cW_c(\nu,\nu) = 0\) for \(\mu\)-almost all \(\xx\) and thus \(\dW(Y,\XX) = 0\). Since \(\overrightarrow{\cW}\) is \(\preccurlyeq_{ccx}\)-increasing due to \eqref{eqtheOT} and since independent random vectors are minimal elements in the conditional convex order due to Theorem \ref{thefundpropS}, we obtain \(\dW(Y',\XX') \geq 0\) for all \((Y',\XX')\). It remains to show that \(\dW(Y',\XX') \leq 1\). But this is a consequence of
\begin{align}\label{eqtheOT6}
    \int_{\R^p} \cW_c(\pi_\xx,\nu) \de\mu( \xx) &\leq \int_{\R^p} \int_\R \int_\R c(y,y') \de\pi_\xx( y) \de\nu( y') \de\mu( \xx) \\
   \nonumber &=  \int_\R \int_\R c(y,y') \de\nu( y) \de\nu( y').
\end{align}

To prove in statement \ref{theOT2} the characterization of independence, it remains to show that \(\dW(Y,\XX) = 0\) implies independence of \(Y\) and \(\XX\).
Therefore, let \(W=F_Y(Y)\) and \(W'=F_{Y'}(Y')\) and consider
\begin{align}\label{eqtheOT4}
    0 = \int_{\R^p} \cW_c(\pi_\xx,\nu) \de \mu(\xx) = \E h(F_{W|\XX}^{-1}(V) - F_W^{-1}(V)) 
\end{align}
where the second equality follows with Lemma \ref{lemsolOT}.
We show that \(\zeta:= F_{W|\XX}^{-1}(V) - F_W^{-1}(V) = 0\) almost surely. Then \(W\) and \(\XX\), and thus \(Y\) and \(\XX\), are independent. To this end, assume that \(P(\zeta\ne 0)>0\). Since \(F_{W|\XX}^{-1}(V) \eqd F_W^{-1}(V)\) due to Lemma \ref{lemcondcomrv}, we have \(\E \zeta = 0\). Using that \(h\) is strictly convex in \(0\), it follows that \(\E h(\zeta) > h(0) = 0\). But this contradicts \eqref{eqtheOT4}. 
The proof that \(\dW_c(Y,\XX) = 1\) characterizes perfect dependence of \(Y\) on \(\XX\) follows in the same line as the proof of \cite[Theorem 2.2(iii)]{wiesel2022} noting that \(h(y'-y) + h(y-y') > h(y-y) + h(y'-y')\) for \(y\ne y'\) since \(h\) is convex and strictly convex in \(0\).

Statement \ref{theOT3} on information monotonicity is a direct consequence of \eqref{eqtheOT}, together with information monotonicity of the conditional convex order.

To prove statement \ref{theOT4}, consider the differences \(\zeta:= F_{F_Y(Y)|\XX}^{-1}(V)-F_{F_Y(Y)}^{-1}(V)\) and \(\zeta':= F_{F_Y(Y)|(\XX,\ZZ)}^{-1}(V)-F_{F_Y(Y)}^{-1}(V)\) for \(V\sim U(0,1)\) independent of \((\XX,\ZZ)\). 
Then, by Lemma \ref{lemcondcomcx}, we have \(\zeta\leq_{cx} \zeta'\). Hence, by Strassen's theorem, there exist versions \(\tilde{\zeta} \eqd \zeta\) and \(\tilde{\zeta}'\eqd \zeta'\) such that \(\E[\tilde{\zeta}'\mid \tilde{\zeta}] = \tilde{\zeta}\) almost surely; see \cite[Corollary 3.36]{Ru-2013}. Since \(h\) is convex, we obtain for \(\alpha := \int_\R \int_\R c(y,y') \de \nu(y) \de \nu(y')\) that
\begin{align}
    \nonumber \alpha\, \dW_c(Y,\XX) = \E h(\tilde{\zeta}) &= \E h(\E[\tilde{\zeta}'\mid \tilde{\zeta}]) = \int_\R h(\E[\tilde{\zeta}'\mid \tilde{\zeta}=w]) \de P^{\tilde{\zeta}}(w)  \\
    \label{eqtheOT8}& \leq \int_\R \E[h(\tilde{\zeta}')\mid \tilde{\zeta}=w] \de P^{\tilde{\zeta}}(w) = \E h(\tilde{\zeta}') = \alpha\, \dW_c(Y,(\XX,\ZZ)),
\end{align}
where the first and last equality are due to \eqref{eqrepW_c} using that \((y,y')\mapsto h(y'-y)\) is submodular because \(h\) is convex. The above inequality is due to Jensen's inequality for conditional expectations. Now, assume that \(Y\) and \(\ZZ\) are conditionally independent given \(\XX\). Then we have \(\zeta\eqd \zeta'\). Hence, using the martingale property, the conditional distribution \(\tilde{\zeta}'\mid \tilde{\zeta}=w\) is degenerate for \(P^{\zeta}\)-almost all \(w\) and the inequality in \eqref{eqtheOT8} becomes an equality. Conversely, if \(Y\) and \(\ZZ\) are not conditionally independent given \(\XX\), then there exists a Borel set \(A\) with \(P(\tilde{\zeta}\in A)>0\) such that \(\tilde{\zeta}'\mid \tilde{\zeta}=w\) is non-degenerate for all \(w\in A\). Then, for \(h\) strictly convex, the inequality in \eqref{eqtheOT8} is strict.
\end{proof}

\begin{proof}[Proof of Theorem \ref{cordepmeas}.] 
    \eqref{cordepmeas1}: Due to Theorem \ref{propcharS}, we know that \((Y,\XX)\preccurlyeq_{ccx} (Y',\XX')\) implies \((F_Y \circ q_{Y;\XX}^{\uparrow U}(V),U) \leq_c (F_{Y'}\circ q_{Y';\XX'}^{\uparrow U}(V),U)\) and thus \(\sR_\mu(Y,\XX) = \mu(F_Y \circ q_{Y;\XX}^{\uparrow U}(V),U) \leq \mu(F_{Y'} \circ q_{Y';\XX'}^{\uparrow U}(V),U) = \sR_\mu (Y',\XX')\). \\
    \ref{cordepmeas2a} \& \ref{cordepmeas2}: \(\XX\) and \(Y\) are independent if and only if \(q_{Y;\XX}^{\uparrow U}(V)\) and \(U\) are independent; see Remark \ref{remcharccx}\,\ref{remcharccx3}.
    Since \(\mu\) characterizes independence on \(\cR^\uparrow\), the latter is equivalent to \(\sR_\mu(Y,\XX) = \mu(F_Y\circ q_{Y;\XX}^{\uparrow U}(V),U) = 0\).
    \(Y\) perfectly depends on \(\XX\) if and only if \(q_{Y;\XX}^{\uparrow U}(V)\) and \(U\) are comonotone; see Remark \ref{remcharccx}\,\ref{remcharccx3}. Since \(\mu\) characterizes comonotonicity on \(\cR^\uparrow\), the latter is equivalent to \(\sR_\mu(Y,\XX) = \mu(F_Y\circ q_{Y;\XX}^{\uparrow U}(V),U) = 1\). 
    Assertion \ref{cordepmeas2a} then follows from \eqref{cordepmeas1}.\\
    \ref{cordepmeas4}: By Theorem \ref{thefundpropS}, we have \((Y,\XX) \preccurlyeq_{ccx} (Y,(\XX,\ZZ))\). Hence, \eqref{cordepmeas1} implies the statement.\\
    \ref{cordepmeas5}: Assume that \(\sR_\mu(Y,\XX) = \sR_\mu(Y,(\XX,\ZZ))\). By information monotonicity of \(\preccurlyeq_{ccx}\) and due to Theorem \ref{propcharS}, 
    we know that \((F_Y \circ q_{Y;\XX}^{\uparrow U}(V),U)\leq_{c} (F_{Y} \circ q_{Y;(\XX,\ZZ)}^{\uparrow U}(V),U)\). This implies \(\sR_\mu(Y,\XX) = \mu(F_{Y} \circ q_{Y;\XX}^{\uparrow U}(V),U) \leq \mu(F_{Y} \circ q_{Y;(\XX,\ZZ)}^{\uparrow U}(V),U) = \sR_\mu(Y,(\XX,\ZZ)) = \sR_\mu(Y,\XX)\) and thus \(\mu(F_{Y} \circ q_{Y;\XX}^{\uparrow U}(V),U) = \mu(F_{Y} \circ q_{Y;(\XX,\ZZ)}^{\uparrow U}(V),U)\). Since \(\mu\) is strictly increasing on \(\cR^\uparrow\), it follows that \((F_{Y} \circ q_{Y;\XX}^{\uparrow U}(V),U) =_c (F_{Y} \circ q_{Y;(\XX,\ZZ)}^{\uparrow U}(V),U)\) and thus, by Theorem \ref{propcharS}, \((Y,\XX) =_{ccx} (Y,(\XX,\ZZ))\). However, the latter means that \(Y\) and \(\ZZ\) are conditionally independent given \(\XX\). For the reverse direction, assume that \(Y\) and \(\ZZ\) are conditionally independent given \(\XX\). Then \((Y,\XX) =_{ccx} (Y,(\XX,\ZZ))\) and thus, by Theorem \ref{propcharS}, \((F_{Y} \circ q_{Y;\XX}^{\uparrow U}(V),U) \eqd (F_{Y} \circ q_{Y;(\XX,\ZZ)}^{\uparrow U}(V),U)\). This implies \(\sR_\mu(Y,\XX) = \mu(F_{Y} \circ q_{Y;\XX}^{\uparrow U}(V),U) = \mu(F_{Y} \circ q_{Y;(\XX,\ZZ)}^{\uparrow U}(V),U) = \sR_\mu(Y,(\XX,\ZZ))\).
\end{proof}




\subsection{Proofs of Section \ref{seccompS}}

\begin{proof}[Proof of Proposition \ref{rembivSchur}.]
Assume that \((Y,X)\preccurlyeq_{ccx} (Y',X')\). Since \(Y\eqd Y'\), we obtain from a variant of Theorem \ref{propcharS} that this is equivalent to 
\begin{align}\label{proofrembivSchur1}
    (q_{Y;X}^{\uparrow U}(V),U) \leq_c (q_{Y';X'}^{\uparrow U}(V),U)
\end{align} 
for \(U,V\) independent and uniform on \((0,1)\).
Since \(Y\uparrow_{st} X\) it follows that \(P(Y\leq y\mid X=q_X(u)) = F_u(y)\) and thus 
\begin{align}\label{proofrembivSchur2}
       q_{Y;X}^{\uparrow U}(V) = F_U^{-1}(V) = F_{Y|X=q_X(U)}^{-1}(V),
\end{align}
with \(q_{Y;X}^{\uparrow u}(v)\) defined by \eqref{defyu}.
Similarly, we obtain \(q_{Y';X'}^{\uparrow U}(V) = F_{Y'|X'=q_{X'}(U)}^{-1}(V)\). Now, observe that 
\begin{align}\label{proofrembivSchur3}
    (Y,X) \eqd (F_{Y|X=q_X(U)}^{-1}(V),q_X(U)) \quad \text{and} \quad (Y',X') \eqd (F_{Y'|X'=q_{X'}(U)}^{-1}(V),q_{X'}(U));
\end{align}
see \eqref{eqpropqt}. Since the concordance order is invariant under increasing transformations, we obtain from \(q_X = q_{X'}\) with \eqref{proofrembivSchur1} and \eqref{proofrembivSchur2} that \((Y,X)\leq_c (Y',X')\). \\
The reverse direction is a special case of Theorem \ref{propmvSchur}\,\ref{propmvSchur2}.
\end{proof}


\subsection{Proofs of Appendix \ref{appA}}

\begin{proof}[Proof of Proposition \ref{Prop.Inverse}.]
    For statement \ref{Prop.Inverse1}, see e.g. \cite{Embrechts2013}; for statements \ref{Prop.Inverse1b} -- \ref{Prop.Inverse4}, see \cite[Lemma A.1]{Ansari-2021}.
\end{proof}

\begin{proof}[Proof of Lemma \ref{lemmargconstraint}]
    The equivalence of \ref{lemmargconstraint1} and \ref{lemmargconstraint3} is given in \cite[Proposition 2.14]{Ansari-2021}. The equivalence between \ref{lemmargconstraint2} and \ref{lemmargconstraint3} is a direct consequence of Proposition \ref{Prop.Inverse}\,\ref{Prop.Inverse2} and \ref{Prop.Inverse3}.
\end{proof}

\begin{proof}[Proof of Proposition \ref{propversccx}]
    The equivalence of \ref{propversccx1} and \ref{propversccx2} as well as the equivalence of \ref{propversccx3} and \ref{propversccx4} follow from the invariance properties of the convex order in Lemma \ref{lemminmaxcx}\,\ref{lemminmaxcx3}.
    
    To show that \ref{propversccx2} implies \ref{propversccx3}, we consider two cases for \(v, v_n \in (0,1)\) with \(q_Y(v) < q_Y(v_n)\) for all \(n\).
    In either case, we obtain for \( S_n := P(Y< q_Y(v_n) \mid \XX)\) and \(S := P(Y\leq q_Y(v) \mid \XX)\) that
    \(\E |S_n-S|
        = \E | P(q_Y(v) < Y < q_Y(v_n) \mid \XX) |
        = P(q_Y(v) < Y < q_Y(v_n))
    \). In both cases, we specify \(v_n\) such that 
    \begin{align} \label{proofpropversccx1}
        \lim_{n \to \infty} \E |S_n-S|
        & = P \left( \bigcap_{n\in \N} \{Y < q_Y(v_n)\} \right) - P(Y \leq q_Y(v)) = 0\,.
    \end{align}
    In the first case, let \(v\in \mathrm{int}(\overline{\Ran(F_Y)}) = \mathrm{int}(\overline{\Ran(F_{Y'})})\) and \(v_n\downarrow v\) (i.e., \(v_n > v\) for all \(n\) and \(v_n\to v\) as \(n\to \infty\)).  
    Then, \(y < q_Y(v_n)\) for all \(n\) implies \(y\leq q_Y(v)\) or \(y\in (q_Y(v),q_Y^+(v))\), where \(q_Y^+(t):=\sup\{x\mid F(x) \leq t\}.\) Since \(\{Y\in (q_Y(v),q_Y^+(v)]\}\) is a \(P\)-null set for this choice of \(v\), we obtain \eqref{proofpropversccx1}.\\
    In the second case, let \(v\in (0,1)\setminus \overline{\Ran(F_Y)} = (0,1)\setminus \overline{\Ran(F_{Y'})}\) and \(v_n\downarrow v_0:= F_Y(q_Y(v))\). Then, we have that \(q_Y(v) = q_Y(v_0)\). If \(F_Y\) is strictly increasing on \([q_Y(v_0),q_Y(v_0)+\varepsilon)\) for some \(\varepsilon>0\), then \(q_Y(v_n) \downarrow  q_Y(v_0)\). This implies \(\bigcap_{n\in \N} \{q_Y(v_n) > Y\} = \{q_Y(v) \geq Y\}\), which yields \eqref{proofpropversccx1}. 
    If \(F_Y\) is constant on \([q_Y(v_0),q_Y^+(v_0)]\) and continuous at \(q_Y^+(v_0)\), then \(q_Y(v_n)\downarrow q_Y^+(v_0)\). This implies \(\bigcap_{n\in \N} \{q_Y(v_n) > Y\} = \{q_Y^+(v_0)\geq Y\} = \{q_Y(v_0)\geq Y\} \cup \{Y\in (q_Y(v_0),q_Y^+(v_0)]\}\). Since \(\{Y\in (q_Y(v_0),q_Y^+(v_0)]\}\) is a \(P\)-null set, we again obtain the convergence in \eqref{proofpropversccx1}. 
    If \(F_Y\) is constant on \([q_Y(v_0),q_Y^+(v_0))\) and has a jump at \(q_Y^+(v_0)\), then \(\bigcap_{n\in \N} \{q_Y(v_n) > Y\} = \{q_Y^+(v_0) > Y\} = \{q_Y(v_0)\geq Y\} \cup \{Y\in (q_Y(v_0),q_Y^+(v_0))\}\). Once again, using that \(\{Y\in (q_Y(v_0),q_Y^+(v_0))\}\) is a \(P\)-null set, \eqref{proofpropversccx1} follows. \\
    The convergence in \eqref{proofpropversccx1} implies \(S_n \xrightarrow{~d~} S\) and \(\E|S_n|\to \E|S|\). Similar considerations yield \(S'_n := P(Y'< q_{Y'}(v_n) \mid \XX') \xrightarrow{~d~} S := P(Y'\leq q_{Y'}(v) \mid \XX')\) and \(\E|S'_n|\to \E|S'|\). Using that \(S_n\leq_{cx}S_n'\), we obtain from Lemma \ref{lemminmaxcx}\,\ref{lemminmaxcx4} that \(S\leq_{cx} S'\). This proves \ref{propversccx3}.
    
    It remains to show that \ref{propversccx3} implies \ref{propversccx2}. 
    Therefore, assume that \(P(Y\leq q_Y(u) \mid \XX) \leq_{cx} P(Y'\leq q_{Y'}(u)\mid \XX')\) for all \(u\in (0,1)\setminus N\) for some \(\lambda\)-null set \(N\subset (0,1)\). 
    We distinguish between two cases for \(v, v_n \in (0,1)\) with \(q_Y(v_n) < q_Y(v)\) for all \(n\). 
    Setting \( S_n := P(Y \leq q_Y(v_n) \mid \XX)\) and \(S := P(Y < q_Y(v) \mid \XX)\), we show that
    \begin{align} \label{proofpropversccx2}
        \lim_{n \to \infty} \E |S-S_n|
        & = P(Y < q_Y(v)) - P \left( \bigcup_{n\in \N} \{Y \leq q_Y(v_n)\} \right) = 0\,.
    \end{align}
    In the first case, assume that \(F_Y\) has no jump at \(q_Y(v)\). Choose \(v_n\in (0,1)\setminus N\) with \(v_n\uparrow v\). Then, using left-continuity of \(q_Y\), we obtain \(\bigcup_{n\in \N} \{Y\leq q_Y(v_n)\} = \{Y < q_Y(v)\}\), which implies \eqref{proofpropversccx2}. 
    In the second case, assume that \(F_Y\) has a jump at \(q_Y(v)\) and set \(v_0:= F_Y^-(q_Y(v))\). 
    The case \(v_0=0\) is trivial. Hence, suppose \(v_0>0\). 
    Choose \(v_n\in (0,1)\setminus N\) with \(v_n\uparrow v_0\). 
    Then, a distinction of cases as in the reverse direction yields \(P\left(\bigcup_{n\in\N} \{Y\leq q_Y(v_n)\}\right) = P\left(\{Y < q_Y(v_0)\}\right)\), which again yields  \eqref{proofpropversccx2}.
    Both cases imply \(S_n \xrightarrow{~d~} S\) and \(\E|S_n|\to \E|S|\).   
    Similarly, we obtain \(S_n' := P(Y'\leq q_{Y'}(v_n) \mid \XX') \xrightarrow{~d~} P(Y' < q_{Y'}(v) \mid \XX') =: S'\) and \(\E|S_n'|\to \E|S'|\).
    Using \(S_n \leq_{cx} S_n'\) for all \(n\), Lemma \ref{lemminmaxcx}\,\ref{lemminmaxcx4} yields \(S \leq_{cx} S'\).    
    This proves \ref{propversccx2}.
\end{proof}

\section*{Acknowledgement}

This research was funded in whole by the Austrian Science Fund (FWF) 
{[10.55776/PAT1669224]} project \emph{Stochastic orders for functional dependence} and 
{[10.55776/P36155]} project \emph{ReDim: Quantifying dependence via dimension reduction}.
The second author further acknowledges the support from the WISS 2025 project `IDA-lab Salzburg' 20204-WISS/225/197-2019 and 20102-F1901166-KZP.



\bibliographystyle{abbrvnat}
\bibliography{ccx_arXiv}

\addcontentsline{toc}{section}{Bibliography}

\end{document}